\documentclass[pdflatex,sn-mathphys-num]{sn-jnl}
\usepackage{amsmath,amssymb,amsthm,mathtools,longtable,booktabs,array}
\usepackage{microtype}
\usepackage{xcolor} % Needed by the Springer class for the corresponding-author email footnote.
\hypersetup{hidelinks,pdftitle={Unconditional scattering below the ground state for the three-dimensional focusing energy-critical generalized Hartree equation},pdfauthor={Pang-hung Chung and Dan Han},pdfsubject={Ground-state threshold scattering for the generalized Hartree equation},pdfkeywords={Hartree equation, energy critical, nonradial scattering, Gaussian interaction}}
\theoremstyle{thmstyleone}
\newtheorem{theorem}{Theorem}[section]
\newtheorem{lemma}[theorem]{Lemma}
\newtheorem{proposition}[theorem]{Proposition}
\newtheorem{corollary}[theorem]{Corollary}
\theoremstyle{thmstylethree}

\theoremstyle{thmstyletwo}
\newtheorem{remark}[theorem]{Remark}
\newcommand{\R}{\mathbb R}
\newcommand{\C}{\mathbb C}
\newcommand{\N}{\mathbb N}
\newcommand{\dd}{\,\mathrm d}
\newcommand{\re}{\operatorname{Re}}
\newcommand{\im}{\operatorname{Im}}

\newcommand{\X}{\dot H^1(\R^3)}
\newcommand{\EE}{\mathsf E}
\newcommand{\SSS}{\mathsf S}
\newcommand{\WW}{\mathsf W}

\newcommand{\NN}{\mathsf N}

\DeclareMathOperator{\supp}{supp}

\newcommand{\introheading}[1]{\par\noindent\hspace*{\parindent}\textbf{#1}\enspace\ignorespaces}
\begin{document}
\title[Unconditional threshold scattering for the generalized Hartree equation]{Unconditional scattering below the ground state for the three-dimensional focusing energy-critical generalized Hartree equation}
% Affiliations follow the 39-page two-author JAMS draft.
% Emails follow the authors' earlier coauthored submission materials.
\author[1]{\fnm{Pang-hung} \sur{Chung}}\email{penghongzhong@yahoo.com}
\author*[2]{\fnm{Dan} \sur{Han}}\email{dan.han@louisville.edu}
\affil[1]{\orgdiv{Department of Mathematics}, \orgname{Guangdong University of Education}, \orgaddress{\city{Guangzhou}, \postcode{510640}, \country{People's Republic of China}}}
\affil[2]{\orgdiv{Department of Mathematics}, \orgname{University of Louisville}, \orgaddress{\city{Louisville}, \state{Kentucky}, \postcode{40292}, \country{USA}}}
\abstract{Unconditional global existence and scattering are established for the
three-dimensional focusing energy-critical generalized Hartree equation
with $0<\alpha<3$, below the ground-state energy and kinetic thresholds.
The initial data lie in $\dot H^1$, with no radial symmetry, finite-mass,
or finite-variance assumption. Annular frequency estimates and Gaussian
interaction estimates rule out a compact critical trajectory by first
controlling its nonlinear forcing and then recovering its conserved mass.}
\maketitle
\clearpage
\setcounter{tocdepth}{2}
{\small\tableofcontents}
\clearpage

\section{Introduction}
Hartree equations are canonical effective models for quantum systems in
which a wave field evolves under a potential generated by the field itself.
In a mean-field regime, the microscopic interaction of many particles is
replaced by an averaged field produced by the one-particle density
\cite{LewinHartreeMeanField}. Thus the value of the force at a point is not
determined only by the value of the wave at that point: it depends on an
integral over the whole density profile. This self-consistent, nonlocal
mechanism is what distinguishes Hartree dynamics from a local nonlinear
Schr\"odinger equation and is also what makes localization substantially
more delicate.

The stationary Choquard--Hartree problem has long been connected with
self-consistent quantum models, and its ground states provide the natural
localized states against which dispersion is measured
\cite{LiebChoquard,MorozVanSchaftingen}. The corresponding time-dependent
Hartree evolution has a well-developed energy-space scattering theory
\cite{GinibreVeloNonlocal,GinibreVeloHartree}. In the focusing case the
self-generated field is attractive, so two opposite mechanisms compete:
the Laplacian tends to spread the wave, while the Hartree interaction tends
to concentrate it. A ground state represents the stationary balance of
these effects. The dynamical question is therefore simple to state but
fundamental: if the initial state lies strictly below that balance, must all
of the solution ultimately disperse?

The energy-critical regime is the sharp setting for this question. At the
critical scaling, concentration occurs at exactly the same scale as the
conserved kinetic energy, so neither dispersion nor concentration gains an
energetic advantage merely by changing scale. This is the regime in which
stationary states, concentration compactness, and scattering meet most
directly. In three dimensions the generalized power $p=3+\alpha$ is chosen
precisely so that the homogeneous energy is scale invariant. The family is
therefore best viewed as the energy-critical generalized Hartree--Choquard
model associated with the Riesz interaction. We do not assume that each
value of $\alpha$ comes from a separate microscopic Hamiltonian; the point
is that the whole family preserves the same nonlocal self-interaction and
the same critical competition between dispersion and concentration.

There is also a structural reason to formulate the problem in the full
homogeneous energy space. The critical energy controls $\dot H^1$ but does
not control the mass $\|u\|_2^2$, and translation invariance allows a
concentrated profile to drift through space. Radial symmetry suppresses
that drift, while a finite-mass or finite-variance hypothesis supplies
extra information that is not part of the scale-critical problem. By
``unconditional scattering'' we mean scattering for every datum below the
ground-state thresholds in the natural homogeneous energy class, without
adding radial symmetry, finite mass, finite variance, or an auxiliary
global spacetime assumption.

Fix $0<\alpha<3$, put $p=3+\alpha$, and set
\[
 \begin{aligned}
 I_\alpha(x)&=c_\alpha |x|^{\alpha-3}\qquad(x\ne0),\\
 c_\alpha&=\frac{\Gamma((3-\alpha)/2)}
 {2^\alpha\pi^{3/2}\Gamma(\alpha/2)}.
 \end{aligned}
\]
The three-dimensional focusing energy-critical generalized Hartree equation is
\begin{equation}\label{eq:hartree}
 i\partial_t u+\Delta u+(I_\alpha*|u|^{3+\alpha})|u|^{1+\alpha}u=0,
 \qquad(t,x)\in\R\times\R^3.
\end{equation}
Its critical scaling is
\[
 u(t,x)\longmapsto\lambda^{1/2}u(\lambda^2t,\lambda x),\qquad \lambda>0,
\]
which preserves the homogeneous energy norm. The main theorem concerns
the full energy space, without radial symmetry or a finite-mass
assumption. The ground-state energy and gradient norm specify the two
thresholds.

\introheading{The Hartree ground-state threshold.}
For \eqref{eq:hartree} write
\[
 F(u)=(I_\alpha*|u|^p)|u|^{p-2}u.
\]
The value of $I_\alpha$ at the origin is set to zero; this choice does not
affect its convolution class. The homogeneous energy space $\X$ consists
of complex $L^6$ functions with distributional gradient in $L^2$, and is
equipped with the norm $\|\nabla u\|_2$. Write
\[
 \begin{aligned}
 A(u)&=\int_{\R^3}|\nabla u|^2\dd x, &
 D(u)&=\int_{\R^3}(I_\alpha*|u|^p)|u|^p\dd x,\\
 E(u)&=\frac12A(u)-\frac1{2p}D(u).
 \end{aligned}
\]
The critical scaling also preserves $A$ and $E$.
The free group is $S(t)=e^{it\Delta}$;
with $\widehat f(\xi)=\int e^{-2\pi ix\cdot\xi}f(x)\,dx$, its multiplier is
$e^{-4\pi^2it|\xi|^2}$.

The sharp threshold is fixed by the original positive stationary state.
Set
\[
 \begin{aligned}
 b_\alpha&=\frac{\Gamma((3-\alpha)/2)}
 {2^\alpha\Gamma((3+\alpha)/2)}, &
 a_\alpha&=\left(\frac3{b_\alpha}\right)^{1/(2p-2)},\\
 W(x)&=a_\alpha(1+|x|^2)^{-1/2}.
 \end{aligned}
\]
Then $-\Delta W=F(W)$. The sharp calibration established below gives
\[
 \begin{aligned}
 G&:=A(W)=D(W), & E(W)&=\frac{p-1}{2p}G,\\
 D(u)&\le G^{1-p}A(u)^p.
 \end{aligned}
\]
Translations, constant phases, and critical dilations leave these threshold
values unchanged. In particular, $W\notin L^2$.

The ground-state comparison is therefore formulated in the full homogeneous
energy class. No radiality, scalar $L^2$ membership, finite variance, or
restriction on the spatial center of the solution is imposed. The main scattering theorem is the following.

\begin{theorem}[Unconditional ground-state scattering]\label{thm:main}
For every $0<\alpha<3$ and $u_0\in\X$ satisfying
\[
 E(u_0)<E(W),\qquad A(u_0)<A(W),
\]
the unique maximal homogeneous-energy solution of \eqref{eq:hartree}
with value $u_0$ at time zero is global. Here the solution class is
$u\in C(I;\X)$ with $\nabla u\in L^{10}_{\mathrm{loc}}(I;L^{30/13})$,
satisfying the finite Duhamel identity against Schwartz tests;
maximality and uniqueness are in this class. There exist $u_\pm\in\X$ such that
\[
 \lim_{t\to\pm\infty}\|\nabla(u(t)-S(t)u_\pm)\|_2=0.
\]
\end{theorem}

The radial subthreshold theorem of Almuthaybiri, Peng and Saanouni
\cite{AlmuthaybiriPengSaanouni} is already formulated in $\dot H^1$.
Theorem~\ref{thm:main} removes radial symmetry in dimension three while
retaining the same ground-state thresholds and homogeneous energy class.
No finite-mass assumption is added in place of symmetry. The issue is
to exclude a concentrating solution whose center can move freely and
whose total mass is not known to be finite.

\introheading{Why this result matters.}
The theorem identifies the ground state as the effective threshold on the
full scattering side of the three-dimensional translation-invariant
subthreshold problem, not only inside a radial subclass.
Equivalently, below the two ground-state thresholds there is no hidden
nonradial compact dynamics supported by spatial drift, by an arbitrarily
changing concentration scale, or by infinite mass. This closes the
subthreshold scattering problem for the critical generalized family in
its natural scale-invariant setting and shows that finite mass is not a
missing dynamical hypothesis but a quantity that can be recovered from
the rigidity argument itself.

The mechanism is relevant beyond the statement proved here. Threshold
classification at $E(W)$, dynamics slightly above the threshold,
multi-profile or soliton-resolution scenarios, and other critical
nonlocal dispersive equations all face versions of the same obstruction:
one must control a moving concentrating object using an interaction
functional whose most convenient estimates normally involve quantities
not controlled by the critical energy. The present argument isolates a
way around that obstruction---first obtain mass-independent coercivity
and finite-interval absorption, then recover mass only after the forcing
has become time integrable. The theorem does not by itself settle those
threshold, superthreshold, or soliton-resolution problems, but it removes
a specific structural barrier that recurs in them and supplies estimates
that can be tested in those settings.

\introheading{The new estimates.}
Concentration compactness supplies the reduction to a possible nonzero
critical solution whose orbit is compact after translation and rescaling.
The main work is to rule out this solution despite the nonlocal
interaction and the absence of a mass bound. After selecting a half-orbit
with bounded physical scale, the argument uses the following three
estimates.

\introheading{1. Controlling the forcing before mass is available.}
The energy bound controls the $L^6$ norm of the critical solution, but
on $\R^3$ it gives neither an $L^4$ nor an $L^2$ bound. The interaction
argument starts by proving this additional spatial integrability.
For noninteger powers, low-frequency inputs can still contribute to an
output annulus well beyond their own frequency support. This contribution
cannot be discarded by a Fourier-support argument.

For purely low-frequency input, we control this leakage by differentiating
the full Hartree product and gaining an inverse frequency on each output
annulus. Together with the mixed-frequency bounds, this gives an $L^1$
source with the same annular projection as the forcing, without assuming
that the entire forcing already lies in $L^1$.
The annular no-waste Duhamel formula then represents the critical
solution through its future forcing along the selected time direction.
To close the estimate over all frequency scales, we damp the lowest
frequencies in a weighted supremum, making it finite before absorption.
Compactness supplies the small low-frequency gradient needed to absorb
its contribution. This gives
uniform spatial $L^4$ control of the critical solution and $L^1$ control
of its forcing, starting only from homogeneous energy bounds and
compactness. These are the spatial bounds needed to begin the interaction
argument; integrability over the half-line is obtained later.

\introheading{2. Preserving the ground-state gap under localization.}
The strict threshold gives a global excess of kinetic energy over
attraction. The interaction identity needs a positive estimate after
localization, where the Gaussian weights different directions differently
and the Hartree force still couples the concentrated core to the distant
field. Moreover, removing frequencies near zero produces a field of
finite mass, but this mass need not stay bounded as the cutoff is lowered.
A positive estimate that deteriorates with this mass would not provide
the uniform margin needed to absorb the truncation error.

Directional stretching first converts the ground-state gap into positive
stress in each spatial direction. A core--tail decomposition then
controls the distant field through its small gradient norm. The
localization retains both Hartree pressure terms and the complete density
in the interaction weights. It also retains the full quadratic density
term: its physical kernel changes sign, but its Fourier multiplier is
nonnegative. This positivity allows the tail error to be absorbed while
leaving a positive constant independent of the mass of the truncated
field. Since the interaction depends on relative positions, the estimate
requires no control of the motion of the concentration center.

\introheading{3. Closing the interaction estimate despite truncation.}
The frequency-truncated field satisfies an equation with an additional
source. This defect need not be small in $L^1$: a low-frequency
projection of an integrable function with nonzero mean can retain a
nonzero $L^1$ norm as the cutoff tends to zero. Thus small output
frequency cannot by itself supply a small error in the interaction
identity.

The source is estimated after pairing it with the interaction weights.
The change in the nonlinear input is handled through the full Hartree
pressure identity. The projected output is handled by a frequency gap
between the explicit field factor and the low-frequency parts of the
weights. Compactness makes the product of the cutoff frequency and the
truncated $L^2$ norm small, even though that norm need not stay bounded.
Together these estimates bound the source contribution by a small
multiple of the same quartic space-time term supplied by coercivity.

The localization radius must also adapt to the changing solution. We
choose it from the field's own Gaussian density. Using the same Gaussian
in this constraint and in the interaction action allows the charge
identity to control the unfavorable part of the radius-motion term;
no small bound on the absolute speed of the radius is needed.
For this step, the defect is small in $L^{6/5}$; multiplying it by the
truncated field, uniformly bounded in $L^6$, gives a small charge source
in $L^1$.
The resulting interaction balance closes on finite time intervals at
one fixed cutoff. Sending the terminal time to infinity then gives
time integrability of the original forcing. The no-waste formula and
the dual endpoint estimate now recover finite mass for the original
critical solution and force that mass to tend to zero. Conservation
of this recovered mass yields the contradiction. The order of these
steps is what makes the argument apply without a mass assumption.

\introheading{Related nonlinear Schr\"odinger methods.}
The critical theory of nonlinear Schr\"odinger (NLS) equations rests on
Strichartz estimates \cite{Strichartz1977} and the local and scattering
theories of Ginibre--Velo and Cazenave--Weissler
\cite{GinibreVeloNLS,CazenaveWeissler}. Bourgain treated radial defocusing
energy-critical data \cite{Bourgain1999}; Colliander, Keel, Staffilani,
Takaoka and Tao proved scattering in the three-dimensional nonradial case
\cite{CKSTT2008}. For the focusing problem, the radial below-threshold
theorem of Kenig--Merle \cite{KenigMerle2006} and the threshold analysis of
Duyckaerts--Merle \cite{DuyckaertsMerle2009} established the role of a compact
critical trajectory. Linear profiles and higher-dimensional focusing
scattering are developed in \cite{Keraani2001,KillipVisan2010}.
The Hartree argument retains this compactness reduction but must control
nonlocal interactions without an a priori mass bound.

\introheading{Hartree scattering and threshold dynamics.}
The variational theory begins with Lieb's minimizing-state theorem for the
Coulomb Choquard equation \cite{LiebChoquard} and his sharp
Hardy--Littlewood--Sobolev (HLS) inequality \cite{LiebHLS}.
Moroz and Van Schaftingen studied generalized Choquard ground states and
their decay for the stationary equation with a mass term
\cite{MorozVanSchaftingen}. The critical stationary equation considered
here has no mass term; its threshold is fixed by the explicit state $W$.

Early nonlocal well-posedness and energy-space Hartree scattering were
developed by Ginibre and Velo
\cite{GinibreVeloNonlocal,GinibreVeloHartree}.
Long-range and energy scattering were studied in
\cite{GinibreOzawa1993,Nakanishi1999}, while quadratic Morawetz estimates
were used in \cite{GinibreVeloMorawetz}.
For the quadratic-density critical equation with kernel $|x|^{-4}$ in
dimensions $d\ge5$, Miao, Xu and Zhao proved defocusing scattering for
radial and then general energy data
\cite{MiaoXuZhaoDefocusingRadial,MiaoXuZhaoDefocusing}.
Below-threshold focusing scattering was established in the radial case
by Miao, Xu and Zhao \cite{MiaoXuZhao} and without radial symmetry by
Li, Miao and Zhang \cite{LiMiaoZhang}.

For generalized powers, Arora and Roudenko studied the intercritical
problem below the mass--energy threshold in $H^1$ \cite{AroraRoudenko}.
Zhou studied dynamics at that threshold under ground-state uniqueness
and nondegeneracy assumptions in the general case
\cite{ZhouThreshold2025}; Gao obtained blow-up or grow-up at the
intercritical threshold without finite variance
\cite{GaoDivergence2026}.
For the quadratic-density equation in dimensions $d\ge5$, Miao, Wu
and Xu studied radial dynamics at the energy-critical threshold under
a spectral hypothesis \cite{MiaoWuXuThreshold}.
Li, Liu, Tang and Xu established the required nondegeneracy and studied
radial threshold dynamics for generalized critical powers
\cite{LiLiuTangXuThreshold}; in dimension three, their kernel exponent
$3-\alpha$ gives $p=3+\alpha$.
For the quadratic-density equation with kernel $|x|^{-4}$ in
dimensions $d\ge5$, Li, Liu and Xu treated radial dynamics slightly above
the ground-state energy \cite{LiLiuXuAbove2026}.
These threshold and above-threshold results concern different dynamical
regimes from the strictly subthreshold conclusion of
Theorem~\ref{thm:main}.

The closest strictly subthreshold result for generalized critical powers
is the radial homogeneous-energy theorem
\cite{AlmuthaybiriPengSaanouni}. The issue here is its nonradial extension
for the translation-invariant three-dimensional equation. The argument
below treats the moving center and the absence of an a priori mass bound
within one interaction estimate.

\introheading{Proof strategy.}
Energy conservation and the sharp inequality give, for data in
Theorem~\ref{thm:main}, a constant $a_*<1$ such that
\[
 A(u(t))\le a_*G,\qquad
 A(u(t))-D(u(t))\ge(1-a_*^{p-1})A(u(t)).
\]
These bounds supply the variational gap for concentration compactness.
If scattering fails, the reduction produces a nonzero global critical
solution $w$ with an orbit compact modulo translations and dilations.
After choosing an orientation $\sigma\in\{-1,1\}$ and one fixed
rescaling, the half-orbit $U(s)=w(\sigma s)$, $s\ge0$, has physical
scale $0<r(s)\le1$. Its center is unrestricted, and $r(s)$ need not
have a positive lower bound. All integrability conclusions in the
following outline concern this hypothetical critical solution.

\introheading{The annular estimate and Gaussian coercivity.}
Fix a smooth radial cutoff $\varphi$, equal to one on $B_1$ and zero outside
$B_2$, with $0\le\varphi\le1$. Define the tempered-distribution multipliers
\[
 \widehat{P_{\le\nu}f}(\xi)=\varphi(2\pi\xi/\nu)\widehat f(\xi),
 \qquad P_{>\nu}=1-P_{\le\nu},\qquad h_\nu=P_{>\nu}U.
\]
For fixed $\nu>0$, $h_\nu\in C_sH^1$, and compactness gives, as $\nu\downarrow0$,
\[
 \eta_\nu:=\sup_s\|\nabla P_{\le\nu}U(s)\|_2\longrightarrow0,
 \qquad
 \delta_\nu:=\nu\sup_s\|h_\nu(s)\|_2\longrightarrow0.
\]
The finite mass of each $h_\nu$ is available at a fixed cutoff; an
estimate uniform in this mass is not available. The annular no-waste
argument instead starts from the gradient bound. It estimates the
leakage from low-frequency inputs into the output annulus, and uses
a translated Hardy bound for the mixed term at the unavailable HLS endpoint
$\alpha=2$, before absorbing a finite weighted envelope. This gives
uniform spatial $L^4$ and $L^1$ bounds for $U$ and $F(U)$, respectively.
These bounds do not yet assert integrability in time.

Directional Hartree estimates then give a positive interaction margin
near the normalized compact orbit. Angular and convolution estimates
control the tail while retaining the complete density quadratic form.
The coercivity constant depends on the compact orbit and its variational
gap, but not on $\|h_\nu\|_2$.
Fix one sufficiently small level $\theta>0$ for the orbit. For all
sufficiently small $\nu>0$, define $\lambda_\nu(s)$ by
\[
 \iint e^{-|x-y|^2/(2\lambda_\nu(s)^2)}
       |h_\nu(s,x)|^2|h_\nu(s,y)|^2\,dx\,dy
       =\theta\lambda_\nu(s)^4.
\]
The root is unique and comparable with $a_\nu(s):=\|h_\nu(s)\|_4^4$,
with comparison constants fixed by the orbit and $\theta$.
Using the same Gaussian in the density constraint and the interaction
action makes the radius derivative compatible with the charge identity.
This controls the negative part of the motion contribution without
requiring a small absolute radius speed.

\introheading{The projected source and recovery of mass.}
The defect
\[
 e_\nu=P_{>\nu}F(U)-F(h_\nu)
      =F(U)-F(h_\nu)-P_{\le\nu}F(U)
\]
need not be small in $L^1$. In particular, small output frequency does
not imply that the nonlinear expression contains a small low-frequency
input. The input difference is estimated through the full Hartree pressure
identity; the projected output is treated by a frequency gap while the
density weights remain fixed. Both terms are compared with the same
quartic integral on terminal intervals $[0,T_j]$, $T_j\to\infty$,
chosen from the forcing before the cutoff is fixed.

After the positive interaction and radius-motion terms are combined,
the finite-interval balance has the form
\[
 (\kappa_0-q_\nu)\int_0^{T_j}a_\nu(s)\,\dd s
 \le C_{\mathrm{end},\nu},\qquad
 \kappa_0>0,\quad q_\nu\longrightarrow0.
\]
Here $\kappa_0$ is fixed by the orbit and the Gaussian level;
$C_{\mathrm{end},\nu}$ is finite for each fixed $\nu$ and independent
of $j$. It is not required to be uniform as $\nu\downarrow0$.
Fixing a small $\nu$ first, and only then letting $j\to\infty$, gives
\[
 U\in L^4([0,\infty);L^4),\qquad
 F(U)\in L^2([0,\infty);L^1).
\]
The forcing estimate is obtained without assuming that $U$ has finite
mass. Put $B=\sup_{s\ge0}\|\nabla U(s)\|_2$ and introduce a new
frequency $L>0$, independent of the fixed $\nu$. The dual endpoint
estimate and the no-waste identity identify the low-frequency part in
$L^2$; the gradient controls the high-frequency part. Together they give
\[
 \|U(t)\|_2\le B/L+C L^{1/2}
        \|F(U)\|_{L^2((t,\infty);L^1)}.
\]
The two parts sum to the original $L^6$ representative of $U$, so this
is recovery of its mass, not construction of a different solution.
Letting time tend to infinity at fixed $L$, and then $L\to\infty$,
gives $\|U(t)\|_2\to0$. The localized charge identity now gives
conservation of that mass, contradicting $U\ne0$. This excludes the
critical orbit and completes the scattering argument.

\introheading{Conventions and organization.}
Spatial integrals are over $\R^3$ unless specified. Distribution pairings
are complex bilinear; Hermitian conjugation is explicit. Constants may
depend on $0<\alpha<3$, with no endpoint uniformity.

Section~\ref{std:section} fixes the homogeneous space and the classical
analytic estimates. Section~\ref{local:section} proves local theory,
conservation, stability, and scattering criteria.
Section~\ref{prof:section} develops profile decomposition and nonlinear
reconstruction; Section~\ref{crit:section} constructs the critical orbit
and selects a bounded-scale half-line. Section~\ref{sec:annular} proves
the annular forcing and spatial integrability estimates.
Section~\ref{sec:gaussian} establishes Gaussian coercivity independent
of the projected mass. Section~\ref{sec:highpass-source} estimates the
projected source, evolves the selected radius, and excludes the critical
orbit by finite-interval absorption and mass recovery.

\section{The homogeneous energy space and analytic tools}
\label{std:section}

The variational estimates fix the ground-state threshold, while the linear
estimates provide the solution and forcing spaces. We formulate both in
$\dot H^1$ with its $L^6$ representative, so no finite-mass assumption
enters the regularization or Duhamel arguments.

We use Lebesgue measure in the standard orthonormal coordinates of
$\mathbb R^3$.  A complex homogeneous energy state is an almost-everywhere
class $f\in L^6(\mathbb R^3)$ with three distributional derivatives
$g_j\in L^2(\mathbb R^3)$: for every $\phi\in C_c^\infty(\mathbb R^3;\mathbb C)$,
\[
 \int f\,\partial_j\phi=-\int g_j\phi.
\]
We denote this space by $\X$.  It has the Hilbert norm
$\|f\|_{\dot H^1}=(\sum_j\|g_j\|_2^2)^{1/2}$ and also the useful graph norm
\[
 \|f\|_{\mathrm{gr}}=\max\{\|f\|_6,\max_j\|g_j\|_2\}.
\]
The derivatives are unique; no $L^2$ assumption is imposed on $f$.
Distribution pairings are complex bilinear, and Hermitian pairings
display the conjugate explicitly. Write
$\mathcal S(\R^3)$ for the Schwartz space and
$\langle f,g\rangle_{\mathrm{bil}}=\int_{\R^3}f(x)g(x)\,\dd x$
whenever the product is integrable, extended to distributional pairings
by duality.

\subsection{Density and the sharp ground-state threshold}

We first approximate a state and its weak gradient by a common smooth sequence.

\begin{lemma}[Simultaneous approximation]\label{std:density}
If $f\in L^2$ has distributional gradient $g\in L^2$, there exist
$f_n\in C_c^\infty$ with $f_n\to f$ and $\nabla f_n\to g$ in $L^2$.
If instead $f\in L^6$ and $g\in L^2$, there exist such $f_n$ with
$f_n\to f$ in $L^6$ and $\nabla f_n\to g$ in $L^2$.
The same sequence approximates the value and weak derivatives of complex functions.
\end{lemma}
\begin{proof}
Choose $0\le\chi\in C_c^\infty(B_2)$ with $\chi=1$ on $B_1$ and
put $\chi_R(x)=\chi(x/R)$. For $q\in\{2,6\}$, the weak product rule gives
\[
 v_R=\chi_Rf,\qquad g_R=\nabla v_R=\chi_Rg+f\nabla\chi_R.
\]
Dominated convergence handles $\chi_Rf$ and $\chi_Rg$, while
\begin{equation}\label{eq:density-cutoff-error}
 \|f\nabla\chi_R\|_2\le
 \begin{cases}
 R^{-1}\|\nabla\chi\|_\infty\|f\|_2,&q=2,\\
 \|f1_{\{|x|>R\}}\|_6\|\nabla\chi\|_3,&q=6
 \end{cases}
 \longrightarrow0.
\end{equation}
By \eqref{eq:density-cutoff-error}, choose $R_n\to\infty$ with
$\|v_{R_n}-f\|_q+\|g_{R_n}-g\|_2\le(2n)^{-1}$.
For a smooth compactly supported approximate identity,
Fubini in the weak identity and translation continuity give
\begin{equation}\label{eq:density-mollification}
 \nabla(\eta_\varepsilon*v_R)=\eta_\varepsilon*g_R,
 \qquad
 \|\eta_\varepsilon*v_R-v_R\|_q+
 \|\eta_\varepsilon*g_R-g_R\|_2\longrightarrow0.
\end{equation}
Choose $\varepsilon_n\downarrow0$ so that the error in \eqref{eq:density-mollification} is
at most $(2n)^{-1}$. Then $f_n=\eta_{\varepsilon_n}*v_{R_n}\in C_c^\infty$
and $\|f_n-f\|_q+\|\nabla f_n-g\|_2\le n^{-1}$.
The choices are common to the value and every real and imaginary
gradient coordinate.
\end{proof}

Convolutions will be taken on these representatives.

For $0<\beta<3$ put
\[
 c_\beta=\frac{\Gamma((3-\beta)/2)}
 {2^\beta\pi^{3/2}\Gamma(\beta/2)},\qquad
 I_\beta(x)=c_\beta|x|^{\beta-3}\quad(x\ne0),\qquad I_\beta(0)=0.
\]
Convolutions are defined by absolutely convergent integrals for almost
every $x$. The value on the exceptional null set is immaterial; no
principal-value interpretation is used.

The strict-range HLS estimate includes absolute convergence of the convolution.

\begin{theorem}[Strong HLS inequality]\label{std:hls}
For $0<\beta<3$, $1<q<3/\beta$ and $1/r=1/q-\beta/3$, there is
$C_{\beta,q}<\infty$ such that every complex $f\in L^q$ satisfies
\[
 I_\beta*|f|<\infty\quad\hbox{a.e.},\qquad
 \|I_\beta*f\|_r\le C_{\beta,q}\|f\|_q.
\]

\end{theorem}
\begin{proof}
Let $M$ be the centered Hardy--Littlewood maximal operator on $\R^3$.
Disjoint-ball selection and triple dilations, first on compact subsets of
the open level set, give
\begin{equation}\label{eq:maximal-weak-three}
 |\{Mh>\lambda\}|\le27\lambda^{-1}\|h\|_1
 \qquad(0\le h\in L^1).
\end{equation}
Fixed-radius averages are continuous. For $f\in L^q$,
$h_\lambda=|f|1_{\{|f|>\lambda/2\}}\in L^1$ and
$Mf\le\lambda/2+Mh_\lambda$. Applying \eqref{eq:maximal-weak-three}, Tonelli and layer cake yield
\begin{equation}\label{eq:maximal-strong-three}
 \|Mf\|_q^q
 \le54q\int_{\R^3}|f(x)|\int_0^{2|f(x)|}\lambda^{q-2}\,d\lambda\,dx
 =\frac{27\,2^q q}{q-1}\|f\|_q^q.
\end{equation}
Put $N=\|f\|_q$, $\vartheta=\beta q/3\in(0,1)$ and $q'=q/(q-1)$.
Dyadic shells and H\"older give bounds on the extended nonnegative integrals:
\begin{equation}\label{eq:hls-near-far}
 \begin{aligned}
 \int_{|x-y|<R}|x-y|^{\beta-3}|f(y)|\,dy
 &\le A_\beta R^\beta Mf(x),\\
 \int_{|x-y|\ge R}|x-y|^{\beta-3}|f(y)|\,dy
 &\le B_{\beta,q}R^{\beta-3/q}N,
 \end{aligned}
\end{equation}
\[
 A_\beta=\frac{|B_1|2^{3-\beta}}{1-2^{-\beta}},\qquad
 B_{\beta,q}=\left(\frac{3|B_1|}{(3-\beta)q'-3}\right)^{1/q'}<\infty.
\]
If $N>0$ and $0<Mf(x)<\infty$, set $R=(N/Mf(x))^{q/3}$ in \eqref{eq:hls-near-far}:
\begin{equation}\label{eq:hedberg-pointwise}
 I_\beta*|f|(x)\le c_\beta(A_\beta+B_{\beta,q})
 N^\vartheta Mf(x)^{1-\vartheta}.
\end{equation}
The case $N=0$ is immediate; $Mf(x)=0$ implies $f=0$ almost everywhere
by exhaustion with centered balls. Combining \eqref{eq:hedberg-pointwise} and
\eqref{eq:maximal-strong-three}, since $r(1-\vartheta)=q$,
\[
 \|I_\beta*f\|_r\le c_\beta(A_\beta+B_{\beta,q})
 \left(\frac{27\,2^q q}{q-1}\right)^{(1-\vartheta)/q}\|f\|_q.
\]
Tonelli makes the absolute majorant measurable; the bound makes it finite
almost everywhere. The resulting complex integral, set to zero on the
exceptional null set, defines the claimed $L^r$ class independently of
representatives and of the kernel value at zero.
\end{proof}

We next fix the sharp constants and the ground-state normalization.

\begin{theorem}[Sharp inequalities and the ground state]\label{std:sharp}
Set $p=3+\alpha$ and
\[
 H_\alpha=\frac{\Gamma((3-\alpha)/2)}
 {2^\alpha\pi^{\alpha/2}\Gamma((3+\alpha)/2)}
 \left(\frac4{\sqrt\pi}\right)^{\alpha/3},
 \qquad S_* =\frac13\left(\frac4{\pi^2}\right)^{2/3}.
\]
For all complex $a,b\in L^{6/p}$, the integral
$\iint I_\alpha(x-y)|a(x)b(y)|\,dx\,dy$ is finite and bounded by
$H_\alpha\|a\|_{6/p}\|b\|_{6/p}$.  For every $f\in\X$,
\[
 \|f\|_6^2\le S_* A(f),\qquad
 0\le D(f)\le H_\alpha S_*^p A(f)^p=G^{1-p}A(f)^p.
\]
Here $G=A(W)=D(W)$, with
\[
 b_\alpha=\frac{\Gamma((3-\alpha)/2)}{2^\alpha\Gamma((3+\alpha)/2)},
 \quad a_\alpha=(3/b_\alpha)^{1/(2p-2)},\quad
 W(x)=a_\alpha(1+|x|^2)^{-1/2},
\]
\[
 G=(H_\alpha S_*^p)^{-1/(p-1)},\qquad
 E(W)=\frac{p-1}{2p}G.
\]
Every least-energy nonzero weak stationary state has kinetic energy $G$
and energy $E(W)$.
\end{theorem}
\begin{proof}
The sharp HLS inequality in \cite{FrankLieb} uses the kernel
$|x-y|^{-\lambda}$. Taking $\lambda=3-\alpha$ and the equal input
exponents $6/(6-\lambda)=6/p$, its coefficient is
\[
 C_\lambda=\pi^{\lambda/2}
 \frac{\Gamma((3-\lambda)/2)}{\Gamma(3-\lambda/2)}
 \left(\frac{\Gamma(3)}{\Gamma(3/2)}\right)^{1-\lambda/3},
 \qquad c_\alpha C_{3-\alpha}=H_\alpha.
\]
Monotone compact truncation of $|a|,|b|$ proves absolute integrability
before Fubini. The spherical Sobolev inequality in \cite{FrankLieb}, transferred
stereographically to $\R^3$ and extended by Lemma~\ref{std:density},
gives the real homogeneous Sobolev constant
$[(3/4)(2\pi^2)^{2/3}]^{-1}=S_*$. For $f=a+ib$,
\[
 \begin{aligned}
 \|f\|_6^2=\|a^2+b^2\|_3
 &\le\|a\|_6^2+\|b\|_6^2\le S_*A(f),\\
 D(f)&\le H_\alpha\|f\|_6^{2p}\le H_\alpha S_*^pA(f)^p.
 \end{aligned}
\]
Put $Q(x)=(1+|x|^2)^{-1/2}$ and $M=\pi^2/4$. Polar integration gives
\[
 \int Q^6=M,\qquad \int|\nabla Q|^2=3M,
 \qquad -\Delta Q=3Q^5.
\]
For $\beta=(3-\alpha)/2$ and $R=|x|^2$, positive Gamma representations
and Gaussian integration give
\[
 \begin{aligned}
 (I_\alpha*Q^p)(x)&=
 \frac{J_\beta(R)}{2^\alpha\Gamma(\alpha/2)\Gamma(p/2)},\\
 J_\beta(R)&=\int_0^\infty\!\int_0^\infty
 \frac{r^{\beta-1}s^{2-\beta}e^{-s-rsR/(r+s)}}{(r+s)^{3/2}}\,dr\,ds.
 \end{aligned}
\]
Under $r=st$, the integrand with its Jacobian is bounded by
$t^{\beta-1}(1+t)^{-3/2}s^{1/2}e^{-s}$, whose integral is
$\Gamma(\beta)\Gamma(3/2-\beta)<\infty$.
Tonelli, Fubini and the beta integral therefore yield, for each $x$,
\[
 J_\beta(R)=\Gamma(\beta)\Gamma(3/2-\beta)(1+R)^{-\beta},
 \qquad I_\alpha*Q^p=b_\alpha Q^{3-\alpha}.
\]
Since $b_\alpha a_\alpha^{2p-2}=3$,
\begin{equation}\label{eq:ground-state-calibration}
 \begin{gathered}
 -\Delta W=F(W),\qquad A(W)=D(W)=3Ma_\alpha^2,\\
 H_\alpha=b_\alpha M^{-\alpha/3},\qquad S_*={\tfrac13}M^{-2/3},
 \qquad (H_\alpha S_*^p)A(W)^{p-1}=1.
 \end{gathered}
\end{equation}
Equation~\eqref{eq:ground-state-calibration} determines $G$ and $E(W)$. For a nonzero weak stationary
$v\in\X$, HLS gives $F(v)\in L^{6/5}$. Testing with the simultaneous
approximants of Lemma~\ref{std:density} gives
\begin{equation}\label{eq:stationary-threshold}
 0<A(v)=D(v)\le G^{1-p}A(v)^p,
 \qquad A(v)\ge G,
 \qquad E(v)\ge\frac{p-1}{2p}G.
\end{equation}
Equality in \eqref{eq:stationary-threshold} is attained by $W$, fixing both least-energy thresholds.
\end{proof}

Hereafter $C_S=\sqrt{S_*}$, so $\|f\|_6\le C_S\|\nabla f\|_2$.

The graph space $\X$ is complete: it is the intersection of the closed kernels
$\partial_j f-g_j$ in $L^6\times(L^2)^3$, where convergence in either $L^q$
implies convergence on every Schwartz test.  Moreover
\[
 \|f\|_{\mathrm{gr}}\le\max(\sqrt{S_*},1)\|f\|_{\dot H^1},
 \qquad \|f\|_{\dot H^1}\le\sqrt3\|f\|_{\mathrm{gr}}.
\]
Lemma~\ref{std:density} identifies this same space, including its weak topology,
with the usual homogeneous completion of $C_c^\infty$.

\subsection{The free group, endpoint estimates, and weak uniqueness}

Our Fourier transform is $\widehat f(\xi)=\int e^{-2\pi ix\cdot\xi}f(x)\,dx$.
Define $S(t)$ on Schwartz functions and tempered distributions by the multiplier
$m_t(\xi)=\exp(-4\pi^2it|\xi|^2)$, and on $L^2$ by Plancherel.  It is a strongly
continuous unitary group on $L^2$.  On $\X$ it is the continuous extension from
Lemma~\ref{std:density}; its gradient coordinates are exactly $S(t)\partial_j f$.
The Sobolev inequality makes this extension a continuous group in the graph
norm and an isometry in the homogeneous norm.  Its bilinear test transpose is
$S(t)$, with the same time sign: $\mathcal Fm_t\mathcal F^{-1}
=\mathcal F^{-1}m_t\mathcal F$, since $m_t$ is even and $\mathcal F^2$
is reflection.

The following estimates use this free group and its dual realization.

\begin{theorem}[Endpoint and retarded estimates]
\label{std:endpoint}
Let $Y=L^2(\mathbb R;L^{6/5}(\mathbb R^3))$.  There are finite constants such that
\[
 \|S(\cdot)f\|_{L^2_tL^6_x}+\|S(\cdot)f\|_{L^{10}_tL^{30/13}_x}
 \le C\|f\|_2.
\]
For every measurable $J\subset\mathbb R$, $t\in\mathbb R$, $H\in Y$ and
Schwartz $\phi$, the scalar integral below is absolutely integrable and
\begin{equation}\label{std:dual-endpoint}
 \int_J|\langle H(s),S(t-s)\phi\rangle_{\rm bil}|\,ds
 \le C\|1_JH\|_Y\|\phi\|_2.
\end{equation}
It determines a unique $L^2$ element $\mathcal R_{J}H(t)$ by its pairings with
all Schwartz tests.  For $\mathcal R H(t)=\mathcal R_{(-\infty,t)}H(t)$,
\[
 \mathcal R H\in C_tL^2_x,\qquad
 \|\mathcal RH\|_{L^\infty_tL^2_x}
 +\|\mathcal RH\|_{L^2_tL^6_x}
 +\|\mathcal RH\|_{L^{10}_tL^{30/13}_x}\le C\|H\|_Y.
\]
The bounds also hold for future integrals, uniformly over intervals. All
source integrals are understood in the dual $L^2$ realization above.
\end{theorem}
\begin{proof}
Analytic continuation of the real Gaussian in the right half-plane,
justified by uniform Gaussian domination on compact subsets, gives
\begin{equation}\label{eq:regularized-free-kernel}
 \mathcal F^{-1}\!\left(e^{-4\pi^2(\varepsilon+it)|\xi|^2}\right)(x)
 =[4\pi(\varepsilon+it)]^{-3/2}e^{-|x|^2/[4(\varepsilon+it)]}.
\end{equation}
For $t\ne0$, the modulus of the kernel in \eqref{eq:regularized-free-kernel} is at most $(4\pi|t|)^{-3/2}$. Let
$\varepsilon\downarrow0$ on Schwartz functions and use simultaneous
$L^1,L^2$ approximation:
\begin{equation}\label{eq:free-dispersive-bound}
 \begin{aligned}
 S(t)f&=(4\pi it)^{-3/2}e^{i|\cdot|^2/(4t)}*f,\\
 \|S(t)f\|_\infty&\le(4\pi|t|)^{-3/2}\|f\|_1
 \qquad(f\in L^1\cap L^2).
 \end{aligned}
\end{equation}
Unitarity and \eqref{eq:free-dispersive-bound} are the hypotheses of
the abstract Strichartz theorem of \cite{KeelTao}, with $\sigma=3/2$.
Its sharp admissibility relation is $2/q+3/r=3/2$, $q,r\ge2$.
We use the output pairs $(\infty,2)$, $(2,6)$ and $(10,30/13)$,
with dual source exponents $(2,6/5)$. In particular, the retarded
$L^2_tL^{6/5}_x\to L^2_tL^6_x$ estimate is the double endpoint
included in that theorem.

For every Schwartz $\phi$, H\"older and the homogeneous endpoint bound
give
\[
 \int_J|\langle H(s),S(t-s)\phi\rangle_{\rm bil}|\,ds
 \le\|1_JH\|_Y\|S(t-\cdot)\phi\|_{L^2L^6}
 \le C\|1_JH\|_Y\|\phi\|_2.
\]
This is the estimate \eqref{std:dual-endpoint}.
Strong measurability and complex Riesz representation give a unique
$L^2$ element with these bilinear pairings.

Choose finite smooth compactly supported time-space tensors $H_n\to H$
in $Y$. Their Bochner integrals $V_n(t)=\int_{-\infty}^tS(t-s)H_n(s)\,ds$
satisfy
\begin{equation}\label{eq:retarded-tensor-cauchy}
 \|V_n-V_m\|_{C_tL^2}
 +\|V_n-V_m\|_{L^2L^6}
 +\|V_n-V_m\|_{L^{10}L^{30/13}}
 \le C\|H_n-H_m\|_Y.
\end{equation}
On bounded $I\times B$, H\"older turns each convergence in \eqref{eq:retarded-tensor-cauchy} into local
$L^1$ convergence. Thus the three limits have one continuous $L^2$
representative with the prescribed scalar pairings. The same density argument treats the free estimates.
Finally, the group identity and the absolutely convergent scalar splitting give
\begin{equation}\label{eq:future-past-decomposition}
 \mathcal R_{(t,\infty)}H(t)
 =S(t)\mathcal R_{\R}H(0)-\mathcal RH(t).
\end{equation}
The free and past estimates applied to \eqref{eq:future-past-decomposition} give future continuity and all future
bounds, with constants independent of interval length.
\end{proof}

To identify the Duhamel solution, we also need uniqueness for the weak linear equation.

\begin{lemma}[Generator and weak uniqueness]
\label{std:generator}\label{std:linear-uniqueness}
For every $T\in\mathcal S'$ and $\phi\in\mathcal S$,
$\frac d{dt}(S(t)T)(\phi)=i(S(t)T)(\Delta\phi)$.
If $z\in C([a,b];L^2)$, $t_0\in[a,b]$, $z(t_0)=0$, and
\[
 \int_a^b\{-i\eta'(t)\langle z(t),\phi\rangle_{\rm bil}
       +\eta(t)\langle z(t),\Delta\phi\rangle_{\rm bil}\}\,dt=0
\]
for all $\eta\in C_c^\infty((a,b))$ and $\phi\in C_c^\infty(\mathbb R^3)$,
then $z=0$ on $[a,b]$.
\end{lemma}
\begin{proof}
For $\omega=4\pi^2|\xi|^2$, Taylor's formula gives
\begin{equation}\label{eq:free-generator-remainder}
 \frac{m_{t+h}-m_t}{h}+i\omega m_t
 =h\int_0^1(1-\vartheta)\partial_s^2m_{t+\vartheta h}\,d\vartheta.
\end{equation}
On $|s-t|\le1$, $\partial_\xi^\gamma\partial_s^2m_s$ grows at most
as $(1+|\xi|)^{2|\gamma|+4}$. Multiplication by a Schwartz function
therefore makes the remainder in \eqref{eq:free-generator-remainder} tend to zero in every Schwartz seminorm.
Fourier continuity and application of $T$ give the generator identity.

For uniqueness, cutoff tests extend the weak equation to $\phi\in\mathcal S$:
\[
 \|\Delta(\chi_R\phi)-\chi_R\Delta\phi\|_2
 \le CR^{-1}\|\nabla\phi\|_2+CR^{-2}\|\phi\|_2\longrightarrow0.
\]
Subtracting the scalar primitive and mollifying gives its integral form;
continuity includes the endpoint values. Take a real even smooth cutoff
$P_N$ equal to one on $|\xi|\le N$ and supported in $|\xi|\le2N$.
Schwartz separation gives
\begin{equation}\label{eq:localized-linear-primitive}
 P_Nz(t)-P_Nz(s)=\int_s^t i\Delta P_Nz(\tau)\,d\tau.
\end{equation}
Since $\Delta P_N$ is bounded on $L^2$, \eqref{eq:localized-linear-primitive} gives $P_Nz\in C^1_tL^2$, and
\begin{equation}\label{eq:localized-linear-charge}
 \frac d{dt}\|P_Nz(t)\|_2^2
 =2\re\int i\Delta P_Nz(t)\,\overline{P_Nz(t)}=0,
 \qquad \|P_Nz(t)\|_2^2=\|P_Nz(t_0)\|_2^2=0.
\end{equation}
Let $N\to\infty$ in \eqref{eq:localized-linear-charge}, using $L^2$ convergence. If $a=b$, the zero trace already suffices.
\end{proof}

We next record the spatial regularization and compactness tools.

\subsection{Bessel regularization, local compactness, and gluing}

For $a>0$, write $B_a$ for the Fourier multiplier
$(1+a^2|\xi|^2)^{-1}$.

\begin{lemma}[Bessel regularization]\label{std:bessel}
For $a>0$ the multiplier $(1+a^2|\xi|^2)^{-1}$ is convolution with
\[
 K_\mu(x)=\frac{e^{-|x|/\mu}}{4\pi\mu^2|x|},\qquad \mu=a/(2\pi).
\]
It is a contraction on every complex $L^q$, $1<q<\infty$, converging
strongly to the identity as $a\downarrow0$ and to zero as $a\to\infty$.
For every $H\in L^{6/5}$ the same convolution is in $L^6$ and its actual
coordinate derivatives $(\partial_jK_\mu)*H$ belong to $L^2$.
\end{lemma}
\begin{proof}
Put $c_\mu=(4\pi\mu^2)^{-1}$. Polar integration gives
\[
 K_\mu\ge0,\qquad \int K_\mu=1,\qquad
 K_\mu(x)=\mu^{-3}K_1(x/\mu),\qquad
 \int_{\R^2}K_\mu(t,y_1,y_2)\,dy_1dy_2=\frac{e^{-|t|/\mu}}{2\mu}.
\]
Fubini in a frame parallel to $\xi$ yields
\[
 \widehat K_\mu(\xi)=\frac1{2\mu}
 \left(\frac1{\mu^{-1}+2\pi i|\xi|}
       +\frac1{\mu^{-1}-2\pi i|\xi|}\right)
 =\frac1{1+4\pi^2\mu^2|\xi|^2}.
\]
Thus $\mu=a/(2\pi)$ gives the multiplier. The Bochner average
$T_\mu f=\int K_\mu(y)f(\cdot-y)\,dy$ is complex linear, and
\begin{equation}\label{eq:bessel-contraction-identity}
 \|T_\mu f\|_q\le\|f\|_q,\qquad
 \|T_\mu f-f\|_q
 \le\int K_1(y)\|f(\cdot-\mu y)-f\|_q\,dy\longrightarrow0
 \quad(\mu\downarrow0).
\end{equation}
Schwartz Fourier inversion and $L^q$ density identify it with the
multiplier. For Schwartz $\phi$, $K_1\in L^2$ gives
\begin{equation}\label{eq:bessel-large-scale}
 \begin{aligned}
 \|T_\mu\phi\|_\infty&\le\mu^{-3/2}\|K_1\|_2\|\phi\|_2,
 &\|T_\mu\phi\|_1&\le\|\phi\|_1,\\
 \|T_\mu\phi\|_q^q
 &\le\mu^{-3(q-1)/2}\|K_1\|_2^{q-1}\|\phi\|_2^{q-1}\|\phi\|_1
 &&\longrightarrow0\quad(\mu\to\infty).
 \end{aligned}
\end{equation}
Density and the contraction in \eqref{eq:bessel-contraction-identity} extend
the limit \eqref{eq:bessel-large-scale} to every $f\in L^q$, $1<q<\infty$.

Off zero, set
\[
 G_\mu(x)=-c_\mu\frac{e^{-|x|/\mu}(1+|x|/\mu)}{|x|^3}x,
 \qquad \|G_\mu\|_1=2/\mu.
\]
Smooth the radial profile by $|x|\mapsto(|x|^2+\varepsilon)^{1/2}$.
The kernels and gradients are dominated by $K_\mu$ and $|G_\mu|$;
compact-test limits give $\nabla K_\mu=G_\mu$, with no atom at zero.
The pointwise bounds
\begin{equation}\label{eq:bessel-riesz-majorants}
 K_\mu\le\mu^{-2}I_2,\qquad
 |(G_\mu)_j|\le\frac{\pi}{2\mu^2}I_1
\end{equation}
in \eqref{eq:bessel-riesz-majorants} and Theorem~\ref{std:hls} give absolutely convergent convolutions
$K_\mu*:L^{6/5}\to L^6$ and $(G_\mu)_j*:L^{6/5}\to L^2$.
The value is also the $L^{6/5}$ contraction. Differentiate Schwartz
sources and pass through one common $L^{6/5}$ approximation:
\[
 \partial_j(K_\mu*H)=(G_\mu)_j*H\quad\text{in }\mathcal S',
 \qquad j=1,2,3.
\]
These are the value and weak derivatives of one state in $\X$.
\end{proof}

Local compactness is needed only on bounded balls.

\begin{lemma}[Rellich on a ball]\label{std:rellich}
Let $R>0$.  If real $v_n$ have global weak gradients $G_n$, both restricted
to $B_R$ lie in $L^2$, and
$\sup_n(\|v_n\|_{L^2(B_R)}+\|G_n\|_{L^2(B_R)})<\infty$, then a strictly
increasing subsequence is Cauchy in the scalar space $L^2(B_R)$.
The conclusion also holds for complex sequences.
\end{lemma}
\begin{proof}
For $0<\varepsilon\le1/2$, take a nonnegative smooth probability
kernel $k_\varepsilon$ supported in $B_{\varepsilon R/2}$. Put
\[
 F=1_{B_R}f,\qquad
 A_\varepsilon f(x)=\int k_\varepsilon(y)F((1-\varepsilon)x-y)\,dy.
\]
For $x\in\overline B_R$ this uses only interior values; no weak gradient
of the zero extension is used. For smooth $f$,
\[
 f(x)-f((1-\varepsilon)x-y)
 =\int_0^1 Df((1-t\varepsilon)x-ty)[\varepsilon x+y]\,dt.
\]
Since $|\varepsilon x+y|\le3\varepsilon R/2$ and
$(1-t\varepsilon)^{-3}\le8$, Cauchy--Schwarz and change of variables give
\begin{equation}\label{eq:rellich-inward-displacement}
 \int_{B_R}|f(x)-f((1-\varepsilon)x-y)|^2\,dx
 \le18\varepsilon^2R^2\int_{B_R}|G|^2.
\end{equation}
For weak pairs fix $|y|<\varepsilon R/2$ and $2|y|/\varepsilon<r<R$.
Mollification with radii below $R-r$ gives on $B_r$
\[
 f_j\to f\text{ in }L^2,
 \qquad \nabla f_j=k_j*G,
 \qquad \|\nabla f_j\|_{L^2(B_r)}\le\|G\|_{L^2(B_R)}.
\]
Fatou, including the nonsingular affine pullback, passes \eqref{eq:rellich-inward-displacement}
to $f$; then let $r\uparrow R$. Jensen and Fubini yield
\begin{equation}\label{eq:rellich-inward-error}
 \|f-A_\varepsilon f\|_{L^2(B_R)}^2
 \le18\varepsilon^2R^2\|G\|_{L^2(B_R)}^2.
\end{equation}
At fixed $\varepsilon$, convolution bounds give a uniformly bounded
$C^1(\overline B_R)$ family $\{A_\varepsilon v_n\}$; Arzel\`a--Ascoli
makes it totally bounded in $L^2(B_R)$.
The uniform error \eqref{eq:rellich-inward-error} tends to zero, so
$\{v_n\}$ is totally bounded as well. Completeness gives the increasing
Cauchy subsequence. Successive extraction for the real and imaginary
parts gives the complex case, without boundary or exterior $L^2$ assumptions.
\end{proof}

Piecewise smooth constructions are admissible when their traces agree.

\begin{lemma}[Regular-level gluing]\label{std:gluing}
Let $\rho,u,v$ be real $C^1$ functions on $\mathbb R^3$, with $d\rho\ne0$
on $\Sigma=\{\rho=0\}$ and $u=v$ on $\Sigma$.  If
$h=u1_{\{\rho<0\}}+v1_{\{\rho\ge0\}}$ has compact topological support,
then $h\in H^1$ and its weak differential is exactly
$D=du1_{\{\rho<0\}}+dv1_{\{\rho\ge0\}}$.
The same assertion holds on $P=\mathbb R\times(\mathbb R\times\mathbb R)$
with its maximum product norm, product measure, and operator-norm covectors.
\end{lemma}
\begin{proof}
On $P$, the fibers of
$A_j=\{\rho=0,\ d\rho[e_j]\ne0\}$ parallel to $e_j$ are isolated,
hence countable. Thus Fubini and $\Sigma=\bigcup_{j=1}^3A_j$ give
$|\Sigma|=0$.

Choose $\chi\in C_c^\infty$ equal to one near $\operatorname{supp}h$.
The functions $U=\chi u$, $V=\chi v$ are globally Lipschitz and agree
on $\Sigma$. For opposite-side $x,y$, choose $z\in[x,y]\cap\Sigma$:
\begin{equation}\label{eq:gluing-lipschitz}
 |h(x)-h(y)|\le C_U\|x-z\|_P+C_V\|z-y\|_P
 \le(C_U+C_V)\|x-y\|_P.
\end{equation}
Same-side points use one piece. Since $h=\chi h$, the selected
differentials of $U,V$ equal $D$ off $\Sigma$ and have compact support.
The Lipschitz bound \eqref{eq:gluing-lipschitz} permits integration by parts:
\begin{equation}\label{eq:gluing-product-weak}
 h,D\in L^2(P),\qquad
 \int_P h\,d\phi[e]=-\int_P D[e]\,\phi
 \quad(e\in P,\ \phi\in C_c^\infty(P)).
\end{equation}
For the Euclidean assertion, use the measure-preserving maps
$T(x)=(x_1,(x_2,x_3))$ and $L=T^{-1}$, with
\begin{equation}\label{eq:coordinate-norm-comparison}
 \|Tx\|_P\le|x|\le\sqrt3\|Tx\|_P,\qquad
 \Bigl(\sum_jc_j^2\Bigr)^{1/2}\le\sum_j|c_j|
 \le\sqrt3\Bigl(\sum_jc_j^2\Bigr)^{1/2}.
\end{equation}
Pull back $\rho,u,v$ along $L$ and let $D_P$ be the resulting weak
differential. Then
\[
 D_E(x)[e]=D_P(Tx)[Te],\qquad
 d(\phi\circ L)(z)[Te]=d\phi(Lz)[e].
\]
The bounds \eqref{eq:coordinate-norm-comparison} give $D_E\in L^2$; change of variables
transports \eqref{eq:gluing-product-weak} without a Jacobian factor. The chain rule identifies
$D_E$ with $du,dv$ on their respective sides, and Euclidean Riesz
identification gives the vector gradient.
\end{proof}
The same coordinate map transports the simultaneous approximation of
Lemma~\ref{std:density} to $P$. The scalar $L^2$ norm is unchanged,
and \eqref{eq:coordinate-norm-comparison} bounds the covector error
by $\sqrt3$ times the Euclidean gradient error. Thus one compact
smooth sequence approximates the value and the operator-norm weak
derivative, including both components in the complex case.

\subsection{Refined Sobolev, Hardy, and maximal estimates}

Fix once and for all a real radial $\varphi\in C_c^\infty(\mathbb R^3)$ with
$0\le\varphi\le1$, equal to one on $B_1$ and zero outside $B_2$.
No monotonicity or idempotence is assumed.  In angular frequency
$\theta=2\pi\xi$ we use
$\mathcal F_{\rm ang}f(\theta)=\int e^{-ix\cdot\theta}f(x)\,dx$,
with inverse factor $(2\pi)^{-3}$. Write
\[
 R_L=P_{\le L}=\varphi(\theta/L)(D),\quad
 P_L=R_L-R_{L/2},\quad Q_L=R_{2L}-R_{L/4},\quad L>0.
\]
If $K$ is the angular inverse transform of $\varphi$, including the factor
$(2\pi)^{-3}$, then $R_Lf=K_L*f$, $K_L(x)=L^3K(Lx)$, $\int K=1$,
and $K$ is Schwartz.  It can change sign; put $K_0=\|K\|_1$.
Every such convolution and its classical derivatives are defined for $f\in L^6$
by H\"older against the Schwartz kernels.

An annular $L^\infty$ bound can supply the missing $L^4$ integrability
without assuming $f\in L^2$. The following form of refined Sobolev is
proved directly for the chosen smooth cutoffs.

\begin{lemma}[Refined Sobolev inequality]\label{std:refined}
For $f\in\X$ put $\mathcal B(f)=\sup_{L>0}L^{-1}\|P_Lf\|_\infty$.
If $\mathcal B(f)<\infty$, then the same actual function belongs to $L^4$ and
\[
 \|f\|_4^4\le C_\varphi\mathcal B(f)^2\|\nabla f\|_2^2.
\]
The same estimate holds with the dyadic supremum in place of $\mathcal B(f)$.
\end{lemma}
\begin{proof}
Set $B=\sup_{j\in\mathbb Z}2^{-j}\|P_{2^j}f\|_\infty<\infty$.
For finite $A\subset\mathbb Z$ and $|a_j|\le B2^j$, let
$S=\sum_A|a_j|$ and $E=\sum_A2^{2j}|a_j|^2$. If $BS>0$, choose
$2^J\le S/(4B)<2^{J+1}$. The lower sum is at most $S/2$, while
Cauchy--Schwarz on the upper sum gives
\begin{equation}\label{eq:annular-scalar-split}
 S\le2\sum_{j>J}|a_j|,\qquad S^22^{2J}\le16E,
 \qquad S^2\le64B^22^{2J}.
\end{equation}
Combining the last two inequalities in \eqref{eq:annular-scalar-split}, also in the zero cases,
\begin{equation}\label{eq:annular-scalar-fourth}
 \left|\sum_{j\in A}a_j\right|^4\le1024B^2
 \sum_{j\in A}2^{2j}|a_j|^2.
\end{equation}
Write $m_L(\theta)=\varphi(\theta/L)-\varphi(2\theta/L)$. The annular
$L^2$ representative is determined by the actual gradients:
\[
 P_Lf=\sum_{k=1}^3\mathcal F_{\rm ang}^{-1}
 \left[m_L(\theta)\frac{-i\theta_k}{|\theta|^2}
              \widehat{\partial_kf}_{\rm ang}(\theta)\right],
 \qquad \|P_Lf\|_2\le2L^{-1}\|\nabla f\|_2.
\]
The symbols vanish near zero. The distributional derivative identity
identifies this representative with $R_Lf-R_{L/2}f$. If $J$ is the
largest active index, then $2^J<2|\theta|$; thus
\begin{equation}\label{eq:annular-square-symbol}
 \sum_{j\in A}2^{2j}|m_{2^j}(\theta)|^2
 \le4\left(\sum_{j\le J}2^j\right)^2
 \le64|\theta|^2.
\end{equation}
Plancherel applied to \eqref{eq:annular-square-symbol}, followed by simultaneous Schwartz approximation, gives
\begin{equation}\label{eq:annular-square-energy}
 \sum_{j\in A}2^{2j}\|P_{2^j}f\|_2^2\le64\|\nabla f\|_2^2.
\end{equation}
The countable pointwise bounds hold on a common conull set. Integrating
\eqref{eq:annular-scalar-fourth} and using \eqref{eq:annular-square-energy} gives
\begin{equation}\label{eq:annular-finite-fourth}
 \int\left|\sum_{j\in A}P_{2^j}f\right|^4
 \le65536B^2\|\nabla f\|_2^2.
\end{equation}
Finally,
$T_nf=\sum_{j=1-n}^{n}P_{2^j}f=R_{2^n}f-R_{2^{-n}}f\to f$ in $L^6$:
Plancherel and dominated convergence give the gradient limits at both
ends, and Sobolev gives the value limit. An almost-everywhere convergent
subsequence and Fatou applied to \eqref{eq:annular-finite-fourth} yield $f\in L^4$ and
$\|f\|_4^4\le65536B^2\|\nabla f\|_2^2$.
Since $B\le\mathcal B(f)$, both versions follow. The proof uses neither
monotonicity nor idempotence of the cutoff, nor an $L^2$ value of $f$.
\end{proof}

Translated Hardy inequalities control the spatial weights.

\begin{lemma}[Translated Hardy inequalities]\label{std:hardy}
For every $f\in\X$ and every $y\in\mathbb R^3$,
\[
 \left\|\frac{f}{|\cdot-y|}\right\|_2\le2\|\nabla f\|_2.
\]
For $0\le\gamma\le2$ this implies
\[
 \sup_y\int\frac{|f(x)|^{6-2\gamma}}{|x-y|^\gamma}\,dx
 \le2^\gamma S_*^{(6-3\gamma)/2}\|\nabla f\|_2^{6-2\gamma}.
\]
In particular
$\sup_y\int |x-y|^{-1}|f|^4\le2\|\nabla f\|_2\|f\|_6^3$;
and for $h\in H^1$,
$\sup_y\int |x-y|^{-1}|h|^2\le2\|h\|_2\|\nabla h\|_2$.
\end{lemma}
\begin{proof}
For Schwartz $\phi$, put $q_\varepsilon=(\varepsilon+|x-y|^2)^{-1}$,
$a_\varepsilon=q_\varepsilon(x-y)$ and
$I_\varepsilon=\int q_\varepsilon|\phi|^2$.
Since $|a_\varepsilon|^2\le q_\varepsilon$ and
$\operatorname{div}a_\varepsilon\ge q_\varepsilon$, integration by parts gives
\begin{equation}\label{eq:regularized-hardy}
 I_\varepsilon\le-2\re\int a_\varepsilon\cdot\overline\phi\nabla\phi
 \le\tfrac12I_\varepsilon+2\|\nabla\phi\|_2^2.
\end{equation}
Let $\varepsilon\downarrow0$ in \eqref{eq:regularized-hardy} by Fatou, then apply
Lemma~\ref{std:density} and Fatou again. This proves Hardy for each fixed
$y$ using one approximating sequence, without an uncountable conull-set
intersection. For $0<\gamma<2$, H\"older gives
\begin{equation}\label{eq:weighted-hardy-interpolation}
 \begin{aligned}
 \int\frac{|f|^{6-2\gamma}}{|x-y|^\gamma}
 &\le\left(\int\frac{|f|^2}{|x-y|^2}\right)^{\gamma/2}
       \left(\int|f|^6\right)^{1-\gamma/2}\\
 &\le2^\gamma\|\nabla f\|_2^\gamma\|f\|_6^{6-3\gamma}.
 \end{aligned}
\end{equation}
Sobolev in \eqref{eq:weighted-hardy-interpolation} gives the stated constant; Hardy and Sobolev themselves cover
$\gamma=2,0$. The case $\gamma=1$ gives the first remaining bound.
For $h\in H^1$, Cauchy--Schwarz gives
\[
 \int\frac{|h|^2}{|x-y|}
 \le\|h\|_2\left\|\frac h{|\cdot-y|}\right\|_2
 \le2\|h\|_2\|\nabla h\|_2.
\]
\end{proof}

The one-sided maximal function includes averages extending beyond a
given time interval. The next lemma selects terminal times from a bounded
function alone, so its use does not presuppose global $L^2$ integrability
or depend on a frequency cutoff.

\begin{lemma}[Temporal maximal estimate and terminal times]
\label{std:maximal}
For real $g\in L^2(\mathbb R)$, define
\[
 Mg(s)=\sup_{r>0}\frac1{2r}\int_{s-r}^{s+r}|g|,
 \qquad M^+g(s)=\sup_{r>0}\frac1r\int_s^{s+r}|g|.
\]
These are measurable, $\|Mg\|_2^2\le24\|g\|_2^2$ and
$\|M^+g\|_2^2\le96\|g\|_2^2$.
If $f\ge0$ is bounded and measurable on $[0,\infty)$ and
$Y(s)=\sup_{n\in\mathbb Z}2^{-n}\int_s^{s+2^n}f$, there is an increasing
sequence $T_j\to\infty$, chosen from $f$ alone, such that
\begin{equation*}
 \int_0^{T_j}Y(s)^2\,ds\le388\int_0^{T_j}f(s)^2\,ds.
\end{equation*}

\end{lemma}
\begin{proof}
Moving-interval averages are continuous. Disjoint interval selection,
with dilations $\kappa>3$ followed by the radius limit and
$\kappa\downarrow3$, gives
\begin{equation}\label{eq:maximal-weak-one}
 |\{Mh>\lambda\}|\le3\lambda^{-1}\|h\|_1
 \qquad(h\in L^1).
\end{equation}
Apply \eqref{eq:maximal-weak-one} to $h_\lambda=|g|1_{\{|g|>\lambda/2\}}$. Tonelli gives
\[
 |\{Mg>\lambda\}|\le\frac6\lambda
 \int_{|g|>\lambda/2}|g|,\qquad
 \|Mg\|_2^2\le12\int|g(s)|\int_0^{2|g(s)|}d\lambda\,ds
 =24\|g\|_2^2.
\]
Thus $M^+g\le2Mg$ yields the constant $96$; set exceptional values to zero.

Extend $f$ by zero and put $E(T)=\int_0^Tf^2$, $A(T)=E(T)/T$.
Boundedness of $A$ allows the choices
\[
 \tau(R)\ge R,\qquad \sup_{s\ge R}A(s)\le2A(\tau(R)),
 \qquad T_0=\tau(1),\quad T_{j+1}=\tau(T_j+1),
\]
with $\tau(R)=R$ if the supremum is zero. Hence $T_j\to\infty$ and,
for $T=T_j$,
\begin{equation}\label{eq:common-terminal-comparison}
 E(s)/s\le2E(T)/T\quad(s\ge T),\qquad E(2T)\le4E(T).
\end{equation}
On $0\le s\le T$, averages of length at most $T$ are bounded by
$M^+(f1_{[0,2T]})$. For $r>T$, Cauchy--Schwarz and \eqref{eq:common-terminal-comparison} give
\begin{equation}\label{eq:one-sided-maximal-comparison}
 \left(r^{-1}\int_s^{s+r}f\right)^2
 \le E(s+r)/r\le4E(T)/T.
\end{equation}
Using \eqref{eq:one-sided-maximal-comparison} and \eqref{eq:common-terminal-comparison}, including $E(T)=0$,
\[
 \int_0^TY^2\le96E(2T)+4E(T)\le388E(T).
\]
The choices depend only on $f$, with no global $L^2$ assumption.
\end{proof}

We record explicitly two often delicate HLS specializations.  At $\alpha\le1$
one may apply order-$\alpha$ HLS to a differentiated low-frequency density
with exponents $q=6/(4+\alpha)$ and $r=6/(4-\alpha)$; these are strict,
including $\alpha=1$.  One does not differentiate the kernel there.  If
$\alpha>1$, its actual locally integrable derivative instead satisfies
\[
 |\nabla I_\alpha(x)|=
 \frac{(3-\alpha)c_\alpha}{c_{\alpha-1}}I_{\alpha-1}(x).
\]
The small-sphere boundary in the distributional derivative is
$O(\varepsilon^{\alpha-1})$, so no atom occurs.  Order $\alpha-1$ HLS at
$q=6/(2+\alpha)$, $r=6/(4-\alpha)$ is strict throughout this branch.
The order-$\alpha$ call $6/(4+\alpha)$ becomes the forbidden endpoint $1$
at $\alpha=2$.  At that value the legitimate replacement, for $d\in L^1$,
is the product estimate
\[
 \int (I_2*|d|)|a|^4
 \le2c_2\|d\|_1\|\nabla a\|_2\|a\|_6^3,
\]
proved by Tonelli and Lemma~\ref{std:hardy}; it asserts no strong
$L^1\to L^3$ mapping.  Finally the Gaussian derivative bounds later use
$I_1:L^2\to L^6$ and $I_2:L^{4/3}\to L^{12}$, both strict instances of
Theorem~\ref{std:hls}.  These choices account individually for every Riesz
order used below.
\begin{lemma}[Classical analytic estimates]
\label{std:analytic-interfaces}
Fix $0<\alpha<3$, $p=3+\alpha$, the Riesz normalization $I_\alpha$
and the Fourier realization $S(t)$ used above.  Let
$Y=L^2(\R;L^{6/5}(\R^3))$, and let $\mathcal RH$ denote the
all-past $L^2$ response constructed in Theorem~\ref{std:endpoint}.
The following standard estimates hold with constants depending only on
the indicated exponents and normalizations. In the last assertion,
for $r>0$, $y\in\R^3$ and $s\in\R$, write
$T_{r,y}f(x)=r^{-1/2}f((x-y)/r)$ and
$A_{r,y,s}=T_{r,y}S(s)$, both acting on $\X$.
\begin{enumerate}
\item\textbf{Sharp HLS.}
For $a,b\in L^{6/p}(\R^3;\C)$ the product is absolutely integrable and
\begin{equation}\label{input:eq-sharp-hls}
\iint I_\alpha(x-y)|a(x)b(y)|\dd x\dd y
 \le H_\alpha\|a\|_{6/p}\|b\|_{6/p},
\end{equation}
with $H_\alpha$ given in Theorem~\ref{std:sharp}.
\item\textbf{Sharp Sobolev.}
For every $f\in\X$, in the actual weak-gradient realization,
\begin{equation}\label{input:eq-sharp-sobolev}
\|f\|_6^2\le S_*\|\nabla f\|_2^2,
 \qquad S_*^3=\frac{16}{27\pi^4}.
\end{equation}
\item\textbf{Absolute dual endpoint.}
For $H\in Y$, measurable $J\subset\R$, $t\in\R$ and
$\phi\in\mathcal S(\R^3)$,
\begin{equation}\label{input:eq-dual-endpoint}
\int_J|\langle H(s),S(t-s)\phi\rangle_{\mathrm{bil}}|\dd s
 \le C_{\mathrm d}\|1_JH\|_Y\|\phi\|_2.
\end{equation}
\item\textbf{Free critical Strichartz.}
For all $f\in L^2(\R^3)$,
\begin{equation}\label{input:eq-free-critical}
\|S(\cdot)f\|_{L_t^{10}L_x^{30/13}}\le C_{\mathrm h}\|f\|_2.
\end{equation}
\item\textbf{Retarded critical Strichartz.}
For every $H\in Y$,
\begin{equation}\label{input:eq-retarded-critical}
\|\mathcal RH\|_{L_t^{10}L_x^{30/13}}\le C_{\mathrm r}\|H\|_Y.
\end{equation}
\item\textbf{Free endpoint output.}
For every $f\in L^2(\R^3)$,
\begin{equation}\label{input:eq-free-endpoint}
\|S(\cdot)f\|_{L_t^2L_x^6}\le C_{\mathrm{he}}\|f\|_2.
\end{equation}
\item\textbf{Retarded endpoint output.}
For every $H\in Y$,
\begin{equation}\label{input:eq-retarded-endpoint}
\|\mathcal RH\|_{L_t^2L_x^6}\le C_{\mathrm{re}}\|H\|_Y.
\end{equation}
\item\textbf{Homogeneous linear profiles.}
Every bounded sequence $(f_n)\subset\X$ admits, on one subsequence,
profiles $\phi_j\in\X$ and orthogonal scale--translation--time frames
$A_{r_{j,n},y_{j,n},s_{j,n}}=T_{r_{j,n},y_{j,n}}S(s_{j,n})$
with remainders
\begin{equation}\label{input:eq-linear-profiles}
R_n^J=f_n-\sum_{j<J}A_{r_{j,n},y_{j,n},s_{j,n}}\phi_j.
\end{equation}
For every fixed admissible $J$, the kinetic identity and inverse-frame
weak convergence hold:
\begin{equation}\label{input:eq-linear-kinetic}
\begin{aligned}
 &\|\nabla f_n\|_2^2-\sum_{j<J}\|\nabla\phi_j\|_2^2
  -\|\nabla R_n^J\|_2^2\to0,\\
 &A_{r_{j,n},y_{j,n},s_{j,n}}^{-1}R_n^J\rightharpoonup0.
\end{aligned}
\end{equation}
For distinct active $j,k$, orthogonality means
\begin{equation}\label{input:eq-linear-orthogonality}
 \frac{r_{j,n}}{r_{k,n}}+\frac{r_{k,n}}{r_{j,n}}
 +\frac{|y_{j,n}-y_{k,n}|^2}{r_{j,n}r_{k,n}}
 +\frac{|r_{j,n}^2s_{j,n}-r_{k,n}^2s_{k,n}|}{r_{j,n}r_{k,n}}
 \longrightarrow\infty.
\end{equation}
A sufficiently large \emph{finite} depth has an
arbitrarily small free $L^{10}_{t,x}$ remainder. For a finite family,
the terminal remainder and inactive zero profiles are included.
\end{enumerate}
\end{lemma}
\begin{proof}
The two sharp inequalities are the $N=3$ instances of
\cite{FrankLieb}, with
Riesz factor $c_\alpha$.  Lemma~\ref{std:density} and the
real--imaginary argument in Theorem~\ref{std:sharp} give the complex
weak-gradient domain in \eqref{input:eq-sharp-hls} and
\eqref{input:eq-sharp-sobolev}.

Keel--Tao \cite{KeelTao}, for the Schr\"odinger
admissible pairs $(2,6)$ and $(10,30/13)$, supplies the numerical bounds
in \eqref{input:eq-dual-endpoint}--\eqref{input:eq-retarded-endpoint}.
In particular, the double endpoint is invoked
directly and is not inferred by an endpoint limit.  Theorem~\ref{std:endpoint}
identifies the common actual $L^2$ representative and the finite/future
operators; \eqref{input:eq-dual-endpoint} is its localized dual pairing.

The linear extraction in \eqref{input:eq-linear-profiles} is the
$d=3$ case of \cite[Theorem~4.27, (4.61)--(4.65)]{KillipVisan},
which presents Keraani's energy-critical decomposition
\cite{Keraani2001}. It gives the free $L^{10}_{t,x}$ remainder,
gradient decoupling and frame separation used here. The homogeneous
completion and the Fourier convention above identify its profiles with
the $L^6$ representatives in $\X$.
\end{proof}

Lemma~\ref{std:analytic-interfaces} supplies the linear and variational estimates used in the nonlinear argument.

\section{Local theory, conservation, and scattering}
\label{local:section}

This section constructs the flow in the homogeneous energy space and
establishes the stability and scattering criteria needed for compactness.
The nonlinear estimates retain both differentiated Hartree terms;
the conservation laws are justified by regularization in the same
solution class.

Throughout this section $0<\alpha<3$, $p=3+\alpha$, and
$F(u)=(I_\alpha*|u|^p)|u|^{p-2}u$ with the normalization fixed above.
For an interval $I$, write
\[
 \begin{aligned}
 \mathsf S_I(u)&=\|u\|_{L^{10}_tL^{10}_x(I)},&
 \mathsf W_I(u)&=\|\nabla u\|_{L^{10}_tL^{30/13}_x(I)},\\
 \mathsf E_I(u)&=\|\nabla u\|_{L^\infty_tL^2_x(I)},&
 \mathsf Z_I(u)&=\|\nabla u\|_{L^2_tL^6_x(I)},\\
 \mathsf N_I(H)&=\|H\|_{L^2_tL^{6/5}_x(I)}.
 \end{aligned}
\]
The linear estimates \eqref{input:eq-dual-endpoint}--
\eqref{input:eq-retarded-endpoint} act on the weak derivatives.
Spatial vector norms use the Euclidean norm on $\C^3$.
On a compact interval $I$, the solution space is
\[
 \mathcal X(I)=\left\{u\in C(I;\X):
 \partial_j u\in L^{10}(I;L^{30/13}),\ 1\le j\le3\right\}.
\]
Its norm is
\[
 \|u\|_{\mathcal X(I)}=
 \max\left\{\sup_{t\in I}\|u(t)\|_{\mathrm{gr}},
             \max_j\|\partial_j u\|_{L^{10}_tL^{30/13}_x}\right\}.
\]
Here the Bochner gradients are the weak derivatives of the continuous
energy path. This compatibility is closed in
$C(I;\X)\times(L^{10}(I;L^{30/13}))^3$: take one subsequence converging
almost everywhere in time in all three gradient coordinates, and identify
the limits distributionally. Hence $\mathcal X(I)$ is complete, with
norm equivalent to $\mathsf E_I+\mathsf W_I$ by fixed dimensional and
Sobolev constants.

The critical value norm is controlled by a spatial Sobolev estimate.

\begin{lemma}[Critical Sobolev estimate]\label{local:critical-sobolev}
If $f\in\X$ and $\nabla f\in L^{30/13}$, then
\[
 \|f\|_{10}\le C\|\nabla f\|_{30/13}.
\]
Consequently $\mathsf S_I(u)\le C\mathsf W_I(u)$ on the path graph.
\end{lemma}
\begin{proof}
For real $f\in C_c^\infty$, the Sobolev estimate
\eqref{input:eq-sharp-sobolev} applied to a smooth approximation
of $|f|^{5/3}$ and H\"older give
\begin{equation}\label{eq:critical-sobolev-compact}
 \|f\|_{10}^{5/3}
 \le\tfrac53\sqrt{S_*}\|f\|_{10}^{2/3}\|\nabla f\|_{30/13}.
\end{equation}
Cancellation in \eqref{eq:critical-sobolev-compact} proves the estimate. For $f\in L^6$ with
$\nabla f\in L^q$, $q=30/13$, cutoff and convolution give
\begin{equation}\label{eq:critical-sobolev-density}
 f_n\in C_c^\infty,\qquad f_n\to f\text{ in }L^6,\qquad
 \nabla f_n\to\nabla f\text{ in }L^q,
\end{equation}
since the cutoff error satisfies
\[
 \|f\nabla\chi_R\|_q
 \le CR^{-1/5}\|f1_{\{R\le|x|\le2R\}}\|_6\longrightarrow0.
\]
Applying \eqref{eq:critical-sobolev-compact} to differences in \eqref{eq:critical-sobolev-density} makes $(f_n)$ Cauchy in $L^{10}$; local H\"older
identifies its limit with $f$. Apply the real estimate to both
components of a complex value, then integrate in time.
\end{proof}

We use this comparison in the nonlinear estimates.

\subsection{Nonlinear estimates and weak derivatives}

\begin{lemma}[Nonlinear estimates]
\label{local:nonlinear}
The actual convolution map $F:L^6\to L^{6/5}$ is locally Lipschitz.
On $\mathcal X(I)$ it has an actual weak gradient and
\begin{equation}\label{local:single-source}
 \mathsf N_I(\nabla F(u))
 \le C\|u\|_{L^\infty_tL^6_x}^{2\alpha}\mathsf S_I(u)^4\mathsf W_I(u),
 \qquad
 \|F(u)\|_{L^2_tL^2_x}
 \le C\|u\|_{L^\infty_tL^6_x}^{2\alpha}\mathsf S_I(u)^5.
\end{equation}
For $u,v\in\mathcal X(I)$, put
\[
 \begin{gathered}
 M=\|u\|_{L^\infty L^6}+\|v\|_{L^\infty L^6},\quad
 N=\mathsf S_I(u)+\mathsf S_I(v),\quad Y=\mathsf W_I(v),\\
 D=\|u-v\|_{L^\infty L^6},\quad e=\mathsf S_I(u-v),\quad
 z=\mathsf W_I(u-v),\quad \vartheta=(3-\alpha)/3.
 \end{gathered}
\]
Then
\begin{equation}\label{local:difference-source}
 \mathsf N_I(\nabla(F(u)-F(v)))
 \le C\left(M^{2\alpha}N^4z+M^{2\alpha}N^3Ye
 +M^{5\alpha/3}N^{4-\vartheta}Y D^{\alpha/3}e^\vartheta\right).
\end{equation}

\end{lemma}
\begin{proof}
The derivative may fall on the convolved density or on the outside
factor. We estimate these contributions with separate HLS allocations,
keeping the derivative on the density when it lies inside the convolution.
For $R(z)=|z|^p$ and $B(z)=|z|^{p-2}z$, segment integration on
$\C\simeq\R^2$ gives, since $p>3$,
\begin{equation}\label{eq:power-differential-bounds}
 \begin{aligned}
 |R(z)-R(w)|&\le p(|z|+|w|)^{p-1}|z-w|,\\
 |B(z)-B(w)|&\le C_p(|z|+|w|)^{p-2}|z-w|,\\
 \|DR(z)-DR(w)\|&\le p(p-1)(|z|+|w|)^{p-2}|z-w|,\\
 \|DB(z)-DB(w)\|&\le C_p(|z|+|w|)^{p-3}|z-w|.
 \end{aligned}
\end{equation}
The HLS estimate \eqref{input:eq-sharp-hls} at
$6/p\to6/(3-\alpha)$ and H\"older applied to
\begin{equation}\label{eq:hartree-two-slot-difference}
 F(u)-F(v)=(I_\alpha*(R(u)-R(v)))B(u)
           +(I_\alpha*R(v))(B(u)-B(v))
\end{equation}
and \eqref{eq:power-differential-bounds} give the energy-level domain
and local Lipschitz bound from \eqref{eq:hartree-two-slot-difference}.

At a fixed time put $V=|u|+|v|$, $d=|u-v|$,
$Y_x=|\nabla v|$ and $Z_x=|\nabla u-\nabla v|$.
The gradient difference is bounded by $C_p\sum_{i=1}^6T_i$, where
\[
 \begin{array}{ll}
 T_1=(I_\alpha*V^p)V^{p-2}Z_x,&
 T_2=(I_\alpha*(V^{p-1}d))V^{p-2}Y_x,\\
 T_3=(I_\alpha*V^p)V^{p-3}dY_x,&
 T_4=(I_\alpha*(V^{p-1}Z_x))V^{p-1},\\
 T_5=(I_\alpha*(V^{p-2}dY_x))V^{p-1},&
 T_6=(I_\alpha*(V^{p-1}Y_x))V^{p-2}d.
 \end{array}
\]
Each pair in the table allocates powers of $V$ to $L^6,L^{10}$;
each gradient lies in $L^{30/13}$:
\[
 \begin{array}{c|c|c|c}
 &\text{inside }I_\alpha&\text{outside }I_\alpha&\text{difference factor}\\ \hline
 T_1&(2\alpha,3-\alpha)&(0,1+\alpha)&Z_x\\
 T_2&(5\alpha/3,2-2\alpha/3)&(0,1+\alpha)&\|d\|_6^{\alpha/3}\|d\|_{10}^{\vartheta}\\
 T_3&(2\alpha,3-\alpha)&(0,\alpha)&\|d\|_{10}\\
 T_4&(\alpha/3,2+2\alpha/3)&(5\alpha/3,2-2\alpha/3)&Z_x\\
 T_5&(\alpha/3,1+2\alpha/3)&(5\alpha/3,2-2\alpha/3)&\|d\|_{10}\\
 T_6&(\alpha/3,2+2\alpha/3)&(4\alpha/3,\vartheta)&\|d\|_6^{\alpha/3}\|d\|_{10}^{\vartheta}
 \end{array}
\]
All powers are nonnegative. The HLS input reciprocals are
$a=(9+7\alpha)/30$ and $b=(57+11\alpha)/90$, and
\begin{equation}\label{eq:nonlinear-hls-exponents}
 \begin{gathered}
 a-\alpha/3=(3-\alpha)/10,\qquad 1-a=7(3-\alpha)/30,\\
 b-\alpha/3=19(3-\alpha)/90,\qquad 1-b=11(3-\alpha)/90.
 \end{gathered}
\end{equation}
The identities \eqref{eq:nonlinear-hls-exponents} verify both strict HLS calls; each product has reciprocal exponent $5/6$.
The pairs $(T_1,T_4)$, $(T_3,T_5)$ and $(T_2,T_6)$ give the three
terms of \eqref{local:difference-source}, using
\[
 \frac4{10}+\frac1{10}
 =\frac3{10}+\frac1{10}+\frac1{10}
 =\frac{4-\vartheta}{10}+\frac1{10}+\frac\vartheta{10}
 =\frac12,\qquad \alpha/3+\vartheta=1.
\]
Comparison with zero gives the gradient estimate in
\eqref{local:single-source}. For the scalar estimate use inner powers
$(2\alpha,3-\alpha)$ and outer power $p-1$ in $L^{10}$;
both reciprocal sums equal $1/2$.

For compact smooth values, testing and Fubini differentiate $R(u)$
inside the convolution, where $I_\alpha$ is locally integrable.
For rough slices, cutoff and convolution converge in $L^6\cap L^{10}$
and in both gradient spaces, since
\[
 \|u\nabla\chi_R\|_{30/13}
 \le\|u1_{\{|x|\ge R\}}\|_{10}\|\nabla\chi\|_3\longrightarrow0.
\]
The six bounds make the candidate gradients Cauchy in $L^{6/5}$;
energy-level continuity identifies their weak limit with $\nabla F(u)$.
Compatible Bochner representatives give joint measurability on one
conull time set. The same allocations, with the gradient removed and
outer power increased, give
\begin{equation}\label{eq:scalar-source-difference}
 \|F(u)-F(v)\|_{L^2L^2}
 \le C\bigl(M^{2\alpha}N^4e+
 M^{5\alpha/3}N^{5-\vartheta}D^{\alpha/3}e^\vartheta\bigr).
\end{equation}
The estimate \eqref{eq:scalar-source-difference} proves scalar source continuity.
\end{proof}

The value and gradient source bounds now define the finite Duhamel integral.

\subsection{Linear realization and the nonlinear fixed point}

\begin{lemma}[Compatible sources and the finite Duhamel formula]
\label{local:linear-duhamel}\label{local:mild}
Let $I=[a,b]$.  Suppose $H\in L^1(I;L^{6/5})$ and its actual distributional
gradients $G_j=\partial_jH$ belong to $L^2(I;L^{6/5})$.  Then the literal
scalar distribution
\begin{equation}\label{eq:finite-duhamel-definition}
 D_aH(t)=\int_a^t S(t-s)H(s)\,ds
\end{equation}
has a unique realization in $\mathcal X(I)$ with zero trace at $a$ and
\[
 \|D_aH\|_{\mathcal X(I)}+\mathsf Z_I(D_aH)
 \le C\sum_j\|G_j\|_{L^2_tL^{6/5}_x}.
\]
The integral is absolutely convergent against Schwartz tests. The terminal
operator $D_b^-H(t)=\int_t^bS(t-s)H(s)\,ds$ satisfies the same estimate.
If $u\in\mathcal X(I)$ solves $iu_t+\Delta u=H$ in compact time-space tensor
tests, then, for all $s,t\in I$,
\begin{equation}\label{local:forced-mild}
 u(t)=S(t-s)u(s)-i\int_s^tS(t-\tau)H(\tau)\,d\tau
 \quad\hbox{in }\mathcal S'.
\end{equation}
Conversely, the formula implies the tested equation.
\end{lemma}
\begin{proof}
Let $J_n=(1-B_{n+1})B_{1/(n+1)}$ as in Lemma~\ref{std:bessel}.
It commutes with weak derivatives, is bounded by two on the required
$L^q$ spaces, converges strongly to the identity, and maps
$L^{6/5}$ to $\X$. Thus $D_a(J_nH)$ is a finite $\X$ Bochner
integral. Testing and Fubini identify its gradients with the retarded realizations
of $J_nG_j$. The numerical estimates
\eqref{input:eq-retarded-critical} and
\eqref{input:eq-retarded-endpoint}, with Theorem~\ref{std:endpoint}, give
\begin{equation}\label{eq:duhamel-source-cauchy}
 \|D_a(J_nH-J_mH)\|_{\mathcal X(I)}
 +\mathsf Z_I(D_a(J_nH-J_mH))
 \le C\sum_j\|(J_n-J_m)G_j\|_{L^2L^{6/5}}\longrightarrow0.
\end{equation}
The Cauchy estimate \eqref{eq:duhamel-source-cauchy} and closedness of the graph give a limit $D_aH$. Since
$J_nH\to H$ in $L^1L^{6/5}$, testing against uniformly bounded
free Schwartz orbits in $L^6$ identifies its scalar pairing with
\eqref{eq:finite-duhamel-definition}. Future integration gives $D_b^-H$ with a common
enlarged constant.

Lemma~\ref{std:generator} and Fubini give the tested equation, first
for $J_nH$ and then for its limit. If $u$ solves that equation, put
\begin{equation}\label{eq:weak-duhamel-difference}
 q(t)=u(t)-S(t-a)u(a)+iD_aH(t).
\end{equation}
For $q$ in \eqref{eq:weak-duhamel-difference}, each $\partial_jq\in C_tL^2$ solves the homogeneous equation with
zero trace. Lemma~\ref{std:linear-uniqueness} gives $\nabla q=0$;
Sobolev gives $q=0$. Rebasing and oriented integration yield
\eqref{local:forced-mild} for all $s,t$, including the endpoints by
continuity. The same tested calculation proves the converse.
\end{proof}

The preceding value and gradient bounds make the Hartree Duhamel map
a contraction on $\mathcal X(I)$. Shrinking $I$ makes the free critical
norm small while the energy norm remains bounded.

\begin{theorem}[The maximal local flow]\label{local:flow}
For every $s_0\in\mathbb R$ and $f\in\X$, there is a unique maximal solution
$u\in C(I_{\max};\X)$, on an open interval containing $s_0$, with $u(s_0)=f$,
whose restrictions to compact subintervals belong to $\mathcal X$ and solve
\begin{equation}\label{local:hartree-mild}
 u(t)=S(t-s)u(s)+i\int_s^tS(t-\tau)F(u(\tau))\,d\tau.
\end{equation}
On every compact subinterval,
\begin{equation}\label{local:fullz}
 \nabla u\in L^2_tL^6_x,
\end{equation}
and the four path norms $\mathsf E,\mathsf W,\mathsf Z,\mathsf S$ are finite.
The flow depends continuously on its datum in $C_t\X$ and these norms on
compact lifespan intervals. Compact subsets of $\X$ have a uniform
positive local existence time.
\end{theorem}
\begin{proof}
On a forward interval set
$\Phi_f(v)=S(\cdot-s_0)f+iD_{s_0}F(v)$ and choose
$B>2C\|f\|_{\mathrm{gr}}+1$. On the complete graph ball with
$v(s_0)=f$, $\mathsf E(v)\le B$ and $\mathsf W(v)\le\eta$,
Lemma~\ref{local:critical-sobolev} gives $\mathsf S(v)\le C\eta$.
For the graph distance $d$, the differences in
\eqref{local:difference-source} obey
$D,e,z\le Cd$ and $D^{\alpha/3}e^\vartheta\le Cd$.
Choose $0<\eta\le1$ so that
\begin{equation}\label{eq:local-contraction-conditions}
 \begin{aligned}
 \operatorname{Lip}(\Phi_f)
 &\le C(2B^{2\alpha}+B^{5\alpha/3})\eta^4\le\tfrac12,\\
 CB^{2\alpha}\eta^5&\le\min(\eta/2,B/4).
 \end{aligned}
\end{equation}
The free estimate \eqref{input:eq-free-critical} and absolute
continuity give a short interval with free $\mathsf W\le\eta/2$;
the free energy norm is at most $B/2$. Together with
\eqref{eq:local-contraction-conditions}, this makes $\Phi_f$ an invariant half contraction.
Its fixed point satisfies \eqref{local:hartree-mild}. The homogeneous
endpoint bound \eqref{input:eq-free-endpoint} and its retarded counterpart
\eqref{input:eq-retarded-endpoint} give \eqref{local:fullz}.
The future operator in Lemma~\ref{local:linear-duhamel} gives
$u=S(\cdot-s_0)f-iD_{s_0}^-F(u)$ backward; gluing the two traces
gives a centered solution.

For two solutions, partition a compact overlap so that both critical
norms are below a common threshold. On each piece, the difference
estimate and a common trace give $d\le d/2$, hence equality.
Their union over intervals marked by $(s_0,f)$ is consequently a
well-defined path on an open interval. Gluing members through the two
endpoints of a compact subinterval proves its graph regularity;
the union is maximal.

For nearby data, the same contraction and endpoint estimates give
\begin{equation}\label{eq:local-two-data-bound}
 \|u_f-u_g\|_{\mathcal X}
 +\mathsf Z(u_f-u_g)+\mathsf S(u_f-u_g)
 \le C\|f-g\|_{\mathrm{gr}}.
\end{equation}
A finite small-norm partition propagates \eqref{eq:local-two-data-bound} through any
compact part of the lifespan. The strict contraction bounds hold in a
neighborhood of each datum. A finite subcover of a compact data set
then gives one positive centered existence time.
\end{proof}

Uniqueness identifies the transformed solution under critical rescaling.

\begin{lemma}[Affine covariance]\label{local:covariance}
For $r>0$, $y\in\mathbb R^3$ and $\tau\in\mathbb R$, set
$T_{r,y}f(x)=r^{-1/2}f((x-y)/r)$.  If $u$ is the maximal trajectory marked
at $(s_0,f)$, then
\[
 v(t)=T_{r,y}u((t-\tau)/r^2)
\]
is the maximal trajectory marked at $(\tau+r^2s_0,T_{r,y}f)$, with domain
exactly $\tau+r^2I_{\max}(u)$.  Energy and the four critical mixed norms
are invariant under this simultaneous transformation.
\end{lemma}
\begin{proof}
Changes of variables in the weak derivative and Riesz integrals give
\begin{equation}\label{eq:affine-gradient-nonlinearity}
 \nabla T_{r,y}f=r^{-3/2}(\nabla f)((\cdot-y)/r),\qquad
 F(T_{r,y}f)=r^{-5/2}F(f)((\cdot-y)/r),
\end{equation}
\begin{equation}\label{eq:affine-energy-free-group}
 A(T_{r,y}f)=A(f),\qquad D(T_{r,y}f)=D(f),\qquad
 S(t)T_{r,y}=T_{r,y}S(t/r^2).
\end{equation}
Here the potential factor is $r^{\alpha-p/2}$ and the double interaction
factor is $r^{3+\alpha-p}=1$. Substitute $t=\tau+r^2s$ in
\eqref{local:hartree-mild}, using \eqref{eq:affine-gradient-nonlinearity}
and \eqref{eq:affine-energy-free-group}. The spatial and temporal Jacobians
$r^3,r^2$ also give invariance of the four critical norms. This proves
that the transformed path is a solution on the stated domain. Applying
the inverse transformation to any extension proves maximality and exact
equality of the lifespans.
\end{proof}

We turn from covariance to conservation on the maximal lifespan.

\subsection{Conservation in the homogeneous energy space}

\begin{theorem}[Energy conservation]\label{local:energy}
The functional $E(f)=\frac12A(f)-\frac1{2p}D(f)$ is real $C^1$ on $\X$,
and every trajectory in Theorem~\ref{local:flow} has constant energy on its
actual maximal lifespan.
\end{theorem}
\begin{proof}
For $|t|\le1$, the derivative of the interaction integrand at $f+th$
is dominated by the integrable function
\[
 2pI_\alpha(x-y)(|f(x)|+|h(x)|)^p(|f(y)|+|h(y)|)^p.
\]
Dominated differentiation, symmetry and HLS give
$D'(f)[h]=2p\re\int\overline{F(f)}h$. Thus
\begin{equation}\label{local:energy-differential}
 E'(f)[h]=\sum_j\re\int\overline{\partial_jf}\,\partial_jh
             -\re\int\overline{F(f)}h.
\end{equation}
Continuity of $F:L^6\to L^{6/5}$ and integration on segments make
this a continuous real Fr\'echet derivative.

Fix $[a,b]$ in the lifespan. To justify differentiation of the
energy without assuming finite mass, regularize both the low and high
frequencies. Let $s_n=n+1$ and use the Bessel regularizer $J_n$ with
real even symbol
\[
 m_n(\xi)=\frac{s_n^2|\xi|^2}
 {(1+s_n^2|\xi|^2)(1+s_n^{-2}|\xi|^2)}.
\]
Then $0\le m_n\le1$, $|\xi|m_n\le s_n$,
$|\xi|^2m_n\le s_n^2$, and $m_n\to1$ off zero.
The value $v_n=J_nu$ lies in $L^2$, since
\begin{equation}\label{eq:regularized-value-recovery}
 \widehat v_n=\sum_jr_{n,j}\widehat{\partial_ju},\qquad
 r_{n,j}=\frac{-i}{2\pi}
 \frac{s_n^2\xi_j}{(1+s_n^2|\xi|^2)(1+s_n^{-2}|\xi|^2)},
 \quad\|r_{n,j}\|_\infty\le\frac{s_n}{2\pi}.
\end{equation}
Distributional testing identifies the value in \eqref{eq:regularized-value-recovery} with $J_nu$.
The same multipliers give $\partial_jv_n,\Delta v_n\in\X$ and
$\Delta v_n\in L^2$; these assertions concern only $v_n$.

Set $Q_n=J_nF(u)\in C([a,b];\X)$. The tested Duhamel formula gives
\begin{equation}\label{eq:regularized-mild-flow}
 v_n(t)=S(t-a)J_nu(a)+i\int_a^tS(t-r)Q_n(r)\,dr.
\end{equation}
In interaction coordinates the integral in \eqref{eq:regularized-mild-flow} is $C^1$. The bound
$|\xi|^2m_n\le s_n^2$ permits dominated convergence of gradient
Fourier difference quotients, so
\begin{equation}\label{eq:regularized-flow-derivative}
 v_n\in C^1([a,b];\X),\qquad v_n'=i\Delta v_n+iQ_n.
\end{equation}
Put $R_n=Q_n-F(v_n)$. Lemma~\ref{local:nonlinear} gives, a.e. in time,
\[
 F(v_n),Q_n,R_n\in L^2,\qquad \partial_jR_n\in L^{6/5}.
\]
Substituting \eqref{eq:regularized-flow-derivative} in \eqref{local:energy-differential} yields
\begin{equation}\label{local:energy-residual}
 \frac d{dt}E(v_n)
 =\sum_j\re\int\overline{\partial_jv_n}\,i\partial_jR_n
       -\re\int\overline{F(v_n)}\,iR_n.
\end{equation}
Here integration by parts pairs
$H,\partial_jH\in L^{6/5}$ with $h,\partial_jh\in L^6$.
For $h=\partial_jv_n$, apply the bounded cutoff multipliers to
homogeneous Schwartz approximants of $u$ and pass by H\"older.
The two nonlinear cross terms cancel; Plancherel cancels the kinetic
term, and $\re\int\overline{F(v_n)}\,iF(v_n)=0$ in $L^2$.

The ordinary fundamental theorem for the $C^1$ state path, followed
by \eqref{local:energy-residual} a.e., gives
\begin{equation}\label{eq:energy-commutator-bound}
 \begin{split}
 |E(v_n(b))-E(v_n(a))|
 &\le\sum_j\|\partial_jv_n\|_{L^2L^6}
                \|\partial_jR_n\|_{L^2L^{6/5}}\\
 &\quad+\|F(v_n)\|_{L^2L^2}\|R_n\|_{L^2L^2}.
 \end{split}
\end{equation}
Strong Bessel convergence, uniform on the compact state image, gives
$v_n\to u$ in $C_t\X$ and in the critical graph. Hence
\begin{equation}\label{eq:energy-commutator-convergence}
 \begin{aligned}
 \nabla R_n&=J_n\nabla F(u)-\nabla F(J_nu)
       \longrightarrow0&&\text{in }L^2L^{6/5},\\
 R_n&\longrightarrow0&&\text{in }L^2L^2.
 \end{aligned}
\end{equation}
The remaining factors in \eqref{eq:energy-commutator-bound} are bounded by
\eqref{local:fullz} and \eqref{local:single-source}. Using
\eqref{eq:energy-commutator-convergence} and continuity of $E$ proves $E(u(b))=E(u(a))$.
\end{proof}

Combining conservation with the sharp inequality traps the kinetic energy below the ground state.

\begin{lemma}[Energy trapping]\label{local:trapping}
Suppose $E(f)<E(W)$ and $A(f)<G$.  On the entire actual maximal lifespan,
\[
 E(u(t))=E(f),\qquad
 A(u(t))\le kE(f)<G,\qquad k=\frac{2p}{p-1}.
\]
For $0\le E_*<E(W)$ and $E(f)\le E_*$ the positive number
$\delta_*=1-(kE_*/G)^{p-1}$ satisfies
\[
 A(u(t))-D(u(t))\ge\delta_* A(u(t)).
\]
On this branch, $E(f)=0$ implies $f=0$.
\end{lemma}
\begin{proof}
Set $h(a)=a/2-G^{1-p}a^p/(2p)$. On $[0,G]$ it is increasing, and
$h(G)=E(W)$. Sharp HLS and Sobolev give $E(u(t))\ge h(A(u(t)))$.
Continuity and energy conservation therefore prevent $A(u)$ from reaching
$G$ from below, in either time direction. On this branch,
\begin{equation}\label{eq:trapping-algebra}
 D(u)\le A(u),\qquad E(f)\ge\frac{p-1}{2p}A(u),\qquad
 D(u)\le A(u)\left(\frac{kE_*}{G}\right)^{p-1}.
\end{equation}
The inequalities \eqref{eq:trapping-algebra} give the cap and the stated gap. Sobolev identifies $A=0$ with the
zero state, which proves the remaining assertions.
\end{proof}

The kinetic bound also controls local charge, without requiring finite total mass.

\begin{lemma}[Local charge and flux bounds]\label{local:charge}
Let $u$ be an actual solution and $\zeta\in C_c^\infty(\mathbb R^3;\mathbb R)$.
Then $M_\zeta(t)=\int\zeta^2|u(t)|^2$ is $C^1$ on the lifespan and
\begin{equation}\label{local:charge-identity}
 M_\zeta'(t)=2\operatorname{Im}\int\nabla(\zeta^2)\cdot\nabla u\,\overline u.
\end{equation}
For $\zeta_R(x)=\zeta(x/R)$ and a uniform gradient bound $B$,
\[
 |M_{\zeta_R}'(t)|\le C_\zeta B^2,
\]
with $C_\zeta$ independent of $R$. If $u(t)\in L^2$ on a compact
interval $J$ and $\sup_{t\in J}\|u(t)\|_2<\infty$, then also
$|M_{\zeta_R}'(t)|\le C_\zeta R^{-1}\|u(t)\|_2\|\nabla u(t)\|_2$ on $J$.
\end{lemma}
\begin{proof}
Local H\"older and the equation give
\[
 \begin{gathered}
 v=\zeta u\in C_tH^1,\\
 v_t=i\Delta v+i\zeta F(u)-2i\nabla\zeta\cdot\nabla u
       -i(\Delta\zeta)u\in C_tH^{-1}.
 \end{gathered}
\]
Time convolution on a slightly larger compact interval, followed by
uniform convergence in $H^1$ and $H^{-1}$, yields
\[
 \|v(t)\|_2^2-\|v(s)\|_2^2
 =2\re\int_s^t\langle v_\tau,v\rangle\,d\tau.
\]
The integrand is continuous. Since HLS gives the real $L^1$ function
$F(u)\overline u=(I_\alpha*|u|^p)|u|^p$, its imaginary contribution
vanishes. Integration by parts in the dual pairing proves
\eqref{local:charge-identity}. H\"older and Sobolev then give
\[
 |M_{\zeta_R}'|
 \le2\|\nabla(\zeta_R^2)\|_3\|\nabla u\|_2\|u\|_6
 \le C_\zeta\|\nabla u\|_2^2,
\]
since the coefficient norm is scale invariant. Under the stated
$L^2$ hypothesis, instead use
\[
 |M_{\zeta_R}'|
 \le2\|\nabla(\zeta_R^2)\|_\infty\|\nabla u\|_2\|u\|_2
 \le C_\zeta R^{-1}\|\nabla u\|_2\|u\|_2.
\]
\end{proof}

For reconstruction, stability constants must depend on the energy and
critical spacetime bounds, but not on the length of the comparison
interval. A finite small-norm partition provides this uniformity.

\subsection{Stability, continuation, and wave operators}

\begin{theorem}[Interval-independent perturbation]\label{local:perturbation}
For finite $B,M$ there exist $\varepsilon_0>0$ and $C<\infty$, independent
of the compact interval $I=[a,b]$, with the following property.  Let
$q\in\mathcal X(I)$ solve
$iq_t+\Delta q+F(q)=e$ in actual compact tensor tests, where
$e\in L^1(I;L^{6/5})$ and its actual gradient belongs to the source space.
Suppose $\mathsf E_I(q)\le B$, $\mathsf S_I(q)\le M$, and
\[
 \|\nabla(f-q(a))\|_2+\mathsf N_I(\nabla e)\le\varepsilon\le\varepsilon_0.
\]
Then the genuine solution with initial datum $f$ exists through $I$, and
\begin{equation}\label{local:perturbation-bound}
 \mathsf E_I(u-q)+\mathsf W_I(u-q)+\mathsf Z_I(u-q)+\mathsf S_I(u-q)
 \le C\varepsilon.
\end{equation}
Also $\sup_{t\in I}\|u(t)-q(t)\|_{\mathrm{gr}}\le C\varepsilon$.
The same bounds hold backward from terminal data $f$ close to $q(b)$.
The energy cap may be essential on a positive-length interval.
\end{theorem}
\begin{proof}
Sobolev gives a state cap; continuity extends an essential cap to all
times on a positive-length interval. On $J$ with
$\mathsf S_J(q)\le\eta(B)$, the forced formula gives
\begin{equation}\label{eq:background-graph-absorption}
 \|q\|_{\mathcal X(J)}\le C_fB+Q\mathsf N_J(\nabla e)
       +CB^{2\alpha}\mathsf S_J(q)^4\|q\|_{\mathcal X(J)}.
\end{equation}
Choose $CB^{2\alpha}\eta^4\le1/2$ and absorb in the already finite
graph norm. For total error at most one, a partition by the cumulative
$\mathsf S^{10}$ integral has at most $1+(M/\eta)^{10}$ pieces and
therefore gives a graph bound $L=L(B,M)$.

At state cap $B+1$, choose $\rho$ so that the difference map is a
half contraction when both background norm and correction radius are
at most $\rho$. Fix $0<r\le\min(1,\rho/2)$ and partition the
tenth-power integral of the sum of gradient-coordinate norms into at most
\[
 N_* =\left\lceil1+(3L/r)^{10}\right\rceil
\]
pieces, each with background critical norm at most $r$.
On a piece anchored at $c$, write
\begin{equation}\label{eq:perturbation-piece-map}
 \begin{gathered}
 w=d+\mathcal N(w),\qquad
 d=S(\cdot-c)(f-q(c))+iD_ce,\\
 \mathcal N(w)=\Phi_{q(c)}(q+w)-\Phi_{q(c)}(q),\qquad
 \mathcal N(0)=0,\quad\operatorname{Lip}(\mathcal N)\le\tfrac12,\\
 \|d\|_{\mathcal X}\le C_f\|f-q(c)\|_{\mathrm{gr}}
                                  +Q\mathsf N_J(\nabla e).
 \end{gathered}
\end{equation}
If $\|d\|_{\mathcal X}\le r/2$, the contraction in \eqref{eq:perturbation-piece-map} yields
$\|w\|_{\mathcal X}\le2\|d\|_{\mathcal X}\le r$.
The endpoint estimate and Sobolev also control $\mathsf Z(w)$ and
$\mathsf S(w)$.

Set $D_*=2C_f+2Q+1$ and, including the fixed conversion of the initial
homogeneous error, choose
\[
 \varepsilon_0\le\min\left(1,\frac{r}{2D_*^{N_*+1}+1}\right).
\]
Writing $e_j$ for the entering trace error and $d_j$ for the seed,
the induction is
\begin{equation}\label{eq:perturbation-trace-induction}
 \begin{aligned}
 e_j&\le D_*^j\varepsilon,\\
 \|d_j\|_{\mathcal X}&\le(C_fD_*^j+Q)\varepsilon\le r/2,\\
 e_{j+1}&\le2\|d_j\|_{\mathcal X}\le D_*^{j+1}\varepsilon.
 \end{aligned}
\end{equation}
The endpoint traces glue by rebasing the scalar integrals. Summation
of \eqref{eq:perturbation-trace-induction} over at most $N_*$ pieces proves \eqref{local:perturbation-bound},
including the supremum state bound. Centered existence at the final
trace places $I$ inside an open solution domain.

For terminal data,
$q=\Phi^-_{q(b)}(q)+iD_b^-e$; the seed contains $-iD_b^-e$.
Use the future bounds of Theorem~\ref{std:endpoint} and the same
induction \eqref{eq:perturbation-trace-induction} from right to left, with their own constants.
\end{proof}

Finite critical size then controls continuation and scattering.

\begin{theorem}[Finite-size continuation and scattering]\label{local:finite-s}
At a maximal endpoint, a uniform gradient bound and finite one-sided
$\mathsf S$ norm imply continuation if the endpoint is finite, or scattering
in $\X$ if it is infinite. A finite maximal endpoint also forces the
approaching $\mathcal X$ norms to be unbounded. If the two bounds hold on
both sides, the solution is global and scatters in both directions.
\end{theorem}
\begin{proof}
Let $T$ be a finite right endpoint. Uniform graph bounds on $[a,t]$,
countable exhaustion and Fatou give a compatible source
\begin{equation}\label{eq:finite-endpoint-source}
 \nabla F(u)\in L^2([a,T);L^{6/5}),\qquad
 F(u)\in L^1([a,T);L^{6/5}).
\end{equation}
Compact tests identify the gradient in \eqref{eq:finite-endpoint-source};
the state cap gives its scalar-source membership. The dual endpoint estimate gives
\begin{equation}\label{eq:interaction-coordinate-cauchy}
 \|S(-t)u(t)-S(-s)u(s)\|_{\dot H^1}
 \le C\|\nabla F(u)\|_{L^2((s,t);L^{6/5})}\longrightarrow0
 \quad(s,t\uparrow T).
\end{equation}
By \eqref{eq:interaction-coordinate-cauchy}, completeness and free continuity yield $u(T)\in\X$. The finite scalar
Duhamel formula passes to this trace; its retarded realization closes
the graph on $[a,T]$. Centered existence at $u(T)$ and trace gluing
contradict maximality. Decreasing exhaustion and the future operator
treat a finite left endpoint.

Under a state cap $B$ and finite $\mathsf S$, partition into finitely
many pieces with $\mathsf S_J(u)\le\eta(B)$. On every compact
restriction of a piece, the absorption in \eqref{eq:background-graph-absorption}, with zero error, gives
\begin{equation}\label{eq:finite-size-graph-absorption}
 \|u\|_{\mathcal X(J)}
 \le CB+CB^{2\alpha}\mathsf S_J(u)^4\|u\|_{\mathcal X(J)}
 \le CB+\tfrac12\|u\|_{\mathcal X(J)}.
\end{equation}
Absorption in \eqref{eq:finite-size-graph-absorption} and exhaustion give a uniform graph bound, also on a left tail because
all entering traces satisfy the state cap. Thus every finite endpoint
extends. At an infinite endpoint it instead gives global $\mathsf W$
and an $L^2L^{6/5}$ gradient source with vanishing tails. The same
increment estimate \eqref{eq:interaction-coordinate-cauchy} yields
\[
 u_\pm=\lim_{t\to\pm\infty}S(-t)u(t)\quad\text{in }\X,
 \qquad
 \|u(t)-S(t)u_\pm\|_{\dot H^1}\longrightarrow0.
\]
If the bounds hold throughout the maximal interval, both finite
endpoints are excluded, so its domain is $\R$ and both limits concern
the same solution.
\end{proof}

For small data, the contraction estimates give uniform bounds on both half-lines.

\begin{corollary}[Quantitative small data]\label{local:small-data}
There is $\varepsilon_*>0$ such that if $\|f\|_{\dot H^1}\le\varepsilon_*$,
the actual maximal solution is global, scatters in both directions, and
\[
 \mathsf E_{\mathbb R}(u)+\mathsf W_{\mathbb R}(u)
 +\mathsf Z_{\mathbb R}(u)+\mathsf S_{\mathbb R}(u)
 \le C\|f\|_{\dot H^1},
\]
\[
 \mathsf E_{\mathbb R}(u-S(\cdot-s_0)f)
 +\mathsf W_{\mathbb R}(u-S(\cdot-s_0)f)
 +\mathsf Z_{\mathbb R}(u-S(\cdot-s_0)f)
 +\mathsf S_{\mathbb R}(u-S(\cdot-s_0)f)
 \le C\|f\|_{\dot H^1}^{2p-1}.
\]
In particular there is a positive energy level below which every datum on
the original stable branch scatters.
\end{corollary}
\begin{proof}
Put $a=\|f\|_{\dot H^1}$. For $a>0$, the contraction on each full
half-line, with state and critical radii $2Ca$, has
\begin{equation}\label{eq:small-data-source}
 \text{free size}\le Ca,\qquad
 \text{nonlinear coefficient}\le Ca^{4+2\alpha},\qquad
 \mathsf N(\nabla F(u))\le Ca^{5+2\alpha}=Ca^{2p-1}.
\end{equation}
Choose $\varepsilon_*$ so that $Ca^{4+2\alpha}$ gives an invariant
contraction ball. Countable compact exhaustion makes the half-line
graph complete; its integrals are the compatible finite Duhamel
realizations. The source bound \eqref{eq:small-data-source}, the endpoint estimate and Sobolev
prove both asserted norm bounds. Theorem~\ref{local:finite-s} gives
scattering; $a=0$ gives the zero solution. Finally, trapping gives
\[
 \|\nabla f\|_2^2\le kE(f),\qquad
 kE(f)\le\varepsilon_*^2\ \Longrightarrow\
 \|\nabla f\|_2\le\varepsilon_*.
\]
Thus a positive energy level on the stable branch scatters.
\end{proof}

We construct prescribed scattering states from finite terminal data
and then pass to a common limit on compact intervals. This keeps each
Duhamel integral within the solution space already established.

\begin{theorem}[One-sided wave operators]\label{local:waves}
For every $\phi\in\X$ there is an actual maximal trajectory $u_+$ whose
domain contains $[T,\infty)$ for some finite $T$, and
\[
 \|u_+(t)-S(t)\phi\|_{\mathrm{gr}}\longrightarrow0\quad(t\to+\infty).
\]
There is, separately, an actual maximal trajectory $u_-$ containing
$(-\infty,T']$, with the corresponding limit as $t\to-\infty$.
On the prescribed time direction each trajectory has a finite energy-state
cap and arbitrarily small critical and endpoint tails:
\[
 \mathsf S_{[A,\infty)}(u_+)+\mathsf W_{[A,\infty)}(u_+)
 +\mathsf Z_{[A,\infty)}(u_+)\longrightarrow0
 \quad(A\to+\infty),
\]
and analogously on $(-\infty,-A]$ for $u_-$.  Their energies on their
entire actual maximal domains are $\frac12A(\phi)$.  
\end{theorem}
\begin{proof}
Choose $B$ larger than twice the free state bound plus one, then
$\eta>0$ with $CB^{2\alpha}\eta^5\le\min(\eta/2,B/4)$ and
same-cap difference constant at most $1/2$.
The free critical tail tends to zero, so for some $T$ the backward
contraction gives, uniformly in $R>T$,
\begin{equation}\label{eq:terminal-wave-family}
 \begin{gathered}
 u_R(R)=S(R)\phi,\qquad
 u_R(t)=S(t)\phi-i\int_t^RS(t-s)F(u_R(s))\,ds,\\
 \mathsf E_{[T,R]}(u_R)\le B,\qquad
 \mathsf W_{[T,R]}(u_R)\le\eta,\qquad
 \|u_R-S(t)\phi\|_{\mathcal X([T,R])}\le\eta/2.
 \end{gathered}
\end{equation}
For any smaller $\eta_\delta\le\delta$, repeat the construction on
$[A,R]$ with $A=A(\delta)\ge\delta^{-1}$ sufficiently large. Compact uniqueness at
the common terminal value identifies it with $u_R$ from \eqref{eq:terminal-wave-family}.
Consequently this baseline family has uniformly small late tails.
For $R_n\ge R_m\ge A$, absorption of the common half-contraction gives
\begin{equation}\label{eq:terminal-wave-cauchy}
 \|u_{R_n}-u_{R_m}\|_{\mathcal X([T,R_m])}
 \le C\|u_{R_n}(R_m)-S(R_m)\phi\|_{\mathrm{gr}}
 \le C\delta.
\end{equation}
Take $R_n=T+n$. The Cauchy bound \eqref{eq:terminal-wave-cauchy}, completeness, and \eqref{local:difference-source} pass
the finite Duhamel formula to the compact-interval limits. These
limits are compatible restrictions of the same sequence. Their trace
at $T$ and centered existence give one maximal trajectory $u_+$.

Passing the improved bounds at $A(\delta)$ to the limit, exhausting
the tail and using the endpoint estimate gives
\begin{equation}\label{eq:wave-tail-bounds}
 \begin{aligned}
 \sup_{t\ge A}\|u_+(t)-S(t)\phi\|_{\mathrm{gr}}&\le C\delta,\\
 \mathsf W_{[A,\infty)}(u_+)+\mathsf S_{[A,\infty)}(u_+)&\le C\delta,\\
 \mathsf Z_{[A,\infty)}(u_+)
 &\le\mathsf Z_{[A,\infty)}(S(t)\phi)+CB^{2\alpha}\delta^5.
 \end{aligned}
\end{equation}
Since $A(\delta)\ge\delta^{-1}$, letting $\delta\downarrow0$ in \eqref{eq:wave-tail-bounds} gives
the full real-time asymptote and all three vanishing tails.

For $u_-$, use finite initial data $S(R)\phi$ at $R_n=T'-n$ and
the forward formula on $[R_n,T']$. The same construction with its
forward constants gives the past asymptote and tails directly.

Dispersion and $L^2$ unitarity give $\|S(t)\psi\|_6\to0$ for
Schwartz $\psi$; homogeneous density and Sobolev extend this to
$\phi\in\X$. In the prescribed direction,
\begin{equation}\label{eq:wave-asymptotic-energy}
 A(u_\pm(t))\to A(\phi),\qquad
 D(u_\pm(t))\le C\|u_\pm(t)\|_6^{2p}\to0.
\end{equation}
Energy conservation and \eqref{eq:wave-asymptotic-energy} give $E(u_\pm)=A(\phi)/2$ on their
respective maximal domains.
\end{proof}

These one-sided solutions supply the profiles whose normalized times diverge.

\section{Profile decomposition and nonlinear reconstruction}
\label{prof:section}

The linear profile decomposition is a standard input. Its nonlinear
use here requires two Hartree estimates: decoupling of the nonlocal
potential energy and vanishing of the differentiated interactions
between distinct profiles. Together with stability, these estimates
reconstruct a solution whenever every nonlinear profile scatters.

Throughout this section $P(f)=\|\nabla f\|_2$ and the gradient has its
Euclidean $\C^3$ norm. In addition to the norms of the local theory, put
\[
 \begin{aligned}
 V_I(u)&=\|u\|_{L^\infty_t(I;L_x^6)},\\
 \mathfrak X_I(u)&=\EE_I(u)+\SSS_I(u)+\WW_I(u),\\
 \mathcal N_I(e)&=\|\nabla e\|_{L_t^2(I;L_x^{6/5})}.
 \end{aligned}
\]
Here $\WW_I(u)$ and $\mathcal N_I(e)$ use the weak gradients of
$u$ and $e$. By Lemma~\ref{local:critical-sobolev}, $\mathfrak X_I$
is equivalent to the norm of $\mathcal X(I)$ on that solution space.
For a field $h$, rather than a scalar source, write
$\NN_I(h)=\|h\|_{L^2_t(I;L^{6/5}_x)}$.
All unrestricted time norms in this section are on $\R$.
The value and gradient of a continuous $\X$ path have compatible jointly
measurable representatives: approximate the path by countably valued
measurable maps on compact time intervals in $L^6\times(L^2)^3$, take
compatible limits, and exhaust the line. Spatial distributional
identities identify their gradients. Tonelli shows that two such
representatives agree for product-almost every $(t,x)$.

\subsection{Linear profile decomposition}

Define the critical spatial isometry and its full linear frame by
\[
 \begin{aligned}
 T_{r,y}f(x)&=r^{-1/2}f\left(\frac{x-y}{r}\right),\\
 A_{r,y,s}&=T_{r,y}S(s),&
 A_{r,y,s}^{-1}&=S(-s)T_{r,y}^{-1}.
 \end{aligned}
\]
They preserve $P$, and
\begin{equation}\label{prof:frames}
 S(t)T_{r,y}=T_{r,y}S(t/r^2),\qquad
 S(t)A_{r,y,s}f=T_{r,y}S(s+t/r^2)f.
\end{equation}
Thus $s$ is normalized time; the physical time center is $-r^2s$.

We use the nonradial energy-critical decomposition of
\cite[Theorem~4.27]{KillipVisan}. Its parameters are scale, translation
and time; its small remainder is measured in the free $L^{10}_{t,x}$ norm.
The statement below fixes these conventions explicitly.

\begin{proposition}[Linear profiles in the homogeneous energy space]
\label{prof:linear}
If $P(f_n)\le B$, there are one strictly increasing subsequence, renamed
$n$, a count $M\in\N\cup\{\infty\}$, profiles $\phi_j\in\X$,
positive $r_{j,n}$, centers $y_{j,n}$ and times $s_{j,n}$, with the
following properties. Active labels are $0\le j<M$; inactive profiles
are zero and can be assigned parameters $(1,0,0)$. Set
\[
 b_{j,n}=A_{r_{j,n},y_{j,n},s_{j,n}}\phi_j,
 \qquad R_n^J=f_n-\sum_{j<J}b_{j,n}.
\]
On this subsequence:
\begin{enumerate}
\item For each $\varepsilon>0$ there are finite $J,N$ such that
\begin{equation}\label{prof:selected-remainder}
 (M=\infty\ \hbox{or}\ J=M),\qquad
 \SSS(S(t)R_n^J)\le\varepsilon\quad(n\ge N).
\end{equation}
\item For every fixed admissible $J\le M$,
\[
 P(f_n)^2-\sum_{j<J}P(\phi_j)^2-P(R_n^J)^2\longrightarrow0.
\]
\item For every such $J$ and every $j<J$,
$A_{r_{j,n},y_{j,n},s_{j,n}}^{-1}R_n^J\rightharpoonup0$ in the
homogeneous gradient Hilbert space. 
\item For distinct active $j,k$,
\begin{equation}\label{prof:orthogonality}
 \frac{r_{j,n}}{r_{k,n}}+\frac{r_{k,n}}{r_{j,n}}
 +\frac{|y_{j,n}-y_{k,n}|^2}{r_{j,n}r_{k,n}}
 +\frac{|r_{j,n}^2s_{j,n}-r_{k,n}^2s_{k,n}|}
 {r_{j,n}r_{k,n}}\longrightarrow\infty.
\end{equation}
\end{enumerate}
The free remainder norm is $L^{10}_{t,x}$. If $M<\infty$ then
$R_n^J=R_n^M$ for $J\ge M$, including $M=0$.
\end{proposition}
\begin{proof}
Apply \eqref{input:eq-linear-profiles} from
Lemma~\ref{std:analytic-interfaces}. We only need to identify the
homogeneous space, Fourier convention and frame parameters with those
of the cited decomposition.
The gradient map is an isometry from $(\X,P)$ onto a closed subspace
of $(L^2)^3$: a gradient-Cauchy sequence is Cauchy in $L^6$ by
Sobolev, and testing identifies the limiting weak derivatives.
Lemma~\ref{std:density} identifies this space with the homogeneous
completion of $C_c^\infty$.

For the unitary angular transform
$\mathcal F_{\rm un}f(\theta)=(2\pi)^{-3/2}
\int e^{-ix\cdot\theta}f(x)\dd x$, Plancherel on smooth data gives
\begin{equation}\label{eq:profile-hilbert-isometry}
 P(f)^2=\int|\theta|^2|\mathcal F_{\rm un}f(\theta)|^2\dd\theta.
\end{equation}
Approximate $|\theta|\mathcal F_{\rm un}f$ by compact annular smooth
functions, divide by $|\theta|$, and invert. These Schwartz
approximants converge in $P$; their spatial cutoffs converge because
the gradient tails vanish and
$\|g\nabla\chi_R\|_2\le CR^{-1}\|g\|_2\to0$.
A diagonal identifies the completion isometrically through \eqref{eq:profile-hilbert-isometry}. Only the approximants
require scalar $L^2$; the $L^6$ realization excludes nonzero
polynomials and removes the Fourier ambiguity.

Under $\theta=2\pi\xi$,
$e^{-it|\theta|^2}=e^{-4\pi^2it|\xi|^2}$, so density identifies
the free groups with unchanged time. The published frame is
$T_{r,y}S(s)$; \eqref{prof:frames} gives its physical-time separation
\eqref{prof:orthogonality}. The extraction is nonradial and has
no modulation parameter. Slice equality and Tonelli identify the
actual free representatives and their mixed/product $L^{10}$ norms,
including infinite values.

Keep the published subsequence, reindex from zero, and pad finite
families by zero profiles. For $M=\infty$, choose $J$ with
$\limsup_n\SSS(S(t)R_n^J)<\varepsilon/2$, then $N$ with
$\SSS(S(t)R_n^J)\le\varepsilon$ for all $n\ge N$.
For $M<\infty$, the terminal remainder has limsup zero, so use
$J=M$, also when $M=0$; zero padding handles $J>M$.
The kinetic and inverse-frame relations \eqref{input:eq-linear-kinetic}
transfer by gradient isometry. In an inverse frame, the continuous real-linear tests
$h\mapsto\re\langle h_i,g\rangle_{L^2}$ transfer weak convergence;
testing also $ig$ gives its complex form. The frame orthogonality \eqref{input:eq-linear-orthogonality}
yields \eqref{prof:orthogonality}.
\end{proof}

A further common subsequence fixes the time regime of every profile.

\begin{lemma}[A common subsequence and nonlinear profiles]
\label{prof:countable-family}
On one further common subsequence, every $s_{j,n}$ converges in
$[-\infty,\infty]$ to $s_j^*$. There are corresponding maximal
solutions $U_j$ on $I_j$ such that, for each fixed $j$,
\begin{equation}\label{prof:matching}
 s_{j,n}\in I_j\ \hbox{eventually},\qquad
 P\bigl(U_j(s_{j,n})-S(s_{j,n})\phi_j\bigr)\longrightarrow0.
\end{equation}
Both conclusions hold simultaneously for every fixed finite family.
\end{lemma}
\begin{proof}
Successive extraction in $[-\infty,\infty]$ and an increasing diagonal
give one map $\sigma$ along which every time sequence has a limit.
Compose all sequence fields with $\sigma$. The profiles and count
are unchanged, and $\sigma(n)\ge n$ preserves the selected-depth
bound and the other three conclusions.

Choose one canonical maximal $U_j$ for each label by
\begin{equation}\label{eq:nonlinear-profile-choices}
 \begin{cases}
 U_j(s_j^*)=S(s_j^*)\phi_j,&s_{j,n}\to s_j^*\in\R,\\
 P(U_j(t)-S(t)\phi_j)\to0\quad(t\to\pm\infty),
       &s_{j,n}\to\pm\infty.
 \end{cases}
\end{equation}
The finite case uses Theorem~\ref{local:flow}; the infinite cases
use the corresponding wave operator in Theorem~\ref{local:waves}.
In \eqref{eq:nonlinear-profile-choices}, openness and continuity treat the finite
case, while the prescribed half-line asymptotes treat the infinite cases;
this gives \eqref{prof:matching}.
For finite $F$, choose each membership threshold and error
$<\varepsilon/(1+|F|)$, then take their maximum. This gives simultaneous
membership and summed error $<\varepsilon$, including $F=\varnothing$.
\end{proof}

Kinetic orthogonality alone does not split the Hartree energy, because
the potential couples spatially separated densities. We prove the
nonlocal split at each fixed finite depth and retain the energy of the
time-evolved atom $b_{j,n}$ until its time regime is taken into account.

\subsection{Fixed-depth Hartree energy splitting}

\begin{lemma}[Hartree energy decoupling]
\label{prof:energy-splitting}
For the family just constructed and each fixed admissible $J$, there
is a further strictly increasing subsequence, which may depend on $J$,
on which
\begin{equation}\label{prof:finite-energy}
 E(f_n)-E(R_n^J)-\sum_{j<J}E(b_{j,n})\longrightarrow0.
\end{equation}
Moreover $\|S(s)\phi\|_6\to0$ and $D(S(s)\phi)\to0$ as
$|s|\to\infty$. In a finite-time regime the atom energy tends to
$E(S(s_j^*)\phi_j)$; in a divergent regime it tends to
$P(\phi_j)^2/2$.
\end{lemma}
\begin{proof}
To establish \eqref{prof:finite-energy}, let $g_n$ be bounded in $L^6$, with $g_n\to0$ a.e., and let $f\in L^6$.
For $a=6/p>1$,
\[
 \bigl||z+w|^p-|z|^p\bigr|
 \le\epsilon|z|^p+C_\epsilon|w|^p.
\]
The positive part of
$\bigl||g_n+f|^p-|g_n|^p-|f|^p\bigr|-\epsilon|g_n|^p$
is dominated by $(C_\epsilon+1)|f|^p$ and tends to zero a.e.
Dominated convergence in $L^a$, then $\epsilon\downarrow0$, gives
\begin{equation}\label{eq:nonlocal-brezis-lieb-density}
 d_n:=|g_n+f|^p-|g_n|^p-|f|^p\longrightarrow0\text{ in }L^a,
 \qquad |g_n|^p\rightharpoonup0\text{ in }L^a.
\end{equation}
For the weak limit, use uniform integrability on finite-measure sets
and approximation of $L^{a'}$ tests. HLS makes
$B(v,w)=\int(I_\alpha*v)w$ continuous on $L^a\times L^a$; hence \eqref{eq:nonlocal-brezis-lieb-density} gives
\begin{equation}\label{eq:nonlocal-potential-splitting}
 D(g_n+f)-D(g_n)-D(f)
 =2B(|g_n|^p,|f|^p)+o(1)\longrightarrow0.
\end{equation}
If also $\nabla g_n\rightharpoonup0$, then
\begin{equation}\label{eq:profile-kinetic-cross}
 P(g_n+f)^2-P(g_n)^2-P(f)^2
 =2\re\int\nabla g_n\cdot\overline{\nabla f}\longrightarrow0.
\end{equation}
Combining \eqref{eq:nonlocal-potential-splitting} and \eqref{eq:profile-kinetic-cross} proves energy splitting. A bounded weakly convergent sequence
in $\X$ has the same a.e. limit on a subsequence: local $H^1$ bounds,
Lemma~\ref{std:rellich}, a diagonal on integer balls and a summably
Cauchy refinement give local $L^2$ and a.e. convergence; Sobolev
identifies the limit.

Let $A_n=T_{r_n,y_n}S(s_n)$ and $A_n^{-1}v_n\rightharpoonup\phi$.
If $s_n\to s^*$, then
\[
 T_{r_n,y_n}^{-1}v_n\rightharpoonup S(s^*)\phi
\]
by strong continuity on gradient tests. Apply
\eqref{eq:nonlocal-potential-splitting} and \eqref{eq:profile-kinetic-cross}, replace $S(s^*)\phi$ by $S(s_n)\phi$ using bounded-set
continuity of $E$, and return by affine invariance.
If $|s_n|\to\infty$, compact smooth approximation and
\[
 \|S(s)\psi\|_6\le C|s|^{-1}\|\psi\|_{6/5},\qquad
 \|S(s)(\phi-\psi)\|_6\le C P(\phi-\psi)
\]
give $\|S(s)\phi\|_6\to0$ and $D(S(s)\phi)\to0$.
Continuity of $D$ on bounded $L^6$ sets gives the nonlinear split
without further extraction; inverse-frame weak convergence gives
the kinetic split.

At stage $j$,
\[
 R_n^j=R_n^{j+1}+b_{j,n},\qquad
 A_{j,n}^{-1}R_n^j\rightharpoonup\phi_j,\qquad
 P(R_n^j)\le B+\sum_{k<j}P(\phi_k).
\]
For fixed finite $J$, compose the extractions for $0\le j<J$ and sum
\[
 E(f_n)-E(R_n^J)-\sum_{j<J}E(b_{j,n})
 =\sum_{j<J}\bigl[E(R_n^j)-E(R_n^{j+1})-E(b_{j,n})\bigr]
 \longrightarrow0.
\]
The empty sum handles $J=0$; the subsequence may depend on $J$.
\end{proof}

Reconstruction also requires the differentiated interactions between distinct profiles to vanish.

\subsection{Compact approximation and cross interactions}

Put $q=p-2=1+\alpha$, $\rho(z)=|z|^p$ and
$\ell(z)=|z|^qz$. Derivatives in $z\in\C\simeq\R^2$ are real:
\[
 \rho'(z)h=p|z|^{p-2}\re(\bar zh),\qquad
 \ell'(z)h=|z|^qh+q|z|^{q-2}\re(\bar zh)z.
\]
The second formula at zero means its continuous value zero. Write
$I(g)=I_\alpha*g$ and
\[
 \mathcal C(u,h)=I(\rho'(u)h)\ell(u)+I(\rho(u))\ell'(u)h.
\]
For slices in the spaces of the local theory, the weak chain rule gives
$\partial_kF(u)=\mathcal C(u,\partial_ku)$.
We also need this formula with different profiles in the convolved and
outside factors.
For independent directions put $W_d(h)=\|h\|_{L^{10}_tL^{30/13}_x}$.

\begin{lemma}[Compatible spacetime approximation]
\label{prof:density-factors}
Let $u\in C(\R;\X)$ have finite $\EE(u),\SSS(u),\WW(u)$.
There are jointly continuous, compactly supported fields
$f_m:\R^{1+3}\to\C$ and $G_m:\R^{1+3}\to\C^3$ such that
$\nabla_x f_m=G_m$ weakly on every time slice,
$\sup_m\|G_m\|_{L^\infty_tL^2_x}<\infty$,
$\SSS(f_m-u)\to0$ and $W_d(G_m-\nabla u)\to0$. Each mixed-label expression
$I(\rho'(u)h)\ell(v)$ and $I(\rho(u))\ell'(v)k$ is finite in
$\NN$ on bounded $V,\SSS,W_d$ sets and continuous in the value
$\SSS$ and direction $W_d$ differences, uniformly on such sets.
The modulus is $Cd^\vartheta$ for $d\le1$, where
$\vartheta=(3-\alpha)/3$.
\end{lemma}
\begin{proof}
Choose smooth cutoffs $\eta_T(t)=\eta(t/T)$ and
$\chi_R(x)=\chi(x/R)$, equal to one on the unit interval and ball
and supported in their doubles. The weak product rule gives
\[
 \begin{aligned}
 \nabla(\eta_T\chi_Ru)&=\eta_T\chi_R\nabla u+\eta_Tu\nabla\chi_R,\\
 \EE(\eta_T\chi_Ru)&\le\EE(u)+A_\chi V(u),
       \qquad A_\chi=\|\nabla\chi\|_3,\\
 W_d(\eta_Tu\nabla\chi_R)
 &\le A_\chi\SSS(1_{\{R\le|x|\le2R\}}u)\longrightarrow0.
 \end{aligned}
\]
Dominated convergence handles the other $\SSS,\WW$ cutoff errors.
Set $T=R=m$ and, with a positive compact smooth mollifier, put
\[
 f_{m,\varepsilon}=\phi_\varepsilon*_x(\eta_m\chi_m u),\qquad
 G_{m,\varepsilon}=\phi_\varepsilon*_x\nabla_x(\eta_m\chi_m u).
\]
Then $\nabla_xf_{m,\varepsilon}=G_{m,\varepsilon}$ on every slice.
Continuity in $\X$ and convolution give jointly continuous compact
fields; Minkowski preserves the energy cap. Translation continuity
in the finite mixed norms permits $\varepsilon_m$ with both smoothing
errors at most $1/m$.

Use the following ordered spaces $L^{t_i}_tL^{x_i}_x$, $i=0,1,2,3$:
\[
\begin{array}{c|cc|cc}
 &a_i&b_i&t_i&x_i\\\hline
 X_0&2\alpha&3-\alpha&10/(3-\alpha)&30/(9+7\alpha)\\
 X_1&\alpha/3&2+2\alpha/3&10/(3+2\alpha/3)&90/(57+11\alpha)\\
 Y_0&5\alpha/3&2-2\alpha/3&10/(2-2\alpha/3)&90/(18+19\alpha)\\
 Y_1&0&1+\alpha&10/(2+\alpha)&30/(16+3\alpha).
\end{array}
\]
Rows $X_1,Y_1$ contain one $W_d$ direction. The input exponents
$x_0,x_1$ are exactly the two strict HLS inputs verified in
Lemma~\ref{local:nonlinear}.
Their output exponents $r_0=10/(3-\alpha)$ and
$r_1=90/[19(3-\alpha)]$ satisfy
\begin{equation}\label{eq:mixed-factor-exponents}
\begin{gathered}
 \frac1{r_i}=\frac1{x_i}-\frac\alpha3\quad(i=0,1),\\
 \frac1{r_0}+\frac1{x_3}=\frac1{r_1}+\frac1{x_2}=\frac56,
 \qquad\frac1{t_0}+\frac1{t_3}=\frac1{t_1}+\frac1{t_2}=\frac12.
\end{gathered}
\end{equation}
H\"older gives the factor bounds
\[
\begin{array}{ll}
 \|\rho(u)\|_{X_0}\le B_0(u):=V(u)^{a_0}\SSS(u)^{b_0},&
 \|\rho'(u)h\|_{X_1}\le B_1(u,h):=pV(u)^{a_1}\SSS(u)^{b_1}W_d(h),\\
 \|\ell(v)\|_{Y_0}\le C_0(v):=V(v)^{a_2}\SSS(v)^{b_2},&
 \|\ell'(v)k\|_{Y_1}\le C_1(v,k):=(q+1)\SSS(v)^{b_3}W_d(k).
\end{array}
\]
With the exponent identities \eqref{eq:mixed-factor-exponents}, the HLS constants $H_0,H_1$ give
\begin{equation}\label{eq:mixed-slot-hls}
 \NN(I(f)g)\le H_1\|f\|_{X_1}\|g\|_{Y_0},\qquad
 \NN(I(f)g)\le H_0\|f\|_{X_0}\|g\|_{Y_1}.
\end{equation}
HLS gives a.e. absolute convergence; Tonelli gives joint measurability.

Put $e=V(u)+V(U)$, $s=\SSS(u)+\SSS(U)$, $d=\SSS(u-U)$ and
$\theta_0=\theta_2=\vartheta=(3-\alpha)/3$.
In the four factor spaces the differences are bounded by
\begin{equation}\label{eq:mixed-factor-differences}
\begin{split}
 \Delta_0&=p e^{a_0}s^{b_0-\theta_0}d^{\theta_0},\\
 \Delta_1&=p e^{a_1}s^{b_1}W_d(h-H)
       +p(p-1)e^{a_1}s^{b_1-1}dW_d(H),\\
 \Lambda_0&=(q+1)e^{a_2}s^{b_2-\theta_2}d^{\theta_2},\\
 \Lambda_1&=(q+1)s^{b_3}W_d(k-K)
       +q(q+5)s^{b_3-1}dW_d(K).
\end{split}
\end{equation}
Indeed, segment integration for a factor $A$ of degree $d_0$ gives
\[
 |A(z)-A(w)|\le C(|z|+|w|)^{d_0-\theta}|z-w|^\theta
 \quad(0<\theta\le1),
\]
and for the derivative actions,
\[
 \begin{aligned}
 \|\rho'(z)-\rho'(w)\|
 &\le p(p-1)(|z|+|w|)^{p-2}|z-w|,\\
 \|\ell'(z)-\ell'(w)\|
 &\le q(q+5)(|z|+|w|)^{q-1}|z-w|.
 \end{aligned}
\]
Allocate $a_i$ powers to $V$, $b_i$ to $\SSS$ and the direction to
$W_d$; residual powers are nonnegative, with zero powers omitted.
Combining \eqref{eq:mixed-slot-hls} and \eqref{eq:mixed-factor-differences} with convolution linearity gives
\begin{equation}\label{eq:independent-slot-continuity}
\begin{split}
 \NN\bigl(I(\rho'(u)h)\ell(v)-I(\rho'(U)H)\ell(V)\bigr)
 &\le H_1\{\Delta_1(u,h;U,H)C_0(v)+B_1(U,H)\Lambda_0(v,V)\},\\
 \NN\bigl(I\rho(u)\ell'(v)k-I\rho(U)\ell'(V)K\bigr)
 &\le H_0\{\Delta_0(u,U)C_1(v,k)+B_0(U)\Lambda_1(v,k;V,K)\}.
\end{split}
\end{equation}
Since $0<\vartheta\le\min(1,b_0,b_2)$, writing $\mathfrak d$ for
the sum of the value and direction differences in \eqref{eq:independent-slot-continuity}
gives bounded-set modulus $C\mathfrak d^\vartheta$ for $\mathfrak d\le1$; no small
$V$ difference is needed.

The $X_0$ bound and $\ell(v)\in L^{10/(2+\alpha)}_{t,x}$ also give
$I\rho(u)\ell(v)\in L^2_tL^2_x$ and convergence there under the
compact approximations. For the mollified values, differentiate the
$C^1$ density in $\int I_\alpha(z)\rho(u(x-z))\dd z$.
Local integrability of $I_\alpha$ and compact support justify this;
the continuity bounds \eqref{eq:independent-slot-continuity} pass to the limit and give
\begin{equation}\label{eq:mixed-label-weak-gradient}
 \partial_k[I\rho(u)\ell(v)]
 =I(\rho'(u)\partial_ku)\ell(v)+I\rho(u)\ell'(v)\partial_kv.
\end{equation}
The identity \eqref{eq:mixed-label-weak-gradient} uses no kernel derivative, including when $\alpha\le1$.
\end{proof}

Compact approximation reduces the interaction estimates to bounded
profiles with compact support. Disjoint spatial supports do not make a
Hartree interaction vanish: the Riesz kernel still connects them.
We therefore use scale separation or decay of the kernel between the
supports.

\begin{lemma}[Nonlocal cross interactions]\label{prof:pure-cross}
Let $U,V\in C(\R;\X)$ have finite $\mathfrak X$, with their actual gradients.
Frame them by two parameter sequences satisfying
\eqref{prof:orthogonality}. Each of the two differentiated
mixed-label terms, in each source--recipient order, tends to zero
separately in $\NN$.
\end{lemma}
\begin{proof}
There are two cases: widely separated scales yield a positive scaling
power, while comparable scales force separation in space or time.
Lemma~\ref{prof:density-factors} and frame invariance reduce all four
terms uniformly to compatible compact seeds $f,g,h,k$, bounded by
$M_0$ and supported in $[-T,T]\times B_R$.
Normalize the larger scale to one and set
\[
 \begin{aligned}
 f_\varepsilon(t,x)&=\varepsilon^{-1/2}
 f((t-\tau)/\varepsilon^2,(x-z)/\varepsilon),\\
 h_\varepsilon(t,x)&=\varepsilon^{-3/2}
 h((t-\tau)/\varepsilon^2,(x-z)/\varepsilon),
 \qquad0<\varepsilon\le1.
 \end{aligned}
\]
The two majorants are
\[
 \mathcal A(u;v,k)=I(|u|^p)|v|^q|k|,\qquad
 \mathcal B(u,h;v)=I(|u|^{p-1}|h|)|v|^{q+1}.
\]
Put $\delta=(3-\alpha)/4>0$. HLS at
\[
 \left(\frac{12}{9+\alpha},\frac4{3-\alpha}\right),\qquad
 \left(\frac4{1+\alpha},\frac{12}{3-\alpha}\right)
\]
with recipient exponents $12/(1+3\alpha)$ and $12/(7+\alpha)$ gives
\begin{equation}\label{eq:cross-separated-scales}
\begin{array}{ll}
 \NN(\mathcal A(f_\varepsilon;g,k))\le C\varepsilon^{\delta+1},&
 \NN(\mathcal B(f_\varepsilon,h_\varepsilon;g))\le C\varepsilon^\delta,\\
 \NN(\mathcal A(g;f_\varepsilon,h_\varepsilon))\le C\varepsilon^\delta,&
 \NN(\mathcal B(g,k;f_\varepsilon))\le C\varepsilon^{\delta+1}.
\end{array}
\end{equation}
The input gaps from $1$ and $3/\alpha$ are, respectively,
\[
 \left(\frac{3-\alpha}{9+\alpha},\frac{3-\alpha}{1+\alpha}\right),
 \qquad
 \left(\frac{9(3-\alpha)}{\alpha(9+\alpha)},
       \frac{3-\alpha}{\alpha(1+\alpha)}\right)>0.
\]
The first density scales as
$\varepsilon^{3(9+\alpha)/12-p/2}=\varepsilon^\delta$;
a direction lowers the power by one. Reverse recipients give the
opposite pair of powers. The time intersection has length
$\le2T\varepsilon^2$, supplying $\varepsilon$ in $L^2_t$.

For $|\tau|>2T$ all four terms vanish. For $|z|\ge4R$, the kernel
between supports is bounded by
$c_\alpha2^{3-\alpha}|z|^{\alpha-3}$, giving, in the same order,
\begin{equation}\label{eq:cross-separated-centers}
 C|z|^{\alpha-3}
 \bigl(\varepsilon^{(5-\alpha)/2},\varepsilon^{(3-\alpha)/2},
       \varepsilon^{(3-\alpha)/2},\varepsilon^{(5-\alpha)/2}\bigr).
\end{equation}
All powers are positive, including the directional-source cases.
The common-frame weight $7/2$ preserves $\NN$, and orthogonality is
\begin{equation}\label{eq:cross-frame-escape}
 \varepsilon+\varepsilon^{-1}
       +( |z|^2+|\tau| )/\varepsilon\longrightarrow\infty.
\end{equation}
By \eqref{eq:cross-separated-scales} and \eqref{eq:cross-separated-centers},
choose a small $\eta$, then a large $D$. The errors are small for
$\varepsilon<\eta$ or $|z|>D$, and vanish for $|\tau|>2T$.
On the remaining configurations,
\begin{equation}\label{eq:cross-bounded-configuration}
 \varepsilon+\varepsilon^{-1}+(|z|^2+|\tau|)/\varepsilon
 \le1+\eta^{-1}(1+D^2+2T),
\end{equation}
so \eqref{eq:cross-bounded-configuration} contradicts \eqref{eq:cross-frame-escape}
for large $n$, in either scale order.
All compact integrals are absolute since $\alpha>0$, with derivatives
on the smooth densities. Choosing the joint approximation before $n$
transfers each separate limit to $U,V$ along the original sequence.
\end{proof}

We now sum the interactions for a fixed finite family.

\begin{proposition}[The finite-profile nonlinear defect]
\label{prof:finite-defect}
For finitely many global profiles $U_j\in C(\R;\X)$ with finite $\mathfrak X$ in
pairwise orthogonal frames, set
\[
 u_{j,n}(t,x)=r_{j,n}^{-1/2}
 U_j(s_{j,n}+t/r_{j,n}^2,(x-y_{j,n})/r_{j,n}),\qquad
 Z_n=\sum_{j<J}u_{j,n}.
\]
Then the actual scalar $e_n=F(Z_n)-\sum_{j<J}F(u_{j,n})$ satisfies
$\mathcal N(e_n)\to0$. It belongs to
$L^1(I;L^{6/5})$ on every compact interval, and its actual derivative
in $L^2(I;L^{6/5})$ is compatible with that same scalar source.
\end{proposition}
\begin{proof}
Put
\[
 P_2(f,g)=\||f||g|\|_{L^5_tL^5_x},\qquad
 Q_2(f,h)=\||f||h|\|_{L^5_tL^{15/8}_x}.
\]
For compact seeds at scales $1,\varepsilon\le1$, their support volume,
time length and amplitudes give
\[
 P_2\le C\varepsilon^{1/2},\qquad
 Q_2\le C\varepsilon^{3/2}\ \text{or }C\varepsilon^{1/2}
\]
in the two direction placements. Separated comparable frames have
disjoint supports. The bounded-configuration argument in
Lemma~\ref{prof:pure-cross} and
\[
 P_2(f,g)\le\SSS(f)\SSS(g),\qquad
 Q_2(f,h)\le\SSS(f)W_d(h)
\]
pass compact approximation to the original sequence and give
\begin{equation}\label{eq:profile-product-vanishing}
 P_2(u_{i,n},u_{j,n})\to0,\qquad
 Q_2(u_{i,n},\partial_ku_{j,n})\to0\quad(i\ne j),
\end{equation}
with both direction placements.

For $A=\rho,\ell$ of degree $d_A=p,q+1$, integration of
$[DA(v+tw)-DA(tw)]w$ gives
\[
 |A(v+w)-A(v)-A(w)|\le K_A(|v|+|w|)^{d_A-2}|v||w|,
\]
where $K_\rho=p(p-1)2^{p-2}$ and $K_\ell=q(q+5)2^{q-1}$.
Writing $R=\sum|z_j|$, telescoping and real linearity in the direction
give
\begin{equation}\label{eq:finite-power-telescoping}
 \begin{aligned}
 |A(\sum z_j)-\sum A(z_j)|
 &\le K_A R^{d_A-2\delta}\sum_{i<j}(|z_i||z_j|)^\delta,
       &&0<\delta\le1,\\
 |DA(\sum z_j)[\sum h_j]-\sum DA(z_j)[h_j]|
 &\le K_A R^{d_A-2}\sum_{i\ne j}|z_i||h_j|.
 \end{aligned}
\end{equation}
Both vanish at $R=0$.

Suppress $n$, fix a coordinate, and set
$h_j=\partial_ku_j$, $H=\sum h_j$, $Z=\sum u_j$ and
\[
 \begin{aligned}
 R_0&=\rho(Z)-\sum\rho(u_j),&
 R_1&=\rho'(Z)H-\sum\rho'(u_j)h_j,\\
 T_0&=\ell(Z)-\sum\ell(u_j),&
 T_1&=\ell'(Z)H-\sum\ell'(u_j)h_j.
 \end{aligned}
\]
With $V_\Sigma=\sum V(u_j)$, $S_\Sigma=\sum\SSS(u_j)$,
$\delta_0=\min(1,(3-\alpha)/2)$ and $\delta_2=1-\alpha/3$,
applying H\"older to \eqref{eq:finite-power-telescoping} in the factor spaces gives
\begin{equation}\label{eq:finite-factor-errors}
\begin{split}
 \|R_0\|_{X_0}&\le K_\rho V_\Sigma^{a_0}
 S_\Sigma^{b_0-2\delta_0}\sum_{i<j}P_2(u_i,u_j)^{\delta_0},\\
 \|T_0\|_{Y_0}&\le K_\ell V_\Sigma^{a_2}
 S_\Sigma^{b_2-2\delta_2}\sum_{i<j}P_2(u_i,u_j)^{\delta_2},\\
 \|R_1\|_{X_1}&\le K_\rho V_\Sigma^{a_1}
 S_\Sigma^{b_1-1}\sum_{i\ne j}Q_2(u_i,h_j),\\
 \|T_1\|_{Y_1}&\le K_\ell S_\Sigma^{b_3-1}
 \sum_{i\ne j}Q_2(u_i,h_j).
\end{split}
\end{equation}
All four norms in \eqref{eq:finite-factor-errors} tend to zero by
\eqref{eq:profile-product-vanishing}, since
\[
 b_0-2\delta_0\ge0,\quad b_2-2\delta_2=0,\quad
 b_1-1=1+2\alpha/3,\quad b_3-1=\alpha.
\]
The $Q_2$ reciprocals $1/5,8/15$ replace one $\SSS$ value and one
$W_d$ direction. On the common absolute-integrability set,
\begin{equation}\label{eq:finite-defect-decomposition}\begin{aligned}
 \mathcal C(Z,H)-\sum_i\mathcal C(u_i,h_i)
 ={}&I(R_1)\ell(Z)+\sum_i I(\rho'(u_i)h_i)T_0
       +\sum_{i\ne j}I(\rho'(u_i)h_i)\ell(u_j)\\
 &+I(R_0)\ell'(Z)H+\sum_i I(\rho(u_i))T_1
       +\sum_{i\ne j}I(\rho(u_i))\ell'(u_j)h_j.
\end{aligned}\end{equation}
In \eqref{eq:finite-defect-decomposition}, HLS and \eqref{eq:finite-factor-errors}
control the four local terms; Lemma~\ref{prof:pure-cross} controls each ordered-pair sum. Sum the three coordinates; $J=0,1$
give zero directly.

Equation~\eqref{eq:mixed-label-weak-gradient} identifies the vector in
\eqref{eq:finite-defect-decomposition} as the weak gradient of the same scalar $e_n$. Moreover,
\[
 \|F(v)\|_{6/5}\le C\|v\|_6^{2p-1},\qquad
 e_n\in L^\infty_tL^{6/5}_x\subset L^1(I;L^{6/5})
\]
for compact $I$. Joint measurability gives the Bochner fields;
spatial cutoff and $L^{6/5}$--$L^6$ H\"older extend the slice identities
to Schwartz tests. These are the compatible sources in
Lemma~\ref{local:mild}, for this fixed finite family.
\end{proof}

The free remainder is small in $L^{10}_{t,x}$, but its global gradient
Strichartz norm need not be small. Local smoothing supplies strong
gradient convergence on bounded cylinders; compact approximation then
allows multiplication by each fixed profile.

\subsection{The free remainder gradient and the small nonlinear tail}

\begin{lemma}[Local compactness of the free gradient]
\label{prof:localization}
If $v_n\rightharpoonup0$ in the gradient Hilbert space and
$\sup_nP(v_n)<\infty$, then for every $0<T,R<\infty$
\[
 \|1_{[-T,T]\times B_R}\nabla S(t)v_n\|_{L^2_{t,x}}\longrightarrow0.
\]
For the coefficient conclusions, let $U\in C(\R;\X)$ have finite
$\mathsf E_\R(U)$, $\SSS_\R(U)$ and $\mathsf W_\R(U)$. For each $k\in\{1,2,3\}$ set
$h_{n,k}=\partial_k S(t)v_n$. Then $Q_2(U,h_{n,k})\to0$, and both
$I(\rho'(U)h_{n,k})\ell(U)$ and $I\rho(U)\ell'(U)h_{n,k}$ tend to zero
separately in $\NN$.
\end{lemma}
\begin{proof}
Partition nonzero frequencies into six disjoint signed
dominant-coordinate cones, breaking ties by coordinate order.
For $h\in L^2$ supported in cone $k$ with $N<|\xi|$ and a finite
upper cutoff, use
\[
 |\xi_k|\ge|\xi|/\sqrt3,\qquad
 \omega=-2\pi(\xi_k^2+|\xi'|^2),\qquad
 |d\omega/d\xi_k|=4\pi|\xi_k|.
\]
The fixed sign makes this change injective. Plancherel in $(t,x')$ gives
\[
 \int_{\R\times\R^2}|S(t)h(x_k,x')|^2\dd t\dd x'
 =\int\frac{|\widehat h(\xi)|^2}{4\pi|\xi_k|}\dd\xi
 \le\frac{\sqrt3}{4\pi N}\|h\|_2^2.
\]
Integrate over a width-$2R$ slab and use orthogonality of the data
pieces and six-term Cauchy--Schwarz:
\begin{equation}\label{prof:local-smoothing}
 \int_\R\int_{B_R}|S(t)1_{\{|\xi|>N\}}(D)f|^2\dd x\dd t
 \le\frac{3^{3/2}R}{\pi N}\|f\|_2^2.
\end{equation}
Estimate~\eqref{prof:local-smoothing} closes the upper truncations; unitarity on finite
cylinders identifies the limit without traces of rough $L^2$ values.

For $f_n\rightharpoonup0$ in $L^2$, each low-frequency value is an
$L^2(B_N)$ pairing tending to zero and bounded by
$|B_N|^{1/2}\sup_n\|f_n\|_2$. Dominated convergence on
$K=[-T,T]\times B_R$ and \eqref{prof:local-smoothing} give
\begin{equation}\label{eq:free-gradient-low-high-limit}
 \limsup_n\|1_KS(t)f_n\|_2^2
 \le\frac{C R}{N}\sup_n\|f_n\|_2^2\longrightarrow0
 \qquad(N\to\infty).
\end{equation}
Apply \eqref{eq:free-gradient-low-high-limit} to the three gradient coordinates.

Fix $h_n=\partial_kS(t)v_n$. Energy and Strichartz give finite
$E_*=\sup_n\|h_n\|_{L^\infty L^2}$ and $W_*=\sup_nW_d(h_n)$.
For $|a|\le A$ supported in $K$, spatial H\"older and
$\int b^5\le(\sup b)^3\int b^2$ give
\begin{equation}\label{eq:local-gradient-coefficient}
 Q_2(a,h_n)\le A|B_R|^{1/30}E_*^{3/5}
                    \|1_Kh_n\|_2^{2/5}\longrightarrow0,
 \qquad Q_2(U-a,h_n)\le\SSS(U-a)W_*.
\end{equation}
Compact approximation in $\SSS$, chosen before $n$, and
\eqref{eq:local-gradient-coefficient} give the first coefficient limit. For the Hartree terms, set
\[
 L_{R,\alpha}=\|I_\alpha1_{B_{2R}}\|_1
             =c_\alpha(4\pi/\alpha)(2R)^\alpha.
\]
Young's inequality and $L^2(B_R)\subset L^{6/5}(B_R)$ give separately
\begin{equation}\label{eq:localized-hartree-coefficients}
 \begin{split}
 \NN(I(\rho'(a)h_n)\ell(a))
 &\le pA^{2p-2}L_{R,\alpha}|B_R|^{1/3}\|1_Kh_n\|_2,\\
 \NN(I\rho(a)\ell'(a)h_n)
 &\le(q+1)A^{2p-2}L_{R,\alpha}|B_R|^{1/3}\|1_Kh_n\|_2.
 \end{split}
\end{equation}
Lemma~\ref{prof:density-factors} gives compact coefficients with
bounded $V$ and $d=\SSS(U-a)\to0$. Keeping $h_n$ unchanged, its
common-direction estimates bound both approximation errors by
$C(d+d^{\theta_0}+d^{\theta_2})$, uniformly in $n$.
Combine these approximation errors with \eqref{eq:localized-hartree-coefficients},
choosing the approximation before $n$; no smallness of $W_d(h_n)$ is needed.
\end{proof}

We apply Lemma~\ref{prof:localization} in each profile frame. At fixed $J$ and
$j<J\le M$, let $v_n=A_{j,n}^{-1}R_n^J$ and $w_n^J=S(t)R_n^J$.
Weak remainder and the all-$n$ bound
$P(v_n)\le B+\sum_{k<J}P(\phi_k)$ give the hypotheses above.
The change of frame gives
\[
 r_{j,n}^{3/2}\nabla w_n^J
 \bigl(r_{j,n}^2(\tau-s_{j,n}),y_{j,n}+r_{j,n}z\bigr)
 =\nabla S(\tau)v_n(z).
\]
The spacetime affine map has Jacobian $r_{j,n}^5$ and preserves a.e.
identities. The local strong estimate therefore applies on the normalized
cylinder. Transforming back gives the scale-invariant coefficient
estimate
\begin{equation}\label{prof:localized-pair}
 Q_2(u_{j,n},\partial_k w_n^J)\longrightarrow0
 \quad\hbox{for fixed admissible }J,\ j<J,\ k\in\{1,2,3\}.
\end{equation}
The two Hartree terms have amplitude $r^{-7/2}$ and invariant
$\NN$ norm; their inverse-frame conclusions have exactly the same
physical version.

\begin{lemma}[Small profiles and uniform reconstruction bounds]
\label{prof:tails}
Write $a_j=P(\phi_j)$ for the family of
Lemma~\ref{prof:countable-family}. Then $\sum_ja_j^2\le B^2$.
There are $\delta>0$ and $C_0,C_1\ge0$, independent of $j$, such that whenever
$a_j\le\delta$, the already chosen $U_j$ is global and
\[
 \mathfrak X(U_j)\le C_0a_j,\qquad
 \mathfrak X(U_j-S(t)\phi_j)\le C_1a_j^{m},\qquad m=2p-1>2.
\]
If $a_j=0$, then $U_j=0$.
Assume explicitly that the finitely many remaining profiles are
global with finite $\mathfrak X$. Then, for
\begin{equation}\label{prof:augmented}
 v_n^J=\sum_{j<J}u_{j,n}+S(t)R_n^J,
\end{equation}
there is one $M_0\ge1$ such that $\mathfrak X(v_n^J)\le M_0$ for
all $n$ and all finite depths. 
\end{lemma}
\begin{proof}
Kinetic decoupling and $P(R_n^J)^2\ge0$ give
\[
 \sum_{j<J}a_j^2\le B^2,
 \qquad\sum_ja_j^2\le B^2
\]
by increasing depth, or by the finite terminal depth.
For an already selected $U$, write
$a=P(\phi)$ and $e_n=P(U(s_n)-S(s_n)\phi)\to0$.
Choose $\delta$ half the radius in Corollary~\ref{local:small-data}.
If $a\le\delta$, then $P(U(s_n))\le a+e_n$ is eventually within
that radius. Reanchoring and unrestricted uniqueness identify the
same maximal trajectory with the global small solution. Thus
\begin{equation}\label{eq:small-profile-matching-bound}
 \mathfrak X(U)\le C(a+e_n),\qquad
 \mathfrak X(U-S(t)\phi)
 \le C(a+e_n)^m+C_{\rm free}e_n.
\end{equation}
Let $n\to\infty$ in \eqref{eq:small-profile-matching-bound}. The constants are uniform in the label and time
regime; $a=0$ gives $U=0$.

Put $e_L=(\sum_{j\ge L}a_j^2)^{1/2}\to0$ and choose $L_0$ beyond
the nonsmall profiles, with $L_0=M$ for finite count. Since $m>2$,
\begin{equation}\label{eq:nonlinear-profile-tail}
 \sum_{j\ge L}a_j^m\le e_L^m\quad(L\ge L_0).
\end{equation}
For $z_{j,n}=u_{j,n}-S(t)b_{j,n}$, frame invariance and exact
cancellation in \eqref{prof:augmented} give
\begin{equation}\label{eq:augmented-profile-cancellation}
 \mathfrak X(z_{j,n})\le C_1a_j^m\quad(j\ge L_0),
 \qquad v_n^J=S(t)f_n+\sum_{j<J}z_{j,n}.
\end{equation}
Equations~\eqref{eq:augmented-profile-cancellation} and \eqref{eq:nonlinear-profile-tail}
control every $n,J$ by the fixed bound
\begin{equation}\label{eq:uniform-profile-cap}
 M_0=\max\left\{1,C_{\rm free}B+
 \sum_{j<L_0}\mathfrak X(U_j-S(t)\phi_j)+C_1e_{L_0}^m\right\}.
\end{equation}
The cap \eqref{eq:uniform-profile-cap} includes partial heads and $J=0$ and uses no $\ell^1$ bound on
$(a_j)$.

For fixed admissible $L\le J$, let
$q_n^{L,J}=\sum_{L\le j<J}b_{j,n}=R_n^L-R_n^J$.
Finite inverse-frame weak tests give
$\langle q_n^{L,J},R_n^J\rangle_\nabla\to0$.
Subtracting the kinetic decouplings yields
\begin{equation}\label{eq:profile-tail-kinetic}
 P(q_n^{L,J})^2
 =P(R_n^L)^2-P(R_n^J)^2
   -2\re\langle q_n^{L,J},R_n^J\rangle_\nabla
 \longrightarrow\sum_{L\le j<J}a_j^2.
\end{equation}
Apply the free estimate and \eqref{eq:profile-tail-kinetic} to $S(t)q_n^{L,J}$,
then add the corrections controlled by \eqref{eq:nonlinear-profile-tail}. For fixed $L_0\le L\le J$
and $\eta>0$, eventually
\begin{equation}\label{prof:tail-block}
 \mathfrak X\left(\sum_{L\le j<J}u_{j,n}\right)
 \le C_{\rm free}(e_L+\eta)+C_1e_L^m.
\end{equation}
Only the index threshold depends on $L,J,\eta$; the bound is uniform
in depth. Empty tails and finite zero padding are included.
\end{proof}

The reconstruction bound depends on a fixed finite collection of large
profiles. The remaining nonlinear corrections are summable because
$m>2$ and $(a_j)\in\ell^2$. Thus the perturbation constants can be fixed
before choosing the profile depth.

\subsection{The residual and two-sided reconstruction}

\begin{theorem}[Nonlinear reconstruction at a selected depth]
\label{prof:reconstruction}
For the common family above suppose the finitely many nonsmall
selected profiles are global with finite $\mathfrak X$.
Then the original canonical solutions through $f_n$ are global with
one uniform finite $\mathfrak X$ bound for all sufficiently large $n$.
For every $\varepsilon>0$ there is a depth allowed by
\eqref{prof:selected-remainder} and one subsequent index threshold
such that $\mathfrak X(u_n-v_n^J)<\varepsilon$ beyond that threshold.
Consequently $u_n$ scatters in both directions in $\X$.
\end{theorem}
\begin{proof}
We first control the residual, then compare with the exact flow on
expanding compact intervals. The order of choices is essential: fix
the uniform reconstruction bound and perturbation constants, choose a
finite group of profiles, and only then choose the remainder depth.
Lemma~\ref{prof:tails} fixes $M_0$ and
Theorem~\ref{local:perturbation} fixes $\epsilon_*,C_+,C_->0$
before depths or tolerances. At fixed admissible $J$, kinetic
decoupling gives eventually
\begin{equation}\label{eq:free-remainder-cap}
 P(R_n^J)\le B+1,\qquad
 \mathfrak X(w_n^J)\le C_{\rm free}(B+1),\qquad
 w_n^J=S(t)R_n^J.
\end{equation}
Localization retains its all-$n$ finite-$J$ bound. Put
$e_n^J=F(v_n^J)-\sum_{j<J}F(u_{j,n})$.
For $\theta=(3-\alpha)/3$, \eqref{local:difference-source}, Sobolev
and $\alpha/3+\theta=1$ give a full-$\mathfrak X$ Lipschitz bound
at cap $M_0+1$:
\begin{equation}\label{eq:full-graph-lipschitz}
 \mathcal N(F(f)-F(g))\le K\mathfrak X(f-g).
\end{equation}
For a common direction, Lemma~\ref{prof:density-factors} gives on
bounded $V,\SSS,W_d(h)$ sets
\begin{equation}\label{eq:common-direction-continuity}
 \NN(\mathcal C(U,h)-\mathcal C(V,h))
 \le C\mu(\SSS(U-V)),\qquad
 \mu(d)=d+d^{\theta_0}+d^{\theta_2}.
\end{equation}

Fix $L\ge L_0$. For a subsequently selected $J$, set
\[
 H=\min(J,L),\quad A_n=\sum_{j<H}u_{j,n},\quad
 T_n=\sum_{H\le j<J}u_{j,n},\quad R_n=\sum_{j<H}|u_{j,n}|.
\]
Maximize constants over the finitely many prefixes $H\le L$.
Kernel positivity and the factor estimates give
\begin{equation}\label{eq:profile-head-direction}
 \begin{aligned}
 |\mathcal C(A_n,k)|
 &\le pI(R_n^{p-1}|k|)R_n^{q+1}+(q+1)I(R_n^p)R_n^q|k|,\\
 \NN(\mathcal C(A_n,k))
 &\le C_LQ_2(R_n,k)\le C_L\sum_{j<H}Q_2(u_{j,n},k).
 \end{aligned}
\end{equation}
The allocation uses
$a_1+a_2=a_0=2\alpha$ and
$b_1-1+b_2=b_0+b_3-1=3$.
For $k=\nabla w_n^J$, \eqref{prof:localized-pair} makes the sum in
\eqref{eq:profile-head-direction} tend to zero at fixed $J$. By real linearity,
\begin{equation}\label{eq:head-remainder-splitting}
\begin{split}
 \nabla F(A_n+w_n^J)-\nabla F(A_n)
 ={}&\mathcal C(A_n+w_n^J,\nabla A_n+\nabla w_n^J)
       -\mathcal C(A_n,\nabla A_n+\nabla w_n^J)\\
 &+\mathcal C(A_n,\nabla w_n^J).
\end{split}
\end{equation}
The bound \eqref{eq:free-remainder-cap} controls the directions and coefficients.
Apply \eqref{eq:common-direction-continuity} and \eqref{eq:profile-head-direction}
to \eqref{eq:head-remainder-splitting}, with constants depending only on
$L,B$ and the fixed profiles:
\begin{equation}\label{prof:head-remainder}
 \mathcal N(F(A_n+w_n^J)-F(A_n))
 \le D_L\mu(\SSS(w_n^J))+q_{L,J}(n),\qquad q_{L,J}(n)\to0.
\end{equation}
The estimate \eqref{prof:head-remainder} requires no small $\WW(w_n^J)$.

The single-source estimate and Lemma~\ref{prof:tails} give
$\mathcal N(F(U_j))\le Ca_j^m$ on the tail.
For residual tolerance $\epsilon>0$, choose in order
\begin{equation}\label{eq:residual-parameter-order}
 \begin{array}{c|l}
 d\in(0,1)&Kd<\epsilon/5\\
 L\ge L_0&C_{\rm free}e_L+C_1e_L^m<d/2,\quad Ce_L^m<\epsilon/5\\
 \eta,\sigma>0&C_{\rm free}\eta<d/2,\quad D_L\mu(\sigma)<\epsilon/5\\
 J,N_s&\text{the output of \eqref{prof:selected-remainder} for }\sigma.
 \end{array}
\end{equation}
For finite count take $L=M$, including zero. Keep this $J$, even
when $J<L$, and enlarge only its index threshold for the remainder
energy, \eqref{prof:tail-block}, localization, the finite-head defect
and matching. If $J<L$, the tail is empty; otherwise
$\mathfrak X(T_n)<d$. Thus $A_n+w_n^J=v_n^J-T_n$ has cap $M_0+1$.
The exact decomposition is
\begin{equation}\label{eq:residual-four-part-splitting}
\begin{split}
 \nabla e_n^J={}&[\nabla F(v_n^J)-\nabla F(A_n+w_n^J)]
 +[\nabla F(A_n+w_n^J)-\nabla F(A_n)]\\
 &+[\nabla F(A_n)-\sum_{j<H}\nabla F(u_{j,n})]
 -\sum_{H\le j<J}\nabla F(u_{j,n}).
\end{split}
\end{equation}
Apply \eqref{eq:full-graph-lipschitz} and \eqref{prof:head-remainder}
to \eqref{eq:residual-four-part-splitting}:
\begin{equation}\label{eq:residual-five-bounds}
 \begin{aligned}
 \mathcal N(e_n^J)
 &\le Kd+D_L\mu(\sigma)+q_{L,J}(n)+Ce_L^m\\
 &\quad+\mathcal N\Bigl(F(A_n)-\sum_{j<H}F(u_{j,n})\Bigr)
 <\epsilon.
 \end{aligned}
\end{equation}
eventually, by Proposition~\ref{prof:finite-defect} and the choices in
\eqref{eq:residual-parameter-order}; each term in \eqref{eq:residual-five-bounds}
is $<\epsilon/5$. All coefficient constants precede $J$.
At time zero, matching and affine isometry give
\begin{equation}\label{eq:profile-initial-matching}
 v_n^J(0)-f_n=\sum_{j<J}T_{r_{j,n},y_{j,n}}
 [U_j(s_{j,n})-S(s_{j,n})\phi_j],\qquad
 P(v_n^J(0)-f_n)\longrightarrow0.
\end{equation}

Finite sums and the weak chain rule give compatible fields
\[
 v_n^J\in C_t\X\cap\mathcal X(I),\qquad
 e_n^J\in C_tL^{6/5}_x,\qquad
 \nabla e_n^J\in L^2_tL^{6/5}_x
\]
on every compact $I$. Spatial cutoff extends the weak identities to
Schwartz tests, as in Lemma~\ref{local:mild}.
Frame covariance and the free equation give
\[
 (i\partial_t+\Delta)u_{j,n}=-F(u_{j,n}),\qquad
 (i\partial_t+\Delta)v_n^J+F(v_n^J)=e_n^J.
\]
The compatible mild formula therefore has signs $+iDF(v_n^J)-iDe_n^J$.

Take $2\epsilon\le\epsilon_*$ and, by \eqref{eq:profile-initial-matching},
an index threshold with initial error $<\epsilon$. Apply Theorem~\ref{local:perturbation} on
$[0,k]$ and $[-k,0]$ with the same datum $f_n$.
Unrestricted uniqueness identifies all restrictions and traces;
the group law joins the oriented mild formulas at zero. Their union
is global and canonical maximality identifies it with the original $u_n$.
Mixed-norm exhaustion gives
\[
 \mathfrak X(u_n-v_n^J)\le2(C_++C_-)\epsilon,
 \qquad\mathfrak X(u_n)\le M_0+2(C_++C_-)\epsilon.
\]
Arbitrarily small $\epsilon$ gives approximation; one fixed choice
gives the eventual uniform bound. Theorem~\ref{local:finite-s} then
gives both scattering states for these same solutions.
\end{proof}

The critical-energy argument will use this reconstruction to exclude multiple nontrivial profiles.

\section{The critical trajectory and scale selection}
\label{crit:section}

We now carry out the concentration--compactness reduction for the Hartree
flow. Minimality leaves one nonzero profile, hence a compact orbit modulo
scale and translation. The scale arguments then select a global
half-orbit with a bounded physical scale; neither a bounded center nor
a positive lower scale bound is required.

Set $\kappa=2p/(p-1)$ and $E_W=E(W)=(p-1)G/(2p)$.
The stable set is
\[
 \mathcal K=\{f\in\X:E(f)<E_W,\ P(f)^2<G\}.
\]
For the maximal solution through $f$ at zero, write
$I(f)=(T_-(f),T_+(f))$. A spacetime norm on $I(f)$ may be infinite;
it is the supremum of the corresponding norms on compact subintervals. Sharpness, positivity of $D$, and trapping give
\begin{equation}\label{eq:stable-coercive-bounds}
 \begin{aligned}
 0\le\frac1\kappa P(f)^2\le E(f)\le\frac12P(f)^2
 &\quad(P(f)^2\le G),\\
 P(u_f(t))^2\le\kappa E(f)<G&\quad(f\in\mathcal K).
 \end{aligned}
\end{equation}
The lower bound in \eqref{eq:stable-coercive-bounds} follows from
$D(f)\le G^{1-p}P(f)^{2p}\le P(f)^2$.
The threshold $G$ is the sharp ground-state level fixed in
Theorem~\ref{std:sharp}.

\subsection{The critical energy and compactness modulo symmetries}

We first pass compact-interval stability to a half-line.

\begin{lemma}[One-sided stability and lifespan control]
\label{crit:halfline-stability}
Suppose $v_n$ are actual solutions on $[0,\infty)$ with uniformly
bounded $\EE$ and $\SSS$ there, and
$P(g_n-v_n(0))\to0$. For all large $n$ the canonical solution through
$g_n$ contains this whole future half-line and has uniformly finite
future $\SSS$. Its discrepancy obeys the corresponding endpoint
and spacetime perturbation bounds. The analogous statement holds on
$(-\infty,0]$, using the backward compact theorem.
\end{lemma}
\begin{proof}
Put $d_n=P(g_n-v_n(0))\to0$. The uniform reference caps and
Theorem~\ref{local:perturbation} give $\varepsilon,C>0$, independent of
$k$, such that, for $d_n<\varepsilon$, the solution on $[0,k]$ satisfies
\begin{equation}\label{eq:compact-halfline-stability}
 \SSS_{[0,k]}(u_{g_n}-v_n)
 +\sup_{0\le t\le k}P(u_{g_n}(t)-v_n(t))\le Cd_n.
\end{equation}
Uniqueness identifies these solutions on overlaps. Their union, joined
to the local solution near zero, lies in the canonical maximal trajectory.
Consequently
\[
 [0,\infty)\subset I(g_n),\qquad
 \SSS_{[0,\infty)}(u_{g_n})
 \le\SSS_{[0,\infty)}(v_n)+Cd_n
\]
by monotone convergence in \eqref{eq:compact-halfline-stability}. The remaining perturbation estimates pass to
the same union. For the past, use the backward theorem with datum at
zero on $[-k,0]$.
\end{proof}

We can now minimize the energy among nonscattering trajectories.

\begin{proposition}[The critical element on its maximal lifespan]
\label{crit:element}
If Theorem~\ref{thm:main} fails, there are $0<E_c<E_W$ and a nonzero
actual canonical maximal solution $u_c$ on $I_c$ such that
\begin{equation}\label{crit:critical-properties}
 E(u_c(t))=E_c,\qquad
 2E_c\le P(u_c(t))^2\le\kappa E_c<G,
 \qquad
 \SSS_{I_c\cap(-\infty,t]}(u_c)
 =\SSS_{I_c\cap[t,\infty)}(u_c)=\infty.
\end{equation}
There are pointwise functions $r(t)>0$, $x(t)\in\R^3$ whose
normalized orbit
$\{T_{r(t),x(t)}^{-1}u_c(t):t\in I_c\}$ has compact closure in
$\X$. 
\end{proposition}
\begin{proof}
\emph{Minimal energy and profile allocation.}
For stable data, Theorem~\ref{local:finite-s}, small free tails and
uniqueness give
\begin{equation}\label{eq:finite-size-scattering-equivalence}
 \SSS_{I(f)}(u_f)<\infty
 \quad\Longleftrightarrow\quad u_f\text{ is global and scatters}.
\end{equation}
Set
\[
 \mathcal B=\{E(f):f\in\mathcal K,\ \SSS_{I(f)}(u_f)=\infty\},
 \qquad E_c=\inf\mathcal B.
\]
Failure and Corollary~\ref{local:small-data}, together with coercivity
and trapping, yield $\mathcal B\ne\varnothing$ and $0<E_c<E_W$.
By \eqref{eq:finite-size-scattering-equivalence}, every stable datum of energy below $E_c$ has finite global
$\SSS$ and $\mathfrak X$; these bounds need not be uniform.
Choose bad energies $E_n\to E_c$ with $E_n\le E_*<E_W$.
On each trajectory split a compact interval of $L^{10}$ integral at
least $2n^{10}$ into two pieces of integral at least $n^{10}$.
Recentering at the splitting time gives
\begin{equation}\label{crit:bilateral-sequence}
 E(f_n)\to E_c,\quad P(f_n)^2\le\kappa E_n,
 \quad
 \SSS_{I(f_n)\cap(-\infty,0]}(u_{f_n})\ge n,
 \quad
 \SSS_{I(f_n)\cap[0,\infty)}(u_{f_n})\ge n.
\end{equation}
Local finiteness and continuity of the cumulative integral justify the
split; conservation and uniqueness preserve the trajectory.

Apply Proposition~\ref{prof:linear} and Lemma~\ref{prof:countable-family}
on one common sequence. The energy subsequence in
Lemma~\ref{prof:energy-splitting} is chosen separately at each fixed $J$. With $a_j=P(\phi_j)$, kinetic decoupling gives, for
each admissible finite $J$,
\begin{equation}\label{crit:kinetic-allocation}
 \sum_{j<J}a_j^2\le\kappa E_c,\qquad
 \limsup_nP(R_n^J)^2\le\kappa E_c-\sum_{j<J}a_j^2<G.
\end{equation}
Together with \eqref{crit:kinetic-allocation}, matching and bounded-set energy continuity identify the conserved
energy of $U_j$ as
\begin{equation}\label{eq:profile-energy-limits}
 E(b_{j,n})\longrightarrow e_j=
 \begin{cases}
 E(S(s_j^*)\phi_j),&s_{j,n}\to s_j^*\in\R,\\
 a_j^2/2,&s_{j,n}\to+\infty\text{ or }-\infty.
 \end{cases}
\end{equation}
Since $P(S(s)\phi_j)^2=a_j^2\le\kappa E_c<G$, \eqref{eq:stable-coercive-bounds} and \eqref{eq:profile-energy-limits} yield
$e_j\ge a_j^2/\kappa$. On the energy subsequence for this fixed $J$, eventually
\begin{equation}\label{crit:energy-allocation}
 E(R_n^J)\ge\kappa^{-1}P(R_n^J)^2,
 \qquad E(R_n^J)\longrightarrow E_c-\sum_{j<J}e_j\ge0.
\end{equation}
Thus $\sum_{j<J}e_j\le E_c$ for every $J$, without asserting one
energy subsequence for all $J$. In particular $0\le e_j\le E_c<E_W$;
late matching, reanchoring and trapping place $U_j$ on the stable branch.

\emph{A single profile and finite normalized time.}
If all $e_j<E_c$, their individual finite-$\mathfrak X$ bounds supply
the nonsmall head in Theorem~\ref{prof:reconstruction}, contradicting
\eqref{crit:bilateral-sequence}. Hence $e_{j_0}=E_c$ for some active
$j_0$. Applying \eqref{crit:energy-allocation} to prefixes containing
$j_0$ and $j$ gives
\begin{equation}\label{eq:unique-active-profile}
 j\ne j_0\ \Longrightarrow\ e_j=a_j=0.
\end{equation}
Refine once at $J_0=j_0+1$. Equations \eqref{crit:energy-allocation} and \eqref{eq:unique-active-profile} give
\begin{equation}\label{crit:one-profile}
 P(f_n-T_{r_n,y_n}S(s_n)\phi)\longrightarrow0,
 \qquad\phi=\phi_{j_0}\ne0.
\end{equation}
If $s_n\to+\infty$, use \eqref{crit:one-profile} and the positive wave $V^+$ with asymptote
$S(s)\phi$ and set
\[
 v_n(t)=T_{r_n,y_n}V^+(s_n+t/r_n^2),\qquad t\ge0.
\]
For large $n$ these are actual future solutions with uniformly bounded
energy norm, $\SSS_{[0,\infty)}(v_n)\to0$, and
$P(f_n-v_n(0))\to0$. Lemma~\ref{crit:halfline-stability} contradicts
future divergence. The negative wave and backward stability separately
exclude $s_n\to-\infty$; neither wave needs its opposite half-line.
Thus $s_n\to s^*\in\R$ and
\[
 g_n:=T_{r_n,y_n}^{-1}f_n\longrightarrow f_c:=S(s^*)\phi
 \quad\text{in }\X,
 \qquad E(f_c)=E_c,\quad P(f_c)^2\le\kappa E_c<G.
\]
For the maximal solution $u_c$ through $f_c$, finite size on either
side would first extend that side to infinity by
Theorem~\ref{local:finite-s}; Lemma~\ref{crit:halfline-stability} would
then contradict the corresponding divergence in
\eqref{crit:bilateral-sequence}. Local finiteness moves both divergences
to every time. Conservation, trapping and $D\ge0$ give
\eqref{crit:critical-properties}.

\emph{Compactness of the normalized orbit.}
The argument just given proves Palais--Smale compactness for stable
sequences with energies tending to $E_c$ and both one-sided critical
sizes diverging. To choose a common normalization,
put
\[
 Q_f(a)=\sup_y\int_{B(y,a)}|\nabla f|^2,
 \qquad Q_f(0+)=0,\quad Q_f(\infty)=P(f)^2\ge2E_c.
\]
Absolute continuity on balls and thin shells, uniformly in the center,
makes $Q_f$ continuous and nondecreasing. Choose
\begin{equation}\label{eq:critical-normalization-choice}
 \begin{gathered}
 r(f)=\inf\{a>0:Q_f(a)\ge E_c\}\in(0,\infty),\qquad Q_f(r(f))=E_c,\\
 \int_{B(x(f),r(f))}|\nabla f|^2>E_c/2.
 \end{gathered}
\end{equation}
For $f_n=u_c(t_n)$, Palais--Smale gives
$f_n=T_{\lambda_n,y_n}h_n$, $h_n\to h$ in $\X$, and $E(h)=E_c$.
In particular
\begin{equation}\label{eq:concentration-function-convergence}
 \sup_{a>0}|Q_{h_n}(a)-Q_h(a)|
 \le\||\nabla h_n|^2-|\nabla h|^2\|_1
 \le(P(h_n)+P(h))P(h_n-h)\longrightarrow0.
\end{equation}
Choose $a,b>0$ with $Q_h(a)<E_c/2$ and $Q_h(b)>3E_c/2$.
Scaling, \eqref{eq:concentration-function-convergence}, and uniform tightness give
\[
 a\le\frac{r(f_n)}{\lambda_n}\le b,\qquad
 \sup_n\frac{|x(f_n)-y_n|}{\lambda_n}<\infty
\]
after discarding finitely many terms; the center bound uses the local
gradient concentration in \eqref{eq:critical-normalization-choice}. Extract convergent parameters with positive limiting
scale. Strong affine continuity, obtained by smooth density and isometry,
then gives a strongly convergent subsequence of
$T_{r(f_n),x(f_n)}^{-1}f_n$. Sequential compactness and completeness of
$\X$ prove compact closure on $I_c$.
\end{proof}

At a hypothetical finite endpoint the physical scale collapses.
Localized charge estimates then recover finite mass and force it to
vanish. This argument uses no mass assumption on the initial datum.

% Original environment 5.3 removed from the active main chain.
\setcounter{theorem}{3}

\subsection{Exclusion of finite endpoints}

\begin{proposition}[No finite endpoint of a compact normalized orbit]
\label{crit:global}
A nonzero canonical maximal trajectory with relatively compact
normalized orbit in $\X$ has neither finite maximal endpoint.
In particular the trajectory of Proposition~\ref{crit:element} is
global. 
\end{proposition}
\begin{proof}
Write $u(t)=T_{r(t),x(t)}h_t$, $h_t\in K$, with $K$ compact, and
let $B=\sup_KP$. Theorem~\ref{local:flow} supplies one centered
lifespan $[-\delta,\delta]$ for all data in $K$.
Maximal-domain covariance in Lemma~\ref{local:covariance} gives
\begin{equation}\label{eq:compact-orbit-local-lifespan}
 [t-\delta r(t)^2,t+\delta r(t)^2]\subset I(u).
\end{equation}
If $T_+<\infty$, \eqref{eq:compact-orbit-local-lifespan} forces $\delta r(t)^2<T_+-t$ and
$r(t)\to0$. The left endpoint gives the analogous collapse,
without regularity of the normalizing parameters.

For fixed $R,\epsilon>0$, take a finite smooth $P$-$\epsilon$ net
$\{g_i\}$ of $K$, and put $L_\epsilon=\max_i\|g_i\|_2^2$.
Sobolev and scaling give, uniformly in $h\in K$, $y$ and $\rho>0$,
\begin{equation}\label{eq:shrinking-orbit-local-mass}
 \int_{B_{2R}}|T_{\rho,y}h|^2
 \le2|B_{2R}|^{2/3}C_S^2\epsilon^2+2\rho^2L_\epsilon.
\end{equation}
In \eqref{eq:shrinking-orbit-local-mass}, choose $\epsilon$ before $\rho$. Let $\chi=\zeta^2$ for a real smooth
cutoff $0\le\zeta\le1$, equal to one on $B_1$ and supported in $B_2$.
Then $M_R(t)=\int\chi(x/R)|u(t,x)|^2\dd x\to0$ at each finite
endpoint. Only the approximants have required global $L^2$ membership.
Lemma~\ref{local:charge} and
$\|\nabla\chi(\cdot/R)/R\|_3=\|\nabla\chi\|_3$ give
\begin{equation}\label{eq:finite-endpoint-local-charge}
 |M_R(b)-M_R(a)|\le CB^2(b-a).
\end{equation}
Assuming $T_+<\infty$, let $b\uparrow T_+$ in \eqref{eq:finite-endpoint-local-charge} at fixed $t,R$:
$0\le M_R(t)\le CB^2(T_+-t)$. Then Fatou along integer $R\to\infty$
proves scalar $L^2$ membership and
\begin{equation}\label{crit:hypothetical-mass}
 \|u(t)\|_2^2\le CB^2(T_+-t).
\end{equation}
Only now use the stronger charge bound
\begin{equation}\label{eq:finite-endpoint-improved-charge}
 |M_R(b)-M_R(a)|
 \le\frac{2\|\nabla\chi\|_\infty}{R}(b-a)
       \sup_{[a,b]}\|u(t)\|_2\sup_{[a,b]}P(u(t)).
\end{equation}
For $a=t$, \eqref{crit:hypothetical-mass} bounds the right side of \eqref{eq:finite-endpoint-improved-charge} by
$C_\chi B^2R^{-1}(b-t)(T_+-t)^{1/2}$.
Let $b\uparrow T_+$ to obtain
\[
 0\le M_R(t)\le\frac{C_\chi B^2}{R}(T_+-t)^{3/2}
 \longrightarrow0\qquad(R\to\infty).
\]
Dominated convergence gives $u(t)=0$, contradicting nonzero $u$ by
uniqueness. At a finite left endpoint the same argument uses
$\|u(t)\|_2^2\le CB^2(t-T_-)$.
All mass bounds are conditional on the hypothesized finite endpoint.
\end{proof}

For the global orbit, no finite-mass conclusion has yet been obtained.

\subsection{Compact normalizations and subparabolic scale}

\begin{lemma}[Strong limits at the critical level]
\label{crit:strong-limit}
Let $f_n\to f$ strongly in $\X$, where each canonical solution
through $f_n$ is bilaterally infinite in scalar size on its actual
maximal interval, $E(f_n)=E_c$ and $P(f_n)^2\le\kappa E_c<G$.
Then the canonical solution $V$ through $f$ is bilaterally infinite,
has compact normalized orbit on its own maximal interval, and is
global. Moreover, on each fixed compact interval of $V$, the
canonical solutions through $f_n$ converge to $V$ uniformly in the
strong $\X$ topology, including moving evaluation times.
\end{lemma}
\begin{proof}
Energy continuity and trapping give
\[
 E(f)=E_c>0,\qquad P(f)^2\le\kappa E_c<G.
\]
If $V$ had finite scalar size on either side,
Theorem~\ref{local:finite-s} would extend that side to infinity and
Lemma~\ref{crit:halfline-stability} would contradict the approximants'
divergence. For the past use backward stability. Local finiteness moves
both divergences to every time. Proposition~\ref{crit:element} then gives
compact normalized orbit on $I(V)$, and Proposition~\ref{crit:global}
gives $I(V)=\R$.

Apply Theorem~\ref{local:perturbation} to $V$ on $[0,H]$ and $[-H,0]$.
Construction and uniqueness give, for all large $n$,
\begin{equation}\label{eq:critical-flow-uniform-convergence}
 [-H,H]\subset I(U_n),\qquad
 \sup_{|s|\le H}P(U_n(s)-V(s))\le C_HP(f_n-f)\longrightarrow0.
\end{equation}
For $s_n\to s$, choose $H>|s|$. Eventually $|s_n|\le H$, so \eqref{eq:critical-flow-uniform-convergence} gives
\[
 P(U_n(s_n)-V(s))
 \le C_HP(f_n-f)+P(V(s_n)-V(s))\longrightarrow0.
\]
Sobolev equivalence gives the asserted strong $\X$ convergence.
\end{proof}

We next control the normalization parameters on compact time intervals.

\begin{lemma}[Weak disappearance and local properness]
\label{crit:properness}
For fixed $f\in\X$, put
$D_{r,y}f(z)=r^{1/2}f(y+rz)=T_{r,y}^{-1}f(z)$.
Its gradients converge weakly to zero in each of the following
parameter regimes: $r\to\infty$ with arbitrary centers;
$r\to0$ with arbitrary centers; or $0<r\le A<\infty$ and
$|y|\to\infty$.
If a continuous critical trajectory satisfying
\eqref{crit:critical-properties} has a compact normalized carrier for
arbitrary pointwise choices $r(t),x(t)$, then on each nonempty compact
time interval $J$ there are constants
\[
 0<a_J\le r(t)\le A_J<\infty,\qquad |x(t)|\le X_J<\infty
 \quad(t\in J).
\]
\end{lemma}
\begin{proof}
For compact smooth scalar $\psi$ and vector test $\eta$,
$\nabla D_{r,y}\psi=r^{3/2}\nabla\psi(y+r\cdot)$ gives, uniformly in $y$,
\begin{equation}\label{eq:affine-weak-disappearance}
 |\langle\nabla D_{r,y}\psi,\eta\rangle|
 \le\min\{r^{-3/2}\|\nabla\psi\|_1\|\eta\|_\infty,
          r^{3/2}\|\nabla\psi\|_\infty\|\eta\|_1\}.
\end{equation}
The two terms in \eqref{eq:affine-weak-disappearance} give the first two regimes. If
$\supp\nabla\psi\subset B_L$, $\supp\eta\subset B_H$ and $r\le A$,
the pairing is zero for $|y|>L+AH$.
Density in $P$ and $L^2$, isometry and Cauchy--Schwarz extend the
three limits to arbitrary data and tests.

On compact $J$, a sequence violating a parameter bound admits, after
extraction, $t_n\to t_*$, $u(t_n)\to f=u(t_*)$ and
$h_n=D_{r(t_n),x(t_n)}u(t_n)\to h$ strongly. Since
\begin{equation}\label{eq:normalized-strong-comparison}
 P(h_n-D_{r(t_n),x(t_n)}f)=P(u(t_n)-f)\longrightarrow0,
\end{equation}
\eqref{eq:normalized-strong-comparison} and weak disappearance force $P(h)=0$, contrary to $P(h)^2\ge2E_c$.
Apply this first to $r(t_n)\to\infty$, then to $r(t_n)\to0$, and
finally, using the upper scale bound, to $|x(t_n)|\to\infty$.
No extrema or regularity of the parameters are required.
\end{proof}

If the physical scale grew as fast as $|t|^{1/2}$ along a sequence,
rescaling would place the fixed original time inside a bounded time
window. Stability would give a nonzero strong limit there, whereas
escape of dilations gives weak limit zero.

\begin{proposition}[Subparabolic scale for every compact normalization]
\label{crit:scale}
Let $u$ be a global bilaterally bad critical trajectory and let
$r(t)>0,x(t)$ be any pointwise choices making
$D_{r(t),x(t)}u(t)$ relatively compact in $\X$. Then
\[
 \frac{r(t)^2}{|t|}\longrightarrow0\qquad(|t|\to\infty).
\]
In particular the finite running supremum
$A(T)=\sup_{0\le t\le T}r(t)$ satisfies $A(T)/\sqrt T\to0$.

\end{proposition}
\begin{proof}
Otherwise, for some $\epsilon>0$, choose
$|t_n|\to\infty$ with $r_n^2=r(t_n)^2\ge\epsilon|t_n|$.
After extraction,
$f_n=D_{r_n,x_n}u(t_n)\to f$ in $\X$.
Maximal affine covariance gives global canonical critical trajectories
\[
 U_n(s,z)=r_n^{1/2}u(t_n+r_n^2s,x_n+r_nz).
\]
They retain energy $E_c$, the strict kinetic cap and bilateral divergence.
Lemma~\ref{crit:strong-limit} supplies their global bilaterally bad limit
$V$. Extract once more so that
\begin{equation}\label{eq:rescaled-time-limit}
 \tau_n=-t_n/r_n^2\longrightarrow\tau,\qquad
 |\tau_n|\le\epsilon^{-1},\qquad r_n\longrightarrow\infty.
\end{equation}
Apply the moving-time conclusion to \eqref{eq:rescaled-time-limit}, and use Lemma~\ref{crit:properness}, to obtain
\begin{equation}\label{eq:critical-strong-weak-contradiction}
 U_n(\tau_n)=D_{r_n,x_n}u(0)\longrightarrow V(\tau)
 \quad\text{in }\X,\qquad D_{r_n,x_n}u(0)\rightharpoonup0.
\end{equation}
Equation \eqref{eq:critical-strong-weak-contradiction} implies $V(\tau)=0$, contradicting $E(V(\tau))=E_c>0$ in either time
direction, including $\tau=0$.

Local properness makes $A(T)$ finite. Given $\epsilon>0$, choose
$t_\epsilon\ge1$ with $r(t)^2\le\epsilon t$ for $t\ge t_\epsilon$ and
put $A_\epsilon=\sup_{0\le t\le t_\epsilon}r(t)$. Then
\[
 \limsup_{T\to\infty}\frac{A(T)^2}{T}
 \le\limsup_{T\to\infty}
       \max\!\left\{\frac{A_\epsilon^2}{T},\epsilon\right\}
 =\epsilon.
\]
Let $\epsilon\downarrow0$. No measurability of the parameters is used.
\end{proof}

We select times where the running physical scale is nearly maximal.
After rescaling about those times, the adjacent normalized time windows
have diverging total length. One direction therefore yields a whole
half-line on which the physical scale is bounded above.

\subsection{A bounded-scale half-line}

\begin{theorem}[A bounded physical scale on one half-line]
\label{crit:bounded-halfline}
Under failure of Theorem~\ref{thm:main}, there are a global actual
canonical critical trajectory $w$, a sign $\sigma\in\{+1,-1\}$,
and pointwise functions $r_w(t)>0,x_w(t)$ with compact normalized
carrier, such that $E(w)=E_c$, both actual scalar-size halves are
infinite, and
\[
 0<r_w(t)\le1\quad\hbox{whenever }\sigma t\ge0.
\]
Writing $U(s)=w(\sigma s)$ for $s\ge0$ gives the same actual
homogeneous trajectory in this orientation, with
$i\sigma\partial_sU+\Delta U+F(U)=0$. 
\end{theorem}
\begin{proof}
Take the global $u$ from Propositions~\ref{crit:element}, \ref{crit:global}
and its compact carrier
\[
 K=\overline{\{D_{r(t),x(t)}u(t):t\in\R\}},\qquad
 E(h)=E_c,\quad P(h)^2\ge2E_c\quad(h\in K).
\]
For $T_n=n+1$, local properness and Proposition~\ref{crit:scale} give
\begin{equation}\label{eq:running-scale-selection}
 A_n:=\sup_{0\le t\le T_n}r(t)\in[r(0),\infty),\qquad
 A_n^2/T_n\longrightarrow0.
\end{equation}
Choose $\tau_n\in[0,T_n]$ with $a_n=r(\tau_n)>A_n/2$, and let
$y_n=x(\tau_n)$. By covariance,
\[
 v_n(s,z)=a_n^{1/2}u(\tau_n+a_n^2s,y_n+a_nz),\qquad v_n(0)\in K
\]
are global canonical critical trajectories with bilateral divergence.
The interval $[0,T_n]$ becomes
\begin{equation}\label{eq:normalized-time-window}
 [-\ell_n,q_n],\qquad
 \ell_n=\tau_n/a_n^2,\quad q_n=(T_n-\tau_n)/a_n^2,\qquad
 \ell_n+q_n\ge T_n/A_n^2\longrightarrow\infty.
\end{equation}
By \eqref{eq:running-scale-selection}--\eqref{eq:normalized-time-window}, pass to a fixed longer direction $\sigma\in\{-1,1\}$ and a subsequence
with $v_n(0)\to f\in K$. Then
\begin{equation}\label{eq:oriented-window-exhaustion}
 L_n:=\max(\ell_n,q_n)\longrightarrow\infty,\qquad
 \{s:0\le\sigma s\le L_n\}\subset[-\ell_n,q_n].
\end{equation}
Lemma~\ref{crit:strong-limit} yields a global bilaterally bad $v$ with
compact normalized orbit and $v_n\to v$ strongly, uniformly on compact
time intervals.

We claim that
\[
 \mathcal Y_2=\{T_{b,c}h:h\in K, 0<b\le2, c\in\R^3\}
\]
is closed. If $g_n=T_{b_n,c_n}h_n\to g$, extract $h_n\to h\in K$
and $b_n\to b\in[0,2]$. Then
\begin{equation}\label{eq:carrier-closedness-comparison}
 P(h_n-D_{b_n,c_n}g)=P(g_n-g)\longrightarrow0.
\end{equation}
By \eqref{eq:carrier-closedness-comparison} and Lemma~\ref{crit:properness}, either $b=0$ or an unbounded subsequence
of centers would force $P(h)=0$. Thus $b>0$ and, after extraction,
$c_n\to c$. Strong affine continuity gives $g=T_{b,c}h\in\mathcal Y_2$.
Only closedness is asserted.

Fix $s$ with $\sigma s\ge0$. By \eqref{eq:oriented-window-exhaustion}, eventually
$t_n(s)=\tau_n+a_n^2s\in[0,T_n]$, and
\begin{equation}\label{eq:bounded-scale-limit-frame}
 \begin{gathered}
 v_n(s)=T_{b_n(s),c_n(s)}h_{t_n(s)},\qquad h_{t_n(s)}\in K,\\
 b_n(s)=\frac{r(t_n(s))}{a_n}\le\frac{A_n}{a_n}<2,\qquad
 c_n(s)=\frac{x(t_n(s))-x(\tau_n)}{a_n}.
 \end{gathered}
\end{equation}
Closedness applied to \eqref{eq:bounded-scale-limit-frame} gives $v(s)\in\mathcal Y_2$. Choose one representation for
each such $s$; the parameter subsequences may depend on $s$. On the
opposite open half-line keep the existing normalization of $v$.
The union of the two compact carriers remains compact.

Finally, set $w(t,z)=2^{1/2}v(4t,2z)$. Covariance preserves the critical
properties, while
\[
 2^{1/2}[T_{b(4t),c(4t)}h](2z)
 =T_{b(4t)/2,c(4t)/2}h(z),\qquad
 0<r_w(t)=b(4t)/2\le1\quad(\sigma t\ge0).
\]
Take $x_w(t)=c(4t)/2$ on this closed half-line. Substituting $t=\sigma s$
gives the stated oriented equation; positive dilation preserves $\sigma$.
\end{proof}

The remaining argument is carried out on this half-orbit, without assuming finite mass.

\section{The no-waste formula and the annular forcing bound}
\label{sec:annular}

The first task is to improve the spatial integrability of the critical
trajectory before using an interaction identity. The no-waste formula
reduces this task to an annular $L^1$ bound for the forcing. For
noninteger powers, even a low-frequency field can produce high-frequency
forcing, so this contribution must be estimated explicitly.
Throughout this section $0<\alpha<3$, $p=3+\alpha$, and the positive
Riesz normalization is fixed as in the introduction. We write
\[
 \mathcal A(f)=I_\alpha*|f|^p,\qquad
 \mathfrak b(f)=|f|^{p-2}f,\qquad F(f)=\mathcal A(f)\mathfrak b(f).
\]
Use the $L^6$ representatives of $\X$. Constants may depend on the fixed
$\alpha$, cutoff, Sobolev and strict HLS constants.

\subsection{Fourier cutoffs on the homogeneous space}

Fix the real radial $\varphi\in C_c^\infty(\mathbb R^3)$ of the common
conventions, with $0\leq\varphi\leq1$, $\varphi=1$ on the unit ball and
$\varphi=0$ outside the ball of radius two. With angular frequency
$\vartheta=2\pi\xi$, put
\[
 \ell_M(\vartheta)=\varphi(\vartheta/M),\quad
 P_{\leq M}=\ell_M(D),\quad P_{>M}=1-P_{\leq M},\quad
 P_M=P_{\leq M}-P_{\leq M/2}.
\]
Thus the symbol $m_M$ of $P_M$ is supported in
$\{M/2\leq|\vartheta|\leq2M\}$. An enlarged annular multiplier will
allow an integrable replacement of the forcing without changing its
$P_M$ component:
\begin{equation}\label{ann:enlarged-symbol}
 Q_M=P_{\leq2M}-P_{\leq M/4},\qquad
 q_M(\vartheta)=\varphi(\vartheta/(2M))-\varphi(4\vartheta/M),\qquad
 m_Mq_M=m_M.
\end{equation}
Indeed \eqref{ann:enlarged-symbol} has $q_M=1$ on the support of $m_M$. Neither operator is an idempotent
projection. All equalities below are multiplier identities for these
particular smooth functions.

Since the original field need not lie in $L^2$, the high-frequency
part must first be identified as an $L^2$ function with the prescribed
distributional gradient.

\begin{lemma}[Homogeneous Fourier cutoffs]
\label{ann:cutoff-realization}
For $f\in\X$ and $M>0$ the low pass is the actual smooth convolution
$k_M*f$, where $k_M(x)=M^3k(Mx)$ and $k$ is the inverse angular Fourier
transform of $\varphi$. The high pass has a unique $L^2$ representative
with the same distribution as $(1-\ell_M)(D)f$. They satisfy
\begin{equation}\label{eq:homogeneous-cutoff-bounds}\begin{aligned}
 \|\nabla P_{\leq M}f\|_2+\|\nabla P_{>M}f\|_2
   &\leq2\|\nabla f\|_2,\\
 \|P_{>M}f\|_2&\leq M^{-1}\|\nabla f\|_2,\\
 \|P_{\leq M}f\|_\infty&\leq C M^{1/2}\|\nabla f\|_2,\\
 \|P_Mf\|_2&\leq2M^{-1}\|\nabla f\|_2,\qquad
 \|P_Mf\|_\infty\leq C M^{1/2}\|\nabla f\|_2.
\end{aligned}\end{equation}
Each of the two gradient norms in the first line is separately at most
$\|\nabla f\|_2$.  The annular representative belongs to $C_0(\mathbb R^3)$.
\end{lemma}
\begin{proof}
Take $f_n\in C_c^\infty$ with $f_n\to f$ in $\dot H^1\cap L^6$.
Plancherel and
\[
 0\le\ell_M,1-\ell_M\le1,\qquad
 |1-\ell_M(\vartheta)|\le|\vartheta|/M
\]
give the gradient contractions and
\[
 \|P_{>M}(f_n-f_m)\|_2
 \le M^{-1}P(f_n-f_m).
\]
The highpass estimate in \eqref{eq:homogeneous-cutoff-bounds} follows as $P_{>M}f_n$ converges in $L^2$ and $\X$. Local $L^2$ convergence
identifies the limits; compact tests identify the multiplier and its weak
derivatives.

Schwartz convolution gives the smooth low pass and
\[
 \|k_M*f\|_\infty\le\|k_M\|_{6/5}\|f\|_6,\qquad
 \|k_M\|_{6/5}=M^{1/2}\|k\|_{6/5}.
\]
Its distribution is $f-P_{>M}f$. For the annular part, use
\begin{equation}\label{eq:annular-inverse-gradient-symbol}
 |m_M(\vartheta)|\le2|\vartheta|/M,\qquad Q_MP_M=P_M.
\end{equation}
The symbol bound in \eqref{eq:annular-inverse-gradient-symbol} gives $\|P_Mf\|_2\le2M^{-1}P(f)$.
The kernel of $Q_M$ has $L^2$ norm $CM^{3/2}$, so Young yields
\[
 \|P_Mf\|_\infty\le CM^{3/2}\|P_Mf\|_2\le CM^{1/2}P(f).
\]
Approximation of the $L^2$ input by compact functions gives $P_Mf\in C_0$.
\end{proof}

The Hartree nonlinearity is initially defined in its energy-space target.

\begin{lemma}[Baseline Hartree domains]
\label{ann:baseline}
Set
\[
 s_0=\frac6p,\qquad q_0=\frac6{3-\alpha},\qquad
 t_0=\frac6{p-1}.
\]
For every $f\in\X$, $|f|^p\in L^{s_0}$,
$\mathcal A(f)\in L^{q_0}$, $\mathfrak b(f)\in L^{t_0}$ and
$F(f)\in L^{6/5}$. All convolutions and products are the actual representatives;
$F:L^6\to L^{6/5}$ is continuous. On a bounded
$L^6$ ball it obeys
\[
 \|F(f)-F(g)\|_{6/5}
 \leq C(\|f\|_6+\|g\|_6)^{2p-2}\|f-g\|_6.
\]
\end{lemma}
\begin{proof}
Theorem~\ref{std:hls} and H\"older give the domains, since
\[
 1<s_0<3/\alpha,\qquad q_0^{-1}=s_0^{-1}-\alpha/3,\qquad
 q_0^{-1}+t_0^{-1}=5/6.
\]
For continuity, use the real-plane mean-value bounds
\begin{equation}\label{ann:power-differences}
 \begin{split}
 ||z|^p-|w|^p|&\le C_p|z-w|(|z|+|w|)^{p-1},\\
 |\mathfrak b(z)-\mathfrak b(w)|&\le C_p|z-w|(|z|+|w|)^{p-2}
 \end{split}
\end{equation}
Insert \eqref{ann:power-differences} into
\begin{equation}\label{eq:hartree-two-factor-subtraction}
 \begin{aligned}
 F(f)-F(g)
 &=\mathcal A(f)(\mathfrak b(f)-\mathfrak b(g))\\
 &\quad+(I_\alpha*(|f|^p-|g|^p))\mathfrak b(g).
 \end{aligned}
\end{equation}
Applying HLS and H\"older to \eqref{eq:hartree-two-factor-subtraction} yields the asserted difference estimate.
\end{proof}

We now apply these bounds to the selected critical trajectory.

\subsection{Low-frequency control on the compact half-orbit}

For the rest of the dynamical arguments take exactly the trajectory and
sign provided by Theorem~\ref{crit:bounded-halfline}. Thus
\begin{equation}\label{ann:selected-path}
 U(s)=u(\sigma s),\quad s\geq0,\quad \sigma\in\{-1,1\},\qquad
 f_s=T_{r(s),x(s)}^{-1}U(s)\in K,\quad 0<r(s)\leq1,
\end{equation}
where $K$ is compact in $\X$. Here $r(s),x(s)$ are the pointwise
normalizing choices after composition with physical time $\sigma s$;
no measurability or continuity of these choices is imposed. Put
$B=\sup_{s\geq0}\|\nabla U(s)\|_2\leq\sup_{f\in K}\|\nabla f\|_2<\infty$.

Compactness makes the low-frequency estimates uniform in time.

\begin{proposition}[Uniform low-gradient and weighted high-pass mass]
\label{ann:uniform-highpass}
For the same cutoff and the same $U$, the finite real quantities
\[
 \eta_\nu=\sup_{s\geq0}\|\nabla P_{\leq\nu}U(s)\|_2,\quad
 m_\nu=\sup_{s\geq0}\|P_{>\nu}U(s)\|_2,\quad
 \delta_\nu=\nu m_\nu
\]
satisfy
\[
 m_\nu\leq B/\nu<\infty,\qquad
 \eta_\nu\longrightarrow0,\quad\delta_\nu\longrightarrow0
 \quad(\nu\downarrow0).
\]
For fixed $\nu>0$, $P_{>\nu}U\in C([0,\infty);H^1)$.
\end{proposition}
\begin{proof}
For $f\in\X$, let $G_j=\partial_jf$ and
$A_f(\xi)=\sum_j|\widehat G_j(\xi)|^2\in L^1$. Plancherel gives
\begin{equation}\label{eq:pointwise-low-gradient-limit}
 \|\nabla P_{\le\kappa}f\|_2^2
 =\int|\varphi(2\pi\xi/\kappa)|^2A_f(\xi)\,d\xi\longrightarrow0.
\end{equation}
For the scalar $L^2$ highpass already constructed, the Fourier derivative
identity gives, off the null point $\xi=0$,
\begin{equation}\label{eq:pointwise-weighted-highpass-limit}
 \kappa^2\|P_{>\kappa}f\|_2^2
 =\int\frac{\kappa^2|1-\varphi(2\pi\xi/\kappa)|^2}
                 {4\pi^2|\xi|^2}A_f(\xi)\,d\xi\longrightarrow0.
\end{equation}
The multipliers in \eqref{eq:pointwise-low-gradient-limit}--\eqref{eq:pointwise-weighted-highpass-limit} are bounded by one and tend to zero off zero; dominated convergence gives both limits. Only the gradients are Fourier
transformed as $L^2$ inputs.

Define the finite moduli
\[
 \omega_1(\nu)=\sup_{f\in K,\,0<\kappa\le\nu}
                 \|\nabla P_{\le\kappa}f\|_2,
 \qquad
 \omega_2(\nu)=\sup_{f\in K,\,0<\kappa\le\nu}
                 \kappa\|P_{>\kappa}f\|_2.
\]
The cutoff bounds control both maps by $\|\nabla f\|_2$.
A finite net of $K$ and \eqref{eq:pointwise-low-gradient-limit}--\eqref{eq:pointwise-weighted-highpass-limit} show
$\omega_1(\nu),\omega_2(\nu)\to0$. This uses no monotonicity of
$\varphi$. Fourier change of variables, extended by homogeneous density,
gives
\begin{equation}\label{eq:cutoff-affine-covariance}
 T_{r,x}^{-1}P_{\le\nu}=P_{\le\nu r}T_{r,x}^{-1},\qquad
 T_{r,x}^{-1}P_{>\nu}=P_{>\nu r}T_{r,x}^{-1}.
\end{equation}
The highpass identity in \eqref{eq:cutoff-affine-covariance} scales the scalar $L^2$ norm by $r^{-1}$.
For $\kappa_s=\nu r(s)\le\nu$,
\[
 \|\nabla P_{\le\nu}U(s)\|_2\le\omega_1(\nu),\qquad
 \nu\|P_{>\nu}U(s)\|_2
 =\kappa_s\|P_{>\kappa_s}f_s\|_2\le\omega_2(\nu).
\]
Together with $m_\nu\le B/\nu$, these prove the assertions. The bounded
linear map $P_{>\nu}:\X\to H^1$ and continuity of $U$ give the path
continuity at each fixed $\nu$.
\end{proof}

The small quantities are $\eta_\nu$ and $\nu m_\nu$. The norm
$m_\nu$ itself may grow as $\nu\downarrow0$, and the original
trajectory still need not have finite mass.

\subsection{Weak terminal zero and the signed no-waste identity}

\begin{lemma}[The free kernel]
\label{ann:free-kernel}
For $t\ne0$ and a Schwartz function $g$, the same group
$S(t)=e^{it\Delta}$ obeys
\[
 S(t)g(x)=(4\pi it)^{-3/2}
        \int e^{i|x-y|^2/(4t)}g(y)\,dy,
 \qquad \|S(t)g\|_\infty\leq(4\pi|t|)^{-3/2}\|g\|_1.
\]
Consequently, for the actual annular kernel $K_{L,t}$ of $S(t)P_L$,
\begin{equation}\label{eq:integrable-annular-kernel}
 \|K_{L,t}\|_\infty\leq C_\varphi\min\{L^3,|t|^{-3/2}\},
 \qquad \int_0^\infty\min\{L^3,\tau^{-3/2}\}\,d\tau=3L.
\end{equation}
At $t=0$ the minimum is interpreted as $L^3$.
\end{lemma}
\begin{proof}
The Gaussian regularization in the proof of Theorem~\ref{std:endpoint}
gives the stated kernel for both signs of $t$, with the power branch
continued from $\varepsilon>0$.
Let $K_L$ be the inverse transform of the annular symbol. Then
\[
 K_L(x)=L^3K_1(Lx),\qquad \|K_L\|_1=\|K_1\|_1.
\]
Absolute Fourier inversion and the dispersive estimate give, respectively,
\[
 \|S(t)K_L\|_\infty\le CL^3,\qquad
 \|S(t)K_L\|_\infty\le C|t|^{-3/2}.
\]
Taking the minimum proves the kernel bound. Splitting at $L^{-2}$ gives
\[
 \int_0^\infty\min(L^3,\tau^{-3/2})\,d\tau
 =L^3L^{-2}+2(L^{-2})^{-1/2}=3L.
\]
\end{proof}

This integrable annular bound will be used after eliminating the free terminal term.

\begin{proposition}[The no-waste identity]
\label{ann:no-waste}
The path in \eqref{ann:selected-path} satisfies
\[
 S(-\sigma T)U(T)\rightharpoonup0\quad(T\to\infty)
       \quad\hbox{in }\X.
\]
For every fixed $t\geq0$ the single absolute sequence $T_n=n+1$ satisfies
\begin{equation}\label{ann:no-waste-equation}
 U(t)=-i\sigma\lim_{n\to\infty}
       \int_t^{T_n}S(\sigma(t-s))F(U(s))\,ds
       \quad\hbox{in }\mathcal S'.
\end{equation}
The finite integrals are taken against Schwartz tests. The sequence
$T_n$ is independent of $t$ and the test.
\end{proposition}
\begin{proof}
Fix a Schwartz vector test $\zeta$ and a finite homogeneous Schwartz
$\varepsilon$-net $\{\psi_k\}$ of $K$. In the Hermitian gradient pairing,
\begin{equation}\label{eq:weak-terminal-pairing}
 \begin{aligned}
 |\langle\nabla S(-\sigma T)U(T),\zeta\rangle|
 &\le\varepsilon\|\zeta\|_2
   +|\langle\nabla T_{r(T),x(T)}\psi_k,S(\sigma T)\zeta\rangle|\\
 &\le\varepsilon\|\zeta\|_2+C_{\varepsilon,\zeta}T^{-3/2}.
 \end{aligned}
\end{equation}
Here Lemma~\ref{ann:free-kernel} and $r(T)\le1$ give
$\|\nabla T_{r,x}\psi_k\|_1=r^{3/2}\|\nabla\psi_k\|_1\le C_\varepsilon$,
uniformly in the center. In \eqref{eq:weak-terminal-pairing}, let $T\to\infty$, then $\varepsilon\downarrow0$.
Density, with error $B\|\zeta-\zeta_n\|_2$, extends the limit to all
$L^2$ vector tests. The closed gradient embedding gives weak convergence
in $\X$; Sobolev then gives convergence against Schwartz scalar tests.

The finite Duhamel identity with oriented endpoints, under the physical
time substitution $s=\sigma q$, reads
\begin{equation}\label{ann:finite-signed}
 U(t)=S(\sigma(t-T))U(T)
       -i\sigma\int_t^TS(\sigma(t-s))F(U(s))\,ds.
\end{equation}
Since $F(U)\in C_tL^{6/5}$, these integrals converge absolutely against
Schwartz tests. For the common sequence $T_n=n+1$,
\[
 S(\sigma(t-T_n))U(T_n)
 =S(\sigma t)S(-\sigma T_n)U(T_n)\rightharpoonup0,
\]
which proves \eqref{ann:no-waste-equation}. Each finite integral has the
$\X$ realization $i\sigma[U(t)-S(\sigma(t-T))U(T)]$; no unprojected
infinite-time Bochner integral is used.
\end{proof}

The no-waste identity now has no free terminal term. To apply the
integrable annular kernel, it remains to represent each projected forcing
by an $L^1$ function with a quantitative bound.

\subsection{Low-frequency leakage for noninteger powers}

The Fourier support of $a=P_{\le M}f$ does not force that of $F(a)$
to remain in a bounded low-frequency ball. Differentiating the complete
Hartree product gives an $L^1$ bound, and an inverse derivative on an
annulus then supplies the factor $L^{-1}$.

\begin{lemma}[Pure-low leakage]
\label{ann:pure-low}
For $a=P_{\leq M}f$, $f\in\X$, the actual forcing has a distributional
gradient in $L^1$, and
\begin{equation}\label{ann:pure-low-bound}
 \|\nabla F(a)\|_1
 \leq C_\alpha\|a\|_\infty\|\nabla a\|_2^{2p-2}
 \leq C_{\alpha,\varphi}M^{1/2}\|\nabla f\|_2^{2p-1}.
\end{equation}
If $R_L$ is any fixed smooth annular multiplier rescaled by $L$, then
$R_LF(a)$ has an actual $L^1\cap L^{6/5}$ representative and
\begin{equation}\label{eq:annular-inverse-gradient-estimate}
 \|R_LF(a)\|_1\leq C_R L^{-1}\|\nabla F(a)\|_1.
\end{equation}

\end{lemma}
\begin{proof}
Put $V=\|a\|_\infty$, $S_a=\|a\|_6$ and $G=\|\nabla a\|_2$.
Differentiate first the local factor $\mathfrak b(a)$, then the
Riesz potential. For the first term, HLS and H\"older at the conjugate
pair $s_0,q_0$ give
\[
 \begin{aligned}
 \|\nabla\mathfrak b(a)\|_{s_0}
 &\le(p-1)VS_a^{p-3}G,\qquad \frac12+\frac{p-3}{6}=\frac p6,\\
 \|\mathcal A(a)\nabla\mathfrak b(a)\|_1
 &\le C_\alpha VS_a^{2p-3}G.
 \end{aligned}
\]
For the potential derivative put $q_1=6/(4-\alpha)$, so that
$q_1^{-1}+t_0^{-1}=1$. If $0<\alpha\le1$, then
\begin{equation}\label{eq:pure-low-small-order}
 s_1=\frac6{4+\alpha}>1,\qquad
 \|\nabla|a|^p\|_{s_1}\le pVS_a^{p-2}G,\qquad
 I_\alpha:L^{s_1}\to L^{q_1}.
\end{equation}
Hence
$\|(I_\alpha*\nabla|a|^p)\mathfrak b(a)\|_1
 \le C_\alpha VS_a^{2p-3}G$.
If $1<\alpha<3$, use instead
\[
 \nabla I_\alpha=c_\alpha(\alpha-3)x|x|^{\alpha-5},\qquad
 |\nabla I_\alpha|=d_\alpha I_{\alpha-1},\qquad
 d_\alpha=\frac{(3-\alpha)c_\alpha}{c_{\alpha-1}}.
\]
The small-sphere boundary is $O(\varepsilon^{\alpha-1})$, so this weak
derivative has no atom. Strict HLS gives
\begin{equation}\label{eq:pure-low-large-order}
 \||a|^p\|_{t_0}\le VS_a^{p-1},\qquad
 I_{\alpha-1}:L^{t_0}\to L^{q_1},\qquad
 \|((\nabla I_\alpha)*|a|^p)\mathfrak b(a)\|_1
 \le C_\alpha VS_a^{2p-2}.
\end{equation}
At $\alpha=1$ use \eqref{eq:pure-low-small-order}, not \eqref{eq:pure-low-large-order}.

For the weak identities, take $f_n\in C_c^\infty$ with $f_n\to f$ in
$\X$ and set $a_n=P_{\le M}f_n$. Lemma~\ref{ann:cutoff-realization}
gives $a_n\to a$ in $L^6\cap L^\infty$ and
$\nabla a_n\to\nabla a$ in $L^2$.
The powers of $a_n$ are $C^1$ with rapidly decaying first derivatives,
including at their zeros. Compact-test integration by parts differentiates
the density for $\alpha\le1$ and the kernel for $\alpha>1$.
The power-difference inequalities in Lemma~\ref{local:nonlinear} and
\eqref{eq:pure-low-small-order}--\eqref{eq:pure-low-large-order} give
\begin{equation}\label{eq:pure-low-derivative-limits}
 \begin{gathered}
 F(a_n)\to F(a)\text{ in }L^{6/5},\qquad
 \mathcal A(a_n)\to\mathcal A(a)\text{ in }L^{q_0},\\
 \nabla\mathfrak b(a_n)\to D\mathfrak b(a)\nabla a\text{ in }L^{s_0},\qquad
 \mathfrak b(a_n)\to\mathfrak b(a)\text{ in }L^{t_0},\\
 \nabla\mathcal A(a_n)\longrightarrow
 \begin{cases}
 I_\alpha*\nabla|a|^p,&\alpha\le1,\\
 (\nabla I_\alpha)*|a|^p,&\alpha>1
 \end{cases}
 \quad\text{in }L^{q_1},\\
 \nabla F(a_n)=(\nabla\mathcal A(a_n))\mathfrak b(a_n)
                 +\mathcal A(a_n)\nabla\mathfrak b(a_n).
 \end{gathered}
\end{equation}
The two product terms in \eqref{eq:pure-low-derivative-limits} converge in $L^1$; compact tests identify
their sum as $\nabla F(a)$.
Using $S_a\le\sqrt{S_*}G$ proves \eqref{ann:pure-low-bound}.

For $r_L(\vartheta)=r(\vartheta/L)$, let
$a_{j,L}=-i\vartheta_j|\vartheta|^{-2}r_L$, extended by zero near zero.
Its inverse kernel satisfies
\[
 \begin{gathered}
 A_{j,L}(x)=L^2A_{j,1}(Lx),\qquad
 \|A_{j,L}\|_1=L^{-1}\|A_{j,1}\|_1,\\
 \sum_ja_{j,L}i\vartheta_j=r_L,\qquad
 R_LF(a)=\sum_jA_{j,L}*\partial_jF(a)\quad\text{in }\mathcal S',\\
 \|R_LF(a)\|_1\le C_RL^{-1}\|\nabla F(a)\|_1.
 \end{gathered}
\]
Young gives the $L^1$ representative and \eqref{eq:annular-inverse-gradient-estimate}. Its distribution agrees with the
ordinary annular $L^{6/5}$ convolution, so the two representatives coincide
a.e.
\end{proof}

When $2<\alpha<3$, the mixed density can lie only in $L^1$ at
the required allocation. A translated Hardy estimate supplies the
missing bound while preserving the annular inverse-derivative gain.

\subsection{A weighted mixed derivative for \texorpdfstring{$2<\alpha<3$}{2 < alpha < 3}}

\begin{lemma}[Weighted mixed-gradient replacement]
\label{ann:weighted-gradient}
Let $2<\alpha<3$, $d\in L^1\cap L^{s_0}$, and
$a=P_{\leq M}f$ for $f\in\X$. The actual product
$T=(I_\alpha*d)\mathfrak b(a)$ belongs to $L^{6/5}$ and has a
weak gradient in $L^1$, satisfying
\begin{equation}\label{ann:weighted-gradient-bound}
 \|\nabla T\|_1\leq C_\alpha\|d\|_1
          \|a\|_\infty^{4-\alpha}\|\nabla a\|_2^{2\alpha-2}.
\end{equation}
The annular inverse-gradient conclusion of
Lemma~\ref{ann:pure-low} applies to this product as well.
\end{lemma}
\begin{proof}
Put $V=\|a\|_\infty$, $G=\|\nabla a\|_2$ and $C_S=\sqrt{S_*}$.
Lemma~\ref{std:hardy} gives
\begin{equation}\label{ann:weighted-hardy}
 \sup_y\int |x-y|^{-\gamma}|a(x)|^{6-2\gamma}\,dx
 \le2^\gamma C_S^{6-3\gamma}G^{6-2\gamma},\qquad 0\le\gamma\le2.
\end{equation}
Taking $\gamma=4-\alpha\in(1,2)$ and extracting $V^{4-\alpha}$ yields
\begin{equation}\label{eq:weighted-potential-derivative}
 \sup_y\int|\nabla I_\alpha(x-y)|\,|a(x)|^{p-1}\,dx
 \le C_\alpha V^{4-\alpha}G^{2\alpha-2}.
\end{equation}
For the local derivative, Cauchy--Schwarz and
\eqref{ann:weighted-hardy} with $\gamma=6-2\alpha\in(0,2)$ give
\begin{equation}\label{eq:weighted-local-derivative}
 \begin{aligned}
 \int I_\alpha(x-y)|a|^{p-2}|\nabla a|\,dx
 &\le c_\alpha G
 \left(\int|x-y|^{-(6-2\alpha)}|a|^{2\alpha+2}\,dx\right)^{1/2}\\
 &\le C_\alpha V^{4-\alpha}G^{2\alpha-2}.
 \end{aligned}
\end{equation}
In \eqref{eq:weighted-local-derivative}, the extracted power is $V^{2(4-\alpha)}$ before
taking the square root. Apply Tonelli to
\eqref{eq:weighted-potential-derivative}--\eqref{eq:weighted-local-derivative} against $|d(y)|$; together with
$|\nabla\mathfrak b(a)|\le(p-1)|a|^{p-2}|\nabla a|$, this proves
\eqref{ann:weighted-gradient-bound}.

Choose $d_n\in C_c^\infty$ converging to $d$ in $L^1\cap L^{s_0}$,
and set $T_n=(I_\alpha*d_n)\mathfrak b(a)$.
Local integrability of $I_\alpha,\nabla I_\alpha$, compact-test Fubini
and the $C^1$ product rule give their weak derivatives.
Strict HLS and \eqref{eq:weighted-potential-derivative}--\eqref{eq:weighted-local-derivative} imply
\[
 T_n\to T\text{ in }L^{6/5},\qquad
 \|\nabla T_n-\nabla T_m\|_1
 \le C_\alpha\|d_n-d_m\|_1V^{4-\alpha}G^{2\alpha-2}\longrightarrow0.
\]
Compact tests identify the limiting gradient and its bound.
Lemma~\ref{ann:pure-low} supplies the annular inverse-gradient conclusion.
\end{proof}

Thus the mixed term is controlled without invoking a strong HLS
estimate with input exponent one.

\subsection{Finite frequency envelopes}

The envelope argument must begin with quantities known to be finite;
otherwise absorption could amount to subtracting infinities. Fix
$f\in\X$, a scale $L_0>0$, and put
\[
 v=P_{\leq L_0}f,\quad w=f-v,\quad B_f=\|\nabla f\|_2,
 \quad\eta=\|\nabla v\|_2.
\]
For each output $L>0$ define
\begin{equation}\label{ann:envelope-definitions}
 \begin{gathered}
 b(M)=M^{-1}\|P_Mf\|_\infty,\qquad s_L=L/100,\qquad
 r_L=P_{\leq s_L}v,\qquad z_L=v-r_L,\\
 H_L=\sup_{M\geq L/200}b(M),\qquad
 D_L=\frac1{100}\sum_{n\geq0}2^{-n}b(s_L2^{-n}).
 \end{gathered}
\end{equation}
The letter $r_L$ here is a low-pass \emph{function}, not the physical
scale $r(s)$ of \eqref{ann:selected-path}; it will always carry its
subscript $L$.

We first bound each envelope quantity at a fixed scale.

\begin{lemma}[Finite-scale allocation]
\label{ann:finite-envelopes}
The quantities $H_L,D_L$ are finite for each fixed $L$, and
\begin{equation}\label{eq:finite-envelope-allocation}\begin{aligned}
 \|w\|_2&\leq L_0^{-1}B_f,&\quad\|w\|_3&\leq C L_0^{-1/2}B_f,\\
 \|z_L\|_2&\leq s_L^{-1}\eta,&
 \|z_L\|_6+\|r_L\|_6&\leq C\eta,\\
 \|\nabla z_L\|_2+\|\nabla r_L\|_2&\leq2\eta,&
 \|z_L\|_4^2&\leq C\eta H_L,\\
 \|r_L\|_\infty&\leq C L D_L,&
 \|r_L\|_\infty&\leq C\sqrt{s_L}\eta.
\end{aligned}\end{equation}
\end{lemma}
\begin{proof}
The annular bound gives
\begin{equation}\label{eq:raw-envelope-finiteness}
 b(M)\le CB_fM^{-1/2},\qquad H_L\le CB_f(L/200)^{-1/2},\qquad
 D_L\le\frac{CB_f}{100\sqrt{s_L}}\sum_{n\ge0}2^{-n/2}<\infty.
\end{equation}
Lemma~\ref{ann:cutoff-realization} and Sobolev give the $L^2,L^6$ and
gradient bounds. Interpolation gives
$\|w\|_3\le\|w\|_2^{1/2}\|w\|_6^{1/2}$; in particular
$z_L\in L^2\cap L^6\subset L^4$.
For $M\le s_L/2$, $P_Mz_L=0$. Otherwise
$P_Mz_L=P_{>s_L}P_{\le L_0}P_Mf$, so, with $K_0=\|k\|_1$,
\begin{equation}\label{eq:middle-annular-seminorm}
 \sup_{M>0}M^{-1}\|P_Mz_L\|_\infty
 \le K_0(1+K_0)H_L.
\end{equation}
Apply Lemma~\ref{std:refined} to \eqref{eq:middle-annular-seminorm}:
\[
 \|z_L\|_4^2\le C H_L\|\nabla z_L\|_2\le C\eta H_L.
\]
For the remaining bound, the annular telescope converges absolutely in
$L^\infty$ by \eqref{eq:raw-envelope-finiteness}; its low-frequency endpoint tends to zero
by Lemma~\ref{ann:cutoff-realization}. Hence
\[
 P_{\le s_L}f=\sum_{n\ge0}P_{s_L2^{-n}}f,\qquad
 \|r_L\|_\infty\le K_0\sum_{n\ge0}s_L2^{-n}b(s_L2^{-n})=K_0LD_L.
\]
The direct lowpass bound applied to $v$ gives
$\|r_L\|_\infty\le C\sqrt{s_L}\eta$.
\end{proof}

We split the forcing at two scales: $L_0$ isolates a part with finite
$L^2$ norm, while $L/100$ separates the lowest frequencies from the
output annulus. The remaining bounds will be linear in $H_L+D_L$.

\begin{lemma}[Outer mixed source]
\label{ann:outer-source}
With \eqref{ann:envelope-definitions}, for $0<L\leq L_0$ there is an
actual $\Gamma_L^{\mathrm{out}}\in L^1\cap L^{6/5}$ such that
\begin{equation}\label{eq:outer-source-allocation}
 P_L\Gamma_L^{\mathrm{out}}=P_L(F(f)-F(v)),\qquad
 \|\Gamma_L^{\mathrm{out}}\|_1
 \leq C L_0^{-1/2}B_f^{2p-1}
       +C\mathbf1_{\{\alpha>2\}}B_f^p\eta^{p-2}(H_L+D_L).
\end{equation}
For $0<\alpha\leq2$ the unprojected difference itself may be used.
\end{lemma}
\begin{proof}
Put $d=|f|^p-|v|^p$ and split
\[
 F(f)-F(v)=G_1+G_2,
 \quad G_1=\mathcal A(f)(\mathfrak b(f)-\mathfrak b(v)),
 \quad G_2=(I_\alpha*d)\mathfrak b(v).
\]
Both terms belong to $L^{6/5}$ by Lemma~\ref{ann:baseline}.
Since $1/3+(p-2)/6=p/6$, the power difference bound and
Lemma~\ref{ann:finite-envelopes} yield
\begin{equation}\label{eq:outer-local-factor}
 \|G_1\|_1\leq C\|f\|_6^p\|w\|_3
                          (\|f\|_6+\|v\|_6)^{p-2}
       \leq C L_0^{-1/2}B_f^{2p-1}.
\end{equation}
Here $s_0,q_0$ are conjugate. Always $\|d\|_{s_0}\leq C B_f^p$.

The potential term has three different allocations. If $0<\alpha<2$,
put $t_\alpha=6/(4+\alpha)>1$.
The allocation $1/3+(p-1)/6=(4+\alpha)/6$ gives
$\|d\|_{t_\alpha}\leq C\|w\|_3 B_f^{p-1}$.
This is a strict HLS input, and its output exponent is $6/(4-\alpha)$,
conjugate to $t_0=6/(2+\alpha)$. Thus
$\|G_2\|_1\leq C L_0^{-1/2}B_f^{2p-1}$.

If $\alpha=2$, $p=5$ and the same allocation gives
$\|d\|_1\leq C\|w\|_3B_f^4$.
Translated Hardy and Cauchy--Schwarz give
\begin{equation}\label{ann:hardy-endpoint}
 \sup_y\int |x-y|^{-1}|v(x)|^4\,dx
 \leq \sup_y\left(\int|x-y|^{-2}|v|^2\right)^{1/2}
                  \left(\int|v|^6\right)^{1/2}
 \leq2\|\nabla v\|_2\|v\|_6^3.
\end{equation}
Tonelli against $|d(y)|$ now proves the same $G_2$ bound. This is not
an HLS claim with input exponent one.

Suppose $2<\alpha<3$ and set $\delta=(3-\alpha)/2$.
For $q=6/(4-\alpha)$ one has
$1/q=\delta/2+(1-\delta)/6$ and $1/q+(p-1)/6=1$.
Consequently
\begin{equation}\label{eq:outer-density-integrability}
 \|d\|_1\leq C\|w\|_qB_f^{p-1}
       \leq C L_0^{-\delta}B_f^p,
 \qquad \|d\|_{s_0}\leq C B_f^p.
\end{equation}
Use the pointwise decomposition
\[
 \mathfrak b(v)=\mathfrak b(z_L)+\mathfrak b(r_L)+R_B,
 \qquad |R_B|\leq C_p|r_L||z_L|(|r_L|+|z_L|)^{p-3}.
\]
The cross bound follows from the two-variable fundamental theorem and
$\|D^2\mathfrak b(z)\|\le C_p|z|^{p-3}$, valid at zero since $p>3$.

The exact allocations yield
\begin{equation}\label{eq:outer-recipient-factors}\begin{aligned}
 \|\mathfrak b(z_L)\|_{s_0}
   &\leq\|z_L\|_4^2\|z_L\|_6^{p-3}
        \leq C\eta^{p-2}H_L,\\
 \|R_B\|_{s_0}
   &\leq C\|r_L\|_\infty\|z_L\|_2
                   (\|r_L\|_6+\|z_L\|_6)^{p-3}
        \leq C\eta^{p-2}D_L.
\end{aligned}\end{equation}
The bounds follow from \eqref{eq:finite-envelope-allocation} and the exponent identities
\[
 \frac1{s_0}=\frac p6=\frac24+\frac{p-3}{6}
                =\frac12+\frac{p-3}{6},\qquad \frac L{s_L}=100.
\]
Hence \eqref{eq:outer-local-factor} and the factor bounds give
\[
 \mathcal R_L=G_1+(I_\alpha*d)(\mathfrak b(z_L)+R_B)\in L^1\cap L^{6/5},
 \quad
 \|\mathcal R_L\|_1\leq C L_0^{-1/2}B_f^{2p-1}
                  +C B_f^p\eta^{p-2}(H_L+D_L).
\]
Only $T_L=(I_\alpha*d)\mathfrak b(r_L)$ still needs an annular replacement.
Apply Lemma~\ref{ann:weighted-gradient} in the form \eqref{ann:weighted-gradient-bound}, with \eqref{eq:outer-density-integrability}, and then \eqref{eq:annular-inverse-gradient-estimate}, to obtain
\begin{equation}\label{eq:outer-annular-gradient-bound}
 \|Q_LT_L\|_1\leq C L^{-1}L_0^{-\delta}B_f^p
                    \|r_L\|_\infty^{4-\alpha}\eta^{2\alpha-2}.
\end{equation}
In \eqref{eq:outer-annular-gradient-bound}, retain one factor $\|r_L\|_\infty/L\leq C D_L$, and bound the remaining
power $3-\alpha$ by $(C\sqrt{s_L}\eta)^{3-\alpha}$. Thus
\begin{equation}\label{eq:outer-scale-factor}
 \|Q_LT_L\|_1\leq C B_f^p\eta^{p-2}(L/L_0)^\delta D_L.
\end{equation}
In \eqref{eq:outer-scale-factor}, $(2\alpha-2)+(3-\alpha)=p-2$ and $s_L^\delta/L_0^\delta$ accounts for the scale factor. Set
$\Gamma_L^{\mathrm{out}}=\mathcal R_L+Q_LT_L$.
This belongs to both stated Lebesgue spaces, and
$m_Lq_L=m_L$ gives the exact projected identity.
For $L\leq L_0$ the factor $(L/L_0)^\delta$ is bounded by one,
which proves the assertion. No claim that $T_L\in L^1$ was made.
\end{proof}

For the middle difference, every remaining gradient factor is small.
The same three allocations therefore replace the outer bound by a
coefficient of size $\eta^{2p-2}$.

\begin{lemma}[Middle source]
\label{ann:middle-source}
For every $L>0$ there is an actual
$\Gamma_L^{\mathrm{mid}}\in L^1\cap L^{6/5}$ with
\begin{equation}\label{eq:middle-source-allocation}
 P_L\Gamma_L^{\mathrm{mid}}=P_L(F(v)-F(r_L)),\qquad
 \|\Gamma_L^{\mathrm{mid}}\|_1
       \leq C\eta^{2p-2}(H_L+D_L).
\end{equation}
\end{lemma}
\begin{proof}
Put $d_m=|v|^p-|r_L|^p$ and split
\[
 F(v)-F(r_L)=G_{m,1}+G_{m,2},\quad
 G_{m,1}=\mathcal A(v)(\mathfrak b(v)-\mathfrak b(r_L)),\quad
 G_{m,2}=(I_\alpha*d_m)\mathfrak b(r_L).
\]
The estimates \eqref{eq:outer-recipient-factors}, applied to the difference $\mathfrak b(v)-\mathfrak b(r_L)$, give
\[
 \|\mathcal A(v)\|_{q_0}\le C\eta^p,\qquad
 \|G_{m,1}\|_1\le C\eta^{2p-2}(H_L+D_L).
\]
For the density, integrate
$\|D^2(|\cdot|^p)(z)\|\le C_p|z|^{p-2}$ on the unit square:
\[
 d_m=|z_L|^p+R_\rho,\qquad
 |R_\rho|\le C_p|r_L||z_L|(|r_L|+|z_L|)^{p-2}.
\]
If $0<\alpha\le2$, let $t_\alpha=6/(4+\alpha)\ge1$. The allocation
\[
 \frac1{t_\alpha}=\frac{p+1}{6}
 =\frac24+\frac{p-2}{6}=\frac12+\frac{p-2}{6}
\]
yields
\begin{equation}\label{eq:middle-density-factors}\begin{aligned}
 \||z_L|^p\|_{t_\alpha}
 &\le\|z_L\|_4^2\|z_L\|_6^{p-2}
       \le C\eta^{p-1}H_L,\\
 \|R_\rho\|_{t_\alpha}
 &\le C\|r_L\|_\infty\|z_L\|_2
              (\|r_L\|_6+\|z_L\|_6)^{p-2}
       \le C\eta^{p-1}D_L.
\end{aligned}\end{equation}
Apply \eqref{eq:middle-density-factors}: for $\alpha<2$, use strict HLS
$L^{t_\alpha}\to L^{6/(4-\alpha)}$ and H\"older with $L^{t_0}$.
At $\alpha=2$, use \eqref{ann:hardy-endpoint} for $r_L$ and Tonelli.
In both cases,
\[
 \|G_{m,2}\|_1\le C\eta^{2p-2}(H_L+D_L),\qquad
 \Gamma_L^{\mathrm{mid}}=G_{m,1}+G_{m,2}.
\]

For $2<\alpha<3$, set $\delta=(3-\alpha)/2$.
Interpolating $z_L$ between $L^2$ and $L^6$ at $q=6/(4-\alpha)$ gives
\begin{equation}\label{eq:middle-density-integrability}
 \|d_m\|_1\le Cs_L^{-\delta}\eta^p,\qquad
 \|d_m\|_{s_0}\le C\eta^p.
\end{equation}
Use \eqref{eq:middle-density-integrability}, \eqref{ann:weighted-gradient-bound}, and \eqref{eq:annular-inverse-gradient-estimate} to obtain
\begin{equation}\label{eq:middle-annular-gradient-bound}
 \begin{aligned}
 \|Q_LG_{m,2}\|_1
 &\le CL^{-1}s_L^{-\delta}\eta^p
                  \|r_L\|_\infty^{4-\alpha}\eta^{2\alpha-2}\\
 &\le C s_L^{-\delta}\eta^pD_L
                  s_L^\delta\eta^{3-\alpha}\eta^{2\alpha-2}
 =C\eta^{2p-2}D_L.
 \end{aligned}
\end{equation}
In \eqref{eq:middle-annular-gradient-bound} we used $\|r_L\|_\infty/L\le CD_L$ and
$\|r_L\|_\infty\le C\sqrt{s_L}\eta$.
Set $\Gamma_L^{\mathrm{mid}}=G_{m,1}+Q_LG_{m,2}$.
Baseline HLS puts every term in $L^{6/5}$, while $P_LQ_L=P_L$ gives
the projected identity.
\end{proof}

The pure-low leakage estimate completes the decomposition. The
result is an integrable source with exactly the same annular projection
as $F(f)$; no $L^1$ bound for the entire unprojected forcing is assumed.

\begin{theorem}[The full annular source]
\label{ann:full-source}
For $f\in\X$, the definitions \eqref{ann:envelope-definitions}, and
$0<L\leq L_0$, there is an actual $\Gamma_L\in L^1\cap L^{6/5}$ such that
\begin{equation}\label{ann:full-source-bound}
 \begin{split}
 P_L\Gamma_L&=P_LF(f)\quad\hbox{in }\mathcal S',\\
 \|\Gamma_L\|_1&\leq K+\kappa(H_L+D_L),\\
 K&=C L_0^{-1/2}B_f^{2p-1},\qquad
 \kappa=C\left(\eta^{2p-2}
       +\mathbf1_{\{\alpha>2\}}B_f^p\eta^{p-2}\right).
 \end{split}
\end{equation}
In particular $P_LF(f)$ has an actual $L^1$ representative. 
\end{theorem}
\begin{proof}
Using the same $f,v,r_L$ in Lemmas~\ref{ann:outer-source} and
\ref{ann:middle-source}, put
\[
 \Gamma_L=\Gamma_L^{\mathrm{out}}+\Gamma_L^{\mathrm{mid}}+Q_LF(r_L).
\]
Lemma~\ref{ann:pure-low} and the finite-scale allocation give
\begin{equation}\label{eq:pure-low-envelope-allocation}
 \|Q_LF(r_L)\|_1
 \le CL^{-1}\|r_L\|_\infty\|\nabla r_L\|_2^{2p-2}
 \le C\eta^{2p-2}D_L.
\end{equation}
All three terms lie in $L^1\cap L^{6/5}$, and
\[
 P_L\Gamma_L
 =P_L(F(f)-F(v))+P_L(F(v)-F(r_L))+P_LF(r_L)=P_LF(f).
\]
Summing \eqref{eq:outer-source-allocation}, \eqref{eq:middle-source-allocation}, and \eqref{eq:pure-low-envelope-allocation} proves \eqref{ann:full-source-bound}.
The $L^1$ kernel of $P_L$ gives $P_L\Gamma_L\in L^1$; equality of
distributions identifies it with the annular $L^{6/5}$ representative.
\end{proof}

We now compare its time integral with the no-waste limit.

\subsection{The annular time integral and envelope absorption}

\begin{proposition}[Comparison of the two terminal limits]
\label{ann:annular-bochner}
Fix $L>0$. Suppose the same path $U$ in \eqref{ann:selected-path} has,
for every $s\geq0$, an actual function $\Gamma_L(s)\in L^1$ with
\[
 P_L\Gamma_L(s)=P_LF(U(s))\text{ in }\mathcal S',\qquad
 \sup_{s\geq0}\|\Gamma_L(s)\|_1\leq B_L<\infty.
\]
Then, for every fixed $t\geq0$,
\begin{equation*}
 P_LU(t)=-i\sigma\int_t^\infty S(\sigma(t-s))P_LF(U(s))\,ds
 \quad\hbox{in }C_0(\mathbb R^3),\qquad
 \|P_LU(t)\|_\infty\leq C_\varphi L B_L.
\end{equation*}
The improper integral is absolutely Bochner integrable in $C_0$.
Its finite truncations are also actual $L^2$ Bochner integrals, and
converge weakly in $L^2$ to $P_LU(t)$ after multiplication by $-i\sigma$.
No measurability hypothesis on the chosen map $s\mapsto\Gamma_L(s)$
is needed.
\end{proposition}
\begin{proof}
The annular kernel satisfies
$\|K_L\|_{3/2}=CL$, $\|K_L\|_6=CL^{5/2}$ and $\|K_L\|_1=C$.
Young and compact smooth approximation give compatible maps
\[
 \|P_Lg\|_2\le CL\|g\|_{6/5},\qquad
 \|P_Lg\|_\infty\le CL^{5/2}\|g\|_{6/5},\qquad P_Lg\in C_0.
\]
The kernel of $S(\tau)P_L$ is continuous in $\tau$ in every Schwartz
seminorm on compact time intervals. Since $F(U)\in C_tL^{6/5}$,
\[
 h_{L,t}(s):=S(\sigma(t-s))P_LF(U(s))\in C_s(L^2\cap C_0).
\]
The finite Bochner integrals in these spaces have the same distribution.
Applying $P_L$ to \eqref{ann:finite-signed} gives
\begin{equation}\label{eq:annular-finite-duhamel}
 P_LU(t)=S(\sigma(t-T))P_LU(T)-i\sigma\int_t^Th_{L,t}(s)\,ds
 \quad\text{in }L^2.
\end{equation}
By Proposition~\ref{ann:no-waste} and boundedness of $P_L:\X\to L^2$,
the terminal term in \eqref{eq:annular-finite-duhamel} tends weakly to zero.

The projected source identity estimates this already continuous integrand:
\begin{equation}\label{eq:annular-time-majorant}
 h_{L,t}(s)=K_{L,\sigma(t-s)}*\Gamma_L(s),\qquad
 \|h_{L,t}(s)\|_\infty
 \le C\min(L^3,|t-s|^{-3/2})B_L.
\end{equation}
No measurability of $\Gamma_L$ is needed.
Integrating the kernel bound \eqref{eq:integrable-annular-kernel} gives
\[
 \int_t^\infty\|h_{L,t}(s)\|_\infty\,ds\le3CLB_L.
\]
By \eqref{eq:annular-finite-duhamel} and the integrable majorant \eqref{eq:annular-time-majorant}, along the single sequence $T_n=n+1$, put
$I_n=-i\sigma\int_t^{T_n}h_{L,t}(s)\,ds$. Then
\begin{equation}\label{eq:annular-two-limits}
 I_n\rightharpoonup P_LU(t)\text{ in }L^2,\qquad
 I_n\longrightarrow -i\sigma\int_t^\infty h_{L,t}(s)\,ds
 \text{ in }C_0.
\end{equation}
The limits in \eqref{eq:annular-two-limits} embed in $\mathcal S'$, so they agree. Their continuous
representatives give the claimed $C_0$ identity and norm bound.
\end{proof}

Integrating the annular kernel costs one power of $L$, which cancels
the normalization in $b(L)$. This produces a recurrence with a small
coefficient; a finite weighted supremum makes its absorption legitimate.

\begin{theorem}[The finite envelope and spatial integrability]
\label{ann:orbit-l4}
The selected actual path satisfies
\[
 A_*:=\sup_{M>0}M^{-1}\sup_{s\geq0}\|P_MU(s)\|_\infty<\infty,
 \qquad
 U\in C([0,\infty);L^4)\cap L^\infty([0,\infty);L^4),
\]
\[
 F(U)\in C([0,\infty);L^1)\cap L^\infty([0,\infty);L^1),
\]
\[
 \sup_s\|U(s)\|_4^4\leq C A_*^2 B^2,
 \quad\|F(U(s))\|_1\leq C\|U(s)\|_4^2\|U(s)\|_6^{2p-3}.
\]

\end{theorem}
\begin{proof}
The case $B=0$ is immediate. Put
\begin{equation}\label{eq:raw-time-envelope}
 b(M)=M^{-1}\sup_s\|P_MU(s)\|_\infty\le CBM^{-1/2},\qquad
 \eta=\sup_s\|\nabla P_{\le L_0}U(s)\|_2,
\end{equation}
and define $H_L,D_L$ from these time suprema. Both are finite for each
fixed $L$. Theorem~\ref{ann:full-source} and
Proposition~\ref{ann:annular-bochner} treat $L\le L_0$; the bound \eqref{eq:raw-time-envelope} treats $L\ge L_0$. Thus
\begin{equation}\label{ann:recurrence}
 \begin{gathered}
 b(L)\le K_*+\kappa_*(H_L+D_L),\qquad L>0,\\
 K_*=CL_0^{-1/2}(B+B^{2p-1}),\qquad
 \kappa_*=C\bigl(\eta^{2p-2}
       +\mathbf1_{\{\alpha>2\}}B^p\eta^{p-2}\bigr).
 \end{gathered}
\end{equation}
Choose $\theta\in(1/2,1)$. The lower bound $\theta>1/2$
compensates for the known growth $b(L)\lesssim L^{-1/2}$ at zero;
the upper bound $\theta<1$ makes the dyadic low-frequency sum converge.
Set
\begin{equation}\label{eq:finite-weighted-envelope}
 w_\varepsilon(L)=\min(1,(L/\varepsilon)^\theta),\qquad
 A_\varepsilon=\sup_{L>0}w_\varepsilon(L)b(L)
 \le CB\varepsilon^{-1/2}<\infty.
\end{equation}
The comparison
$w_\varepsilon(L)/w_\varepsilon(M)\le\max(1,(L/M)^\theta)$ yields
\begin{equation}\label{eq:envelope-weight-comparison}
 w_\varepsilon(L)(H_L+D_L)\le C_\theta A_\varepsilon,\qquad
 C_\theta=200^\theta+
 \frac{100^{\theta-1}}{1-2^{-(1-\theta)}}.
\end{equation}
Using \eqref{ann:recurrence}, uniform low-frequency smallness and $p>3$ allow us to fix $L_0>0$,
independently of $\varepsilon$, with $\kappa_*C_\theta\le1/2$. Combining \eqref{ann:recurrence} and \eqref{eq:envelope-weight-comparison} gives
\begin{equation}\label{eq:finite-envelope-absorption}
 A_\varepsilon\le K_*+\tfrac12A_\varepsilon
 \quad\Longrightarrow\quad A_\varepsilon\le2K_*.
\end{equation}
For each $L$, take $\varepsilon=L$ in \eqref{eq:finite-envelope-absorption} to obtain $A_*\le2K_*<\infty$. Finiteness in \eqref{eq:finite-weighted-envelope} justifies the subtraction.

Apply Lemma~\ref{std:refined} directly to the original $U(s)$ and to its
differences, whose annular seminorm is at most $2A_*$:
\[
 \|U(s)\|_4^4\le CA_*^2B^2,\qquad
 \|U(s)-U(s_0)\|_4^2
 \le2CA_*\|\nabla U(s)-\nabla U(s_0)\|_2\longrightarrow0.
\]
For $q_B=6(p-1)/p\in(4,5)$, interpolation and strict HLS give
\begin{equation}\label{eq:hartree-l1-interpolation}
 \begin{aligned}
 \|\mathfrak b(f)\|_{s_0}
 &=\|f\|_{q_B}^{p-1}\le\|f\|_4^2\|f\|_6^{p-3},\\
 \|F(f)\|_1
 &\le\|\mathcal A(f)\|_{q_0}\|\mathfrak b(f)\|_{s_0}
 \le C\|f\|_4^2\|f\|_6^{2p-3}.
 \end{aligned}
\end{equation}
If $f_n\to f$ in $L^4\cap L^6$, the power differences give
$|f_n|^p\to|f|^p$ in $L^{s_0}$ and, by $L^{q_B}$ interpolation,
$\mathfrak b(f_n)\to\mathfrak b(f)$ in $L^{s_0}$.
The two-term identity \eqref{eq:hartree-two-factor-subtraction}, estimated as in \eqref{eq:hartree-l1-interpolation}, gives
$F(f_n)\to F(f)$ in $L^1$. Applied to $U$, this proves the asserted
continuity and uniform bounds for the same representatives.
\end{proof}

These spatial bounds supply the highpass equation with a continuous source.

\begin{corollary}[The highpass equation and its source]
\label{ann:source-interface}
For the same $U$ and every fixed $\nu>0$, let
\[
 h_\nu=P_{>\nu}U,\qquad l_\nu=P_{\leq\nu}U,\qquad
 e_\nu=P_{>\nu}F(U)-F(h_\nu).
\]
Then these are actual functions,
$h_\nu\in C_tH^1$, $e_\nu\in C_t(L^1\cap L^{6/5})$ with uniform
time bounds in both source spaces, and
\begin{equation}\label{ann:forced-equation}
 i\partial_s h_\nu+\sigma\Delta h_\nu+\sigma F(h_\nu)
       =-\sigma e_\nu
       \quad\hbox{in distributions}.
\end{equation}
The error obeys
\[
 \sup_{s\geq0}\|e_\nu(s)\|_{6/5}
 \leq C B^{2p-2}\eta_\nu
       +C\nu^{1/2}\sup_{s\geq0}\|F(U(s))\|_1
       \longrightarrow0.
\]

\end{corollary}
\begin{proof}
The identity
\[
 e_\nu=F(U)-F(h_\nu)-P_{\le\nu}F(U)
\]
holds for actual functions. The cutoff bounds on $L^4$ and $L^6$,
Theorem~\ref{ann:orbit-l4}, Lemma~\ref{ann:baseline}, and continuity of
$F:L^4\cap L^6\to L^1$ give the asserted continuity and uniform source
bounds. By Sobolev and the nonlinear difference estimate,
\begin{equation}\label{eq:highpass-input-error}
 \|F(U)-F(h_\nu)\|_{6/5}\le CB^{2p-2}\|l_\nu\|_6
 \le CB^{2p-2}\eta_\nu.
\end{equation}
The low-pass kernel satisfies $\|k_\nu\|_{6/5}=C\nu^{1/2}$, so Young's
inequality gives
\begin{equation}\label{eq:highpass-output-error}
 \|P_{\le\nu}F(U)\|_{6/5}\le C\nu^{1/2}\|F(U)\|_1.
\end{equation}
The sum of \eqref{eq:highpass-input-error} and \eqref{eq:highpass-output-error} tends uniformly to zero by Proposition~\ref{ann:uniform-highpass}.
Applying the same multiplier to the signed equation
$iU_s+\sigma\Delta U+\sigma F(U)=0$ proves
\eqref{ann:forced-equation}. Applying it to the finite mild identity
identifies the same endpoint representatives. No $L^1$ smallness is used.
\end{proof}

The forcing is now bounded in $L^1$, while the projected defect is
small in $L^{6/5}$. The interaction argument must use these two facts
separately.

\section{Mass-independent Gaussian coercivity}
\label{sec:gaussian}
The next task is a positive interaction bound whose constants do not
deteriorate with the high-frequency mass. The proof first converts the
ground-state gap into a directional stress inequality, then separates a
compact core from a small gradient tail. Both nonlocal pressure terms
are retained, together with the full quadratic density form.
All forms in this section are evaluated on $h\in H^1(\R^3;\C)$,
as is valid for each fixed highpass field. The compact-family constants
will be independent of $\|h\|_2$.
Write $S_0=\frac13(4/\pi^2)^{2/3}$ and $C_6=S_0^3$; thus
$\int|v|^6\le C_6 A(v)^3$. These are the ordinary Sobolev constants,
not the Hartree threshold $G=A(W)$.

\subsection{Gaussian kernels and the principal interaction form}

We begin with the Gaussian forms on an $H^1$ field.

\begin{lemma}[Gaussian forms and their domains]\label{gauss:kernel}
For $\lambda>0$ set
\[
 \phi_\lambda(r)=\lambda^{-1}e^{-r^2/(2\lambda^2)},\qquad
 \xi_\lambda(r)=\frac1r\int_0^r\phi_\lambda(s)\dd s\quad(r>0),
 \qquad \xi_\lambda(0)=\lambda^{-1},
\]
\[
 k_\lambda(x)=\lambda\left[
 \frac{|x|}{\lambda}\int_0^{|x|/\lambda}e^{-s^2/2}\dd s
 -1+e^{-|x|^2/(2\lambda^2)}\right],\qquad
 w_\lambda=\xi_\lambda-\phi_\lambda.
\]
With $P_z=I-\widehat z\otimes\widehat z$ for $z\ne0$, define
\begin{equation}\label{eq:gaussian-hessian-laplacian}
 H_\lambda=D^2k_\lambda
   =\phi_\lambda I+w_\lambda P_z,\qquad
 L_\lambda=\Delta k_\lambda=\phi_\lambda+2\xi_\lambda.
\end{equation}
The continuous values at zero are $H_\lambda(0)=\lambda^{-1}I$
and $w_\lambda(0)=0$. For $h\in H^1$ let
\[
 \rho=|h|^2,\quad j_i=\im(\overline h\,\partial_i h),\quad
 T_{ij}=\re(\partial_i\overline h\,\partial_jh),\quad d=|h|^p,
 \qquad V=I_\alpha*d,
\]
\[
 q(x)=\int\nabla k_\lambda(x-y)\rho(y)\dd y,\qquad
 q_m(y)=\int\nabla k_\lambda(x-y)\cdot j(x)\dd x,
\]
\[
 C_q(x,z)=\frac{(q(x)-q(z))\cdot(x-z)}{|x-z|^2}.
\]
The diagonal value of $C_q$ is immaterial. Put
\[
 E_\lambda(h)=\int (L_\lambda*\rho)Vd,\qquad
 J_\lambda(h)=\iint I_\alpha(x-z)C_q(x,z)d(x)d(z)\dd x\dd z,
\]
\begin{equation}\label{eq:hartree-gaussian-pressure}
 \mathcal P_\lambda(h)=\frac{2(p-2)}p E_\lambda(h)
 +\frac{2(3-\alpha)}p J_\lambda(h),
\end{equation}
\[
 Q_\lambda(\rho)=-\iint\Delta L_\lambda(x-y)\rho(x)\rho(y)\dd x\dd y,
 \qquad M_\lambda(h)=2\int q\cdot j,
\]
\[
 \mathcal R_\lambda(h)=
 4\iint H_\lambda(x-y):[\rho(y)T(x)-j(x)\otimes j(y)]\dd x\dd y
 -\mathcal P_\lambda(h),\qquad
 B_\lambda(h)=Q_\lambda(\rho)+\mathcal R_\lambda(h).
\]
All displayed integrals are absolutely convergent. These definitions use
one full density in both pressure terms and in the current--current term.
For a real $g\in H^1$, writing
\[
 \mathcal H_\lambda(c;g)=\iint |c(y)|^2
       H_\lambda(x-y)\nabla g(x)\cdot\nabla g(x)\dd x\dd y,
\]
we have $\mathcal R_\lambda(g)=4\mathcal H_\lambda(g;g)-\mathcal P_\lambda(g)$.
\end{lemma}
\begin{proof}
Radial differentiation gives
\[
 \nabla k_\lambda(x)=\xi_\lambda(|x|)x,\qquad
 D^2k_\lambda=\phi_\lambda I+(\xi_\lambda-\phi_\lambda)P_x.
\]
The identities and bounds
\begin{equation}\label{eq:gaussian-kernel-pointwise-bounds}
 \xi_\lambda(r)=\int_0^1\phi_\lambda(tr)\dd t,\qquad
 |\xi_\lambda(r)-\lambda^{-1}|\le\frac{r^2}{2\lambda^3},\qquad
 0<\phi_\lambda\le\xi_\lambda\le\lambda^{-1},\qquad
 r\xi_\lambda(r)\le2
\end{equation}
in \eqref{eq:gaussian-kernel-pointwise-bounds} give the smooth extensions and $H_\lambda\ge0$.
At fixed $\lambda$, all required derivatives are bounded and
$\Delta L_\lambda$ is a polynomial times a Gaussian. Since
\begin{equation}\label{eq:gaussian-quadratic-domains}
 \rho,j,T\in L^1,\qquad \|j\|_1\le\|h\|_2\|\nabla h\|_2,\qquad
 |C_q|\le\|Dq\|_\infty\le\lambda^{-1}\|h\|_2^2,
\end{equation}
the density, kinetic and current pairings have integrable product
majorants. Theorem~\ref{std:hls} gives
\begin{equation}\label{eq:gaussian-pressure-domains}
 d\in L^{6/p},\qquad V\in L^{6/(3-\alpha)},\qquad
 \iint I_\alpha(x-z)d(x)d(z)\dd x\dd z<\infty,
\end{equation}
which controls both pressures. The majorants \eqref{eq:gaussian-quadratic-domains}--\eqref{eq:gaussian-pressure-domains} also justify Fubini and convolution differentiation. For real $g$, $j=0$.
\end{proof}

The physical density kernel changes sign. Its Fourier multiplier is
nonnegative, so positivity must be applied to the full density quadratic
form before estimating a tail.

\begin{lemma}[Fourier representations of the quadratic forms]\label{gauss:Q}
For real $\rho\in L^1\cap L^2$ let
$Q_{\phi,\lambda}(\rho)=-\int\rho\,\Delta(\phi_\lambda*\rho)$.
With the Fourier convention of the introduction and $\gamma=(2\pi)^{3/2}$,
\begin{align}
 Q_\lambda(\rho)&=\gamma\int
 (2+4\pi^2\lambda^2|\xi|^2)e^{-2\pi^2\lambda^2|\xi|^2}
 |\widehat\rho(\xi)|^2\dd\xi,\label{gauss:Q-fourier}\\
 Q_{\phi,\lambda}(\rho)&=\gamma\int
 4\pi^2\lambda^2|\xi|^2e^{-2\pi^2\lambda^2|\xi|^2}
 |\widehat\rho(\xi)|^2\dd\xi.\label{gauss:Qphi-fourier}
\end{align}
Consequently $0\le Q_{\phi,\lambda}\le Q_\lambda$. If
$\rho=\rho_c+\rho_t$, all three densities are real and in $L^1\cap L^2$,
and $\rho_c\ge0$, then
\begin{equation}\label{gauss:Q-tail}
 Q_{\phi,\lambda}(\rho_t)\le2Q_\lambda(\rho)
                  +10\lambda^{-3}\|\rho_c\|_1^2.
\end{equation}
\end{lemma}
\begin{proof}
Radial differentiation and Gaussian inversion give
\[
 -\Delta L_\lambda(z)=\lambda^{-3}
 (5-|z|^2/\lambda^2)e^{-|z|^2/(2\lambda^2)},\qquad
 -\Delta\phi_\lambda(z)=\lambda^{-3}
 (3-|z|^2/\lambda^2)e^{-|z|^2/(2\lambda^2)},
\]
\begin{equation}\label{eq:gaussian-quadratic-symbols}
 -\Delta L_\lambda=-\Delta\phi_\lambda+2\lambda^{-2}\phi_\lambda,
 \qquad \widehat\phi_\lambda(\xi)
 =\gamma\lambda^2e^{-2\pi^2\lambda^2|\xi|^2}.
\end{equation}
Both differentiated kernels lie in $L^1\cap L^\infty$.
Boundedness gives the physical $L^1$ pairings; applying Plancherel to \eqref{eq:gaussian-quadratic-symbols} gives
\eqref{gauss:Q-fourier}--\eqref{gauss:Qphi-fourier} on $L^2$.
In particular $\rho=|h|^2\in L^1\cap L^2$ when $h\in H^1$.

In the nonnegative multiplier of $Q_{\phi,\lambda}$, use
$|\widehat\rho_t|^2\le2|\widehat\rho|^2+2|\widehat\rho_c|^2$.
Then
\[
 \begin{aligned}
 Q_{\phi,\lambda}(\rho_t)
 &\le2Q_{\phi,\lambda}(\rho)+2Q_{\phi,\lambda}(\rho_c)\\
 &\le2Q_\lambda(\rho)+2Q_\lambda(\rho_c)
 \le2Q_\lambda(\rho)+10\lambda^{-3}\|\rho_c\|_1^2,
 \end{aligned}
\]
by $|\Delta L_\lambda|\le5\lambda^{-3}$. Positivity is used in
frequency space; the physical kernels change sign.
\end{proof}

The pressure depends only on the modulus. Symmetrizing the
current--current term shows that removing the phase can only decrease
the full interaction form.

\begin{lemma}[Phase reduction]\label{gauss:phase}
For every complex $h\in H^1$, $B_\lambda(|h|)\le
B_\lambda(h)$.
\end{lemma}
\begin{proof}
Put $g=|h|$ and, with value zero on $\{h=0\}$,
\[
 v_i=\frac{\re(\overline h\partial_i h)}{|h|},\qquad
 a_i=\frac{\im(\overline h\partial_i h)}{|h|}.
\]
For $g_\epsilon=(|h|^2+\epsilon^2)^{1/2}-\epsilon$,
\[
 g_\epsilon\to g\text{ in }L^2,\qquad
 \nabla g_\epsilon=
 \frac{\re(\overline h\nabla h)}{(|h|^2+\epsilon^2)^{1/2}}
 \to v\text{ in }L^2,
\]
by $|g_\epsilon|\le|h|$ and $|\nabla g_\epsilon|\le|\nabla h|$.
Weak derivative tests give $\nabla g=v$ and $A(g)\le A(h)$.
Set
\begin{equation}\label{eq:phase-tensor-decomposition}
 Z_{ij}=\mathbf1_{\{h=0\}}\re(\partial_i\overline h\partial_jh),
 \qquad T=v\otimes v+a\otimes a+Z,\qquad j=ga.
\end{equation}
Symmetrize \eqref{eq:phase-tensor-decomposition} with $b(x,y)=g(y)a(x)-g(x)a(y)$ yields
\begin{equation}\label{eq:phase-positive-remainder}
 \begin{aligned}
 B_\lambda(h)-B_\lambda(g)
 ={}&2\iint b(x,y)\cdot H_\lambda(x-y)b(x,y)\dd x\dd y\\
 &+4\iint g(y)^2H_\lambda(x-y):Z(x)\dd x\dd y\ge0.
 \end{aligned}
\end{equation}
The terms in \eqref{eq:phase-positive-remainder} are nonnegative because $H_\lambda\ge0$ and $Z$ is a sum of two Gram tensors.
Lemma~\ref{gauss:kernel} and $|a|,|v|\le|\nabla h|$ give
integrability. The density and pressures are unchanged; no zero-set
derivative assertion is needed.
\end{proof}

It is therefore enough to prove the lower bound for real nonnegative fields.

\subsection{Sharp directional stretching and the core stress}
For a unit vector $\omega$ put
\[
 A_\omega(u)=\int|\omega\cdot\nabla u|^2,\quad
 M_\omega(u)=\iint I_\alpha(x-y)|u(x)|^p|u(y)|^p
       \frac{(\omega\cdot(x-y))^2}{|x-y|^2}\dd x\dd y,
\]
\[
 P_\omega(u)=\frac{(1+\alpha)D(u)+(3-\alpha)M_\omega(u)}{2p}.
\]
These are genuine finite integrals on $\X$; $0\le M_\omega\le D$.

The scalar ground-state gap must control every spatial direction.
Stretching in one direction leaves the transverse gradient unchanged;
convexity of the transformed Riesz energy then converts the two threshold
inequalities into a positive directional stress.

\begin{lemma}[Directional sharp coercivity]\label{gauss:directional}
If $u\in\X$, $u\ne0$, $A(u)<G$, and $E(u)<E(W)$, then for every
unit $\omega$,
\begin{equation}\label{gauss:directional-gap}
 A_\omega(u)-P_\omega(u)
 \ge\frac{2(E(W)-E(u))}{G-A(u)}A_\omega(u)>0.
\end{equation}
Moreover
\begin{equation}\label{gauss:directional-quantitative}
 \int|u|^6\le8C_6 A(u)^2A_\omega(u).
\end{equation}
\end{lemma}
\begin{proof}
For $\ell>0$, set
$L_\ell x=x+(\ell-1)(\omega\cdot x)\omega$ and
$u_\ell(x)=\ell^{1/2}u(L_\ell x)$. Weak derivative tests and a
change of variables give
\begin{equation}\label{eq:directional-stretch-scaling}
 A(u_\ell)=A(u)+(\ell^2-1)A_\omega(u),\qquad
 \int|u_\ell|^6=\ell^2\int|u|^6.
\end{equation}
If $A_\omega=0$, Sobolev applied to \eqref{eq:directional-stretch-scaling} as $\ell\to\infty$ forces $u=0$.
Thus $A_\omega>0$; taking $\ell^2=A/A_\omega$ gives
\[
 \frac A{A_\omega}\int|u|^6
 \le C_6(2A-A_\omega)^3\le8C_6A^3,
\]
and proves \eqref{gauss:directional-quantitative}; the remaining argument establishes \eqref{gauss:directional-gap}.

To prove convexity of $s\mapsto D(u_{\sqrt s})$, let $d=|u|^p$
and choose bounded compactly supported $0\le d_n\uparrow d$.
The transformed physical kernel and its energy are
\[
 K_s(z)=c_\alpha s^2
 \big((\omega\cdot z)^2+s(|z|^2-(\omega\cdot z)^2)\big)^{(\alpha-3)/2},
\]
\begin{equation}\label{eq:stretched-riesz-fourier-energy}
 \begin{aligned}
 D_n(s)&=\iint K_s(x-y)d_n(x)d_n(y)\dd x\dd y\\
 &=(2\pi)^{-\alpha}\int
 \frac{s^{1+\alpha/2}}
 {[s(\omega\cdot\xi)^2+|\xi|^2-(\omega\cdot\xi)^2]^{\alpha/2}}
 |\widehat d_n(\xi)|^2\dd\xi.
 \end{aligned}
\end{equation}
The heat representation, first on compact heat-parameter intervals,
then Tonelli and monotone convergence, gives \eqref{eq:stretched-riesz-fourier-energy}; HLS gives
finiteness. For $\beta=\alpha/2$, $a,b\ge0$, $a+b>0$,
\begin{equation}\label{eq:directional-kernel-convexity}
 \frac{\dd^2}{\dd s^2}\frac{s^{1+\beta}}{(as+b)^\beta}
 =\beta(1+\beta)b^2\frac{s^{\beta-1}}{(as+b)^{\beta+2}}\ge0.
\end{equation}
Integrate the three-point inequality following from \eqref{eq:directional-kernel-convexity}. On compact positive
$s$ intervals, $K_s\le CI_\alpha$, so dominated convergence passes
convexity from $D_n$ to $D(u_{\sqrt s})$. Only the truncated
densities have been Fourier transformed.

The same Riesz majorant controls $\partial_sK_s$ near $1$; hence
\begin{equation}\label{eq:directional-energy-derivative}
 \left.\frac{\dd}{\dd s}D(u_{\sqrt s})\right|_{s=1}
 =\frac{1+\alpha}{2}D(u)+\frac{3-\alpha}{2}M_\omega(u)=pP_\omega(u).
\end{equation}
With $s_*=1+(G-A)/A_\omega>1$, convexity, \eqref{eq:directional-energy-derivative}, and Theorem~\ref{std:sharp} give
\begin{equation}\label{eq:directional-threshold-comparison}
 A(u_{\sqrt{s_*}})=G,\qquad
 D(u)+(s_*-1)pP_\omega(u)\le D(u_{\sqrt{s_*}})\le G.
\end{equation}
Rearranging \eqref{eq:directional-threshold-comparison} gives
\[
 A_\omega-P_\omega\ge
 \frac{A_\omega[(p-1)G-pA+D]}{p(G-A)}
 =\frac{2(E(W)-E(u))A_\omega}{G-A}>0.
\]
\end{proof}

To localize this gap, we first estimate the nonlocal pressure.

\subsection{The Gaussian weight and nonlinear localization}
Localizing the field changes both power densities in the Hartree
pressure. We estimate these changes together, using a weighted analogue
of the Ismagilov--Morgan--Simon (IMS) decomposition. The density that
defines the interaction center remains intact. For $c\in L^2$ put
\[
 U_{\lambda,c}(x)=\int\xi_\lambda(|x-y|)|c(y)|^2\dd y.
\]
In the following lemmas the center is never truncated. The notation
$\mathcal P_\lambda(c;u)$ means the pressure in \eqref{eq:hartree-gaussian-pressure}, with $q=\nabla k_\lambda*|c|^2$ and the
power densities belonging to $u$. Thus $\mathcal P_\lambda(h;h)
=\mathcal P_\lambda(h)$.

The weight behaves like an average of $(\lambda+|x-y|)^{-1}$.
Its uniform $A_1$ bound is the reason weighted maximal and Riesz
estimates can be used without a factor depending on $\|c\|_2$.

\begin{lemma}[A uniform $A_1$ weight and a weighted Riesz bound]
\label{gauss:weighted}
Either $c=0$ a.e. and $U_{\lambda,c}=0$, or $U=U_{\lambda,c}$ is
positive and continuous everywhere and, for every ball $B$ and $x\in B$,
\[
 \frac1{|B|}\int_B U\le1296 U(x).
\]
Let $q_\alpha=6/(3+\alpha)$ and $s_\alpha=q_\alpha-1$.
For nonnegative $a,b$, with $b\in L^{q_\alpha}$ and
$a,b\in L^{q_\alpha}(U\dd x)$, the integral below is finite and
\begin{equation}\label{gauss:weighted-bilinear}
 \int Ua(I_\alpha*b)\le C_\alpha
 \|b\|_{q_\alpha}^{2-q_\alpha}
 \|a\|_{L^{q_\alpha}(U)}
 \|b\|_{L^{q_\alpha}(U)}^{s_\alpha}.
\end{equation}
The constant is independent of $\lambda,c$, and its mass. In particular
\[
 \int U|u|^p(I_\alpha*|u|^p)
 \le C_\alpha\|u\|_6^{2\alpha}\int U|u|^6.
\]
\end{lemma}
\begin{proof}
The Gaussian integral gives
\[
 \frac1{2(\lambda+r)}\le\xi_\lambda(r)\le\frac3{\lambda+r}:
 \qquad
 \phi_\lambda\ge(2\lambda)^{-1}\ (r\le\lambda),\quad
 \lambda\xi_\lambda\le1,\quad r\xi_\lambda\le2.
\]
Let $k_y(z)=(\lambda+|z-y|)^{-1}$ and $x\in B=B(a,r)$.
If $r\le(\lambda+|x-y|)/4$, then $k_y\le2k_y(x)$ on $B$.
Otherwise $B\subset B(y,6r)$ and
\begin{equation}\label{eq:gaussian-weight-ball-average}
 \frac1{|B|}\int_B k_y\le\frac{54}{r}
 \le216k_y(x).
\end{equation}
Integrating \eqref{eq:gaussian-weight-ball-average} against $|c(y)|^2$ proves the $A_1$ bound with
$A_1=3\cdot216\cdot2=1296$. Domination by
$\lambda^{-1}|c|^2$ gives continuity and finiteness of $U$;
positivity holds unless $c=0$ a.e.

For the centered maximal operator $M$, disjoint-ball selection and
fivefold covering give
\begin{equation}\label{eq:weighted-maximal-weak}
 U\{Mb>t\}\le\frac{125A_1}{t}\int|b|U.
\end{equation}
Indeed, for each selected ball,
$t|B|<\int_B|b|$ and
$U(y)\ge U(5B)/(125A_1|B|)$ on $B$.
Sum on bounded regions with bounded radii and then exhaust.
Since $U$ is bounded below on compact balls, weighted $L^q$ data
are locally integrable. Apply \eqref{eq:weighted-maximal-weak} to
$b\mathbf1_{\{|b|>t/2\}}$ and integrate in $t$:
\begin{equation}\label{eq:weighted-maximal-strong}
 \int(Mb)^qU\le D(A_1,q)\int|b|^qU,\qquad
 D(A_1,q)=\frac{250A_1q2^{q-1}}{q-1}\quad(q>1).
\end{equation}
Tonelli justifies this before finiteness; the right side implies
$Mb<\infty$ a.e.

Take $q=q_\alpha$ and $N=\|b\|_q$. The Hedberg estimate \eqref{eq:hedberg-pointwise}, with $\beta=\alpha$, applies since
$1<q<3/\alpha$. Its exponents are $\alpha q/3=2-q$ and
$1-\alpha q/3=q-1$, so
\begin{equation}\label{eq:weighted-hedberg}
 I_\alpha*b\le C_{\rm Hed}N^{2-q}(Mb)^{q-1},
\end{equation}
\[
 \begin{aligned}
 \int Ua(I_\alpha*b)
 &\le C_{\rm Hed}N^{2-q}\|a\|_{L^q(U)}
                    \|Mb\|_{L^q(U)}^{q-1}\\
 &\le C_{\rm Hed}D(A_1,q)^{(q-1)/q}N^{2-q}
           \|a\|_{L^q(U)}\|b\|_{L^q(U)}^{q-1}.
 \end{aligned}
\]
The estimates \eqref{eq:weighted-hedberg} and \eqref{eq:weighted-maximal-strong} justify the two steps and prove absolute integrability. Finally use $a=b=|u|^p$,
$pq=6$ and $p(2-q)=2\alpha$. Boundedness of $U$ supplies the
weighted domains without entering the constant.
\end{proof}

This weighted bound controls the pressure error produced by localization.

\begin{lemma}[Nonlinear Gaussian IMS with both fractional powers]
\label{gauss:IMS}
For every $M\ge0$ and $\epsilon>0$ there is
$K=K(\alpha,M,\epsilon)\ge0$ such that for every $\lambda>0$,
$c\in L^2$, and $u,v,g\in L^6$ satisfying
\[
 |u|^2=|v|^2+|g|^2,\qquad \|u\|_6\le M,\qquad
 \supp v\subset\overline B(0,\lambda),
\]
we have
\begin{equation}\label{gauss:IMS-bound}
 \mathcal P_\lambda(c;u)-\mathcal P_\lambda(c;v)
 \le\epsilon U_{\lambda,c}(0)
       +K\int U_{\lambda,c}|g|^6.
\end{equation}

\end{lemma}
\begin{proof}
Put $U=U_{\lambda,c}$, $d_u=|u|^p$ and $d=d_u-|v|^p\ge0$.
For $G(x)=\xi_\lambda(|x|)x$, radial monotonicity gives
\[
 \begin{aligned}
 &2(x-z)\cdot(G(x)-G(z))
 -(\xi_\lambda(|x|)+\xi_\lambda(|z|))|x-z|^2\\
 &\hspace{15mm}
 =(\xi_\lambda(|x|)-\xi_\lambda(|z|))(|x|^2-|z|^2)\le0.
 \end{aligned}
\]
Translation and integration against the full center yield
\begin{equation}\label{eq:gaussian-chord-pressure-bounds}
 C_q(x,z)\le\tfrac12(U(x)+U(z)),\qquad
 \operatorname{div}q\le3U.
\end{equation}
By \eqref{eq:gaussian-chord-pressure-bounds}, the symmetrized pressure coefficient obeys
\[
 \frac{2(p-2)}p\frac{\operatorname{div}q(x)+\operatorname{div}q(z)}2
 +\frac{2(3-\alpha)}pC_q(x,z)\le2(U(x)+U(z)).
\]
Both pressures are finite by HLS and the bounded chord. Expanding
\[
 0\le d_u(x)d_u(z)-|v(x)|^p|v(z)|^p
 \le d(x)d_u(z)+d_u(x)d(z)
\]
gives
\begin{equation}\label{eq:localized-pressure-two-terms}
 \mathcal P_\lambda(c;u)-\mathcal P_\lambda(c;v)
 \le4\left[\int Ud(I_\alpha*d_u)+\int Ud_u(I_\alpha*d)\right].
\end{equation}

Set $q=q_\alpha$, $s=q-1\in(0,1)$, $k=p/2$ and
$B=\int U|u|^6$, $T=\int U|g|^6$. The modulus split and H\"older give
\[
 0\le d\le k|u|^{p-2}|g|^2,\qquad d\le d_u,\qquad
 \|d\|_{L^q(U)}\le kB^{(1+\alpha)/6}T^{1/3},\qquad
 \|d\|_q\le M^p.
\]
Apply Lemma~\ref{gauss:weighted}, in the form \eqref{gauss:weighted-bilinear}, to the two ordered terms:
\begin{equation}\label{eq:pressure-fractional-estimates}
 \begin{aligned}
 \int Ud(I_\alpha*d_u)&\le C_\alpha kM^{2\alpha}B^{2/3}T^{1/3},\\
 \int Ud_u(I_\alpha*d)&\le C_\alpha k^sM^{2\alpha}B^{1-s/3}T^{s/3},
 \qquad \frac s3=\frac{3-\alpha}{3(3+\alpha)}>0.
 \end{aligned}
\end{equation}
Since $U(x)\le12U(0)$ on $B_\lambda$ and
$|u|^6\le4(|v|^6+|g|^6)$,
\[
 B\le48M^6U(0)+4T.
\]
Choose, before the center, radius and fields,
\[
 C=4C_\alpha M^{2\alpha}(k+k^s),\quad A_0=48M^6,\quad
 \delta=\frac{\epsilon}{2(C+1)(A_0+1)},
\]
\[
 L_t=t\left(\frac{1-t}{\delta}\right)^{(1-t)/t},\qquad
 K=C(8\delta+L_{1/3}+L_{s/3}).
\]
Insert \eqref{eq:pressure-fractional-estimates} into \eqref{eq:localized-pressure-two-terms} and use Young's inequality $B^{1-t}T^t\le\delta B+L_tT$ for $t=1/3,s/3$. This proves \eqref{gauss:IMS-bound}, including all zero cases.
\end{proof}

The two pressure differences have different fractional powers of
the tail norm, both strictly positive for $0<\alpha<3$. To absorb
them, we next control the weighted sixth-power tail by the angular and
isotropic parts of the Gaussian kinetic form.

\subsection{Kelvin extension and the angular tail estimate}
The angular Hessian measures derivatives perpendicular to $x-y$.
An exterior Sobolev inequality, followed by a change of longitudinal
coordinate and an average over directions, inserts the second density
factor without introducing its total mass.
\begin{lemma}[The Gaussian angular inequality]\label{gauss:angular}
For real $g\in H^1$, put $A_g=\int|\nabla g|^2$ and
\[
 \begin{aligned}
 N_w(g)&=\iint w_\lambda(|x-y|)g(y)^2g(x)^6\dd x\dd y,\\
 X_w(g)&=\iint w_\lambda(|x-y|)g(y)^2|P_{x-y}\nabla g(x)|^2\dd x\dd y.
 \end{aligned}
\]
Both pairings are absolutely integrable and
\begin{equation}\label{gauss:angular-bound}
 N_w(g)\le\frac{243}{4}C_6 A_g^2X_w(g).
\end{equation}
\end{lemma}
\begin{proof}
Write $x=(s,y)\in\R\times\R^2$,
$A_\perp=\int|\nabla_yv|^2$, $A_\parallel=\int|v_s|^2$.
Sobolev applied to $v(as,y)$, with $q=a^2$, gives
\begin{equation}\label{eq:anisotropic-sobolev-optimization}
 q\int v^6\le C_6(A_\perp+qA_\parallel)^3,
 \qquad
 \int v^6\le\frac{27}{4}C_6A_\perp^2A_\parallel
 \quad\left(q=\frac{A_\perp}{2A_\parallel}\right).
\end{equation}
Adding $\epsilon$ to both energies and then letting
$\epsilon\downarrow0$ covers zero energies.

For $v\in C_c^\infty$ and $R>0$, define the transverse Kelvin extension
\[
 v_R(s,y)=\begin{cases}
 v(s,R^2y/|y|^2),&0<|y|<R,\\
 0,&y=0,\\
 v(s,y),&|y|\ge R.
 \end{cases}
\]
The inversion fixes the interface $|y|=R$, so the two traces agree
and Lemma~\ref{std:gluing} applies. In two transverse dimensions it
also preserves the transverse Dirichlet energy. For
$\iota(y)=R^2y/|y|^2$,
\begin{equation}\label{eq:kelvin-extension-energy}
 \begin{gathered}
 D\iota=\frac{R^2}{|y|^2}(I-2\widehat y\otimes\widehat y),
 \qquad |\det D\iota|=\frac{R^4}{|y|^4},\\
 \int|\nabla_yv_R|^2=2\int_{|y|>R}|\nabla_yv|^2,
 \qquad
 \int|\partial_sv_R|^2
 =\int_{|y|>R}(1+R^4/|y|^4)|v_s|^2
 \le2\int_{|y|>R}|v_s|^2.
 \end{gathered}
\end{equation}
Applying \eqref{eq:anisotropic-sobolev-optimization} with \eqref{eq:kelvin-extension-energy} gives
\begin{equation}\label{eq:exterior-anisotropic-sobolev}
 \int_{|y|>R}|v|^6
 \le54C_6\left(\int|\nabla_yv|^2\right)^2
                    \int_{|y|>R}|v_s|^2.
\end{equation}
For bounded nonnegative nondecreasing $b$, its Stieltjes
representation by exterior indicators, with null jump cylinders
discarded, and Tonelli applied to \eqref{eq:exterior-anisotropic-sobolev} give
\begin{equation}\label{gauss:weighted-anisotropic}
 \int b(|y-y'|)|v|^6
 \le54C_6\left(\int|\nabla_yv|^2\right)^2
                   \int b(|y-y'|)|v_s|^2.
\end{equation}
Boundedness of $b$ extends this to $H^1$ by smooth approximation.

Fix orthonormal coordinates $(s,y)$ and smooth compact $g$.
Set $d=\partial_sg$, $\beta(y,y')=b(|y-y'|)$ and
\[
 N_b=\int\beta(y,y')g(s,y')^2g(s,y)^6\dd s\dd y\dd y',\qquad
 I_b=\int\beta(y,y')g(s,y')^2d(s,y)^2\dd s\dd y\dd y'.
\]
Choose compact $Y$ containing the transverse supports of $g,d$.
For fixed $y'\in Y$, set $h(s)=g(s,y')$ and
\[
 H=h^2+\epsilon,\quad a=H^{1/4},\quad
 \tau(s)=\int_0^sH(t)^{-1/2}\dd t,\qquad
 v(\tau(s),y)=a(s)g(s,y).
\]
This increasing change of the longitudinal spatial coordinate inserts
$h(s)^2+\epsilon$ into the sixth-power integral while preserving the
transverse Dirichlet energy. It preserves compact support and gives
\[
 \int|\nabla_yv|^2\dd\tau\dd y=A_\perp,
 \qquad \int\beta v^6\dd\tau\dd y
 =\int\beta Hg^6\dd s\dd y.
\]
The longitudinal density satisfies
\begin{equation}\label{eq:longitudinal-density-bound}
 R_\epsilon=(h^2+\epsilon)d^2+hh'gd
             +\frac{h^2(h')^2g^2}{4(h^2+\epsilon)}
 \le\frac32h^2d^2+\frac34g^2(h')^2+\epsilon d^2.
\end{equation}
Apply \eqref{gauss:weighted-anisotropic} with \eqref{eq:longitudinal-density-bound}, integrate over $Y$, and
symmetrize in $y,y'$:
\begin{equation}\label{eq:weighted-directional-pair-bound}
 N_b\le54C_6A_\perp^2
       \left(\frac94I_b+\epsilon C_Y\right),\qquad
 C_Y=\int_{\R\times Y\times Y}\beta(y,y')d(s,y)^2<\infty.
\end{equation}
Letting $\epsilon\downarrow0$ in \eqref{eq:weighted-directional-pair-bound} gives
$N_b\le(243/2)C_6A_\perp^2I_b$.

It remains to convert the directional estimate into the angular
Hessian. The incidence measure supplies the inverse-distance factor,
and its second moment supplies the tangential projection. For $z\ne0$,
let $\nu_z$ be $|z|^{-1}$ times arclength on
$\{\omega\in S^2:\omega\cdot z=0\}$, and put $\nu_0=0$.
A Borel orthonormal frame gives measurability, while circle integration gives
\begin{equation}\label{eq:great-circle-moments}
 \int\dd\nu_z=\frac{2\pi}{|z|},\qquad
 \int|\omega\cdot a|^2\dd\nu_z(\omega)
 =\frac\pi{|z|}|P_za|^2.
\end{equation}
For nonnegative measurable $F$,
\begin{equation}\label{eq:orthogonal-incidence-formula}
 \begin{aligned}
 &\int_{S^2}\int_\R\int_{\omega^\perp}\int_{\omega^\perp}
 F(\omega,s\omega+y,s\omega+y')\dd y'\dd y\dd s\dd\omega\\
 &\hspace{8mm}
 =\iint\int F(\omega,x,x')\dd\nu_{x-x'}(\omega)\dd x\dd x'.
 \end{aligned}
\end{equation}
Indeed, put $x=s\omega+y$, $z=y-y'=rn$.
Rotation invariance on orthogonal unit pairs permits exchanging
$\omega,n$ (both total masses are $8\pi^2$); planar polar measure
$r\dd r\dd\ell(n)$ then becomes $|z|^{-1}\dd z$.
Tonelli completes the identity off the null diagonal.

Now
\[
 b(r)=rw_\lambda(r)=\int_0^r-t\phi_\lambda'(t)\dd t
\]
is nonnegative, increasing and bounded. Average \eqref{eq:weighted-directional-pair-bound} after $\epsilon\downarrow0$, using \eqref{eq:great-circle-moments}--\eqref{eq:orthogonal-incidence-formula} and $A_\perp\le A_g$:
\[
 2\pi N_w\le\frac{243}{2}C_6A_g^2\pi X_w.
\]
Finally take real $g_n\to g$ in $H^1\cap L^6$.
Their squared, sixth and gradient-product moments converge in $L^1$;
$0\le w_\lambda\le\lambda^{-1}$ therefore passes both pairings and
\eqref{gauss:angular-bound} to $g$.
\end{proof}

For the isotropic part, applying Sobolev to $\sqrt{U}\,g$ produces
a density Laplacian term. This term is kept with its sign rather than
estimated by the mass of $g$.

\begin{lemma}[Isotropic convolution Sobolev]\label{gauss:isotropic}
Let $v\ge0$ be even, $C^2$, and bounded together with its first two
derivatives. For real $g\in H^1$, put
$U=v*g^2$, $X_v=\int U|\nabla g|^2$,
$N_v=\int Ug^6$, and $Q_v=-\int g^2\Delta U$.
These are finite, $\Delta U=(\Delta v)*g^2$, and
\begin{equation}\label{gauss:isotropic-bound}
 N_v\le C_6 A_g^2(2X_v+Q_v/2).
\end{equation}
\end{lemma}
\begin{proof}
To prove \eqref{gauss:isotropic-bound}, take real $g\in C_c^\infty$ and put $U_\epsilon=U+\epsilon$.
Integration by parts gives $\nabla U=2v*(g\nabla g)$;
weighted Cauchy--Schwarz and evenness give
\begin{equation}\label{eq:weighted-gradient-cauchy}
 |\nabla U(x)|^2\le4U(x)\int v(x-y)|\nabla g(y)|^2\dd y,
 \qquad \int g^2\frac{|\nabla U|^2}{U_\epsilon}\le4X_v.
\end{equation}
H\"older and sharp Sobolev yield
\begin{equation}\label{eq:weighted-sobolev-product}
 \int U_\epsilon g^6
 \le\left(\int g^6\right)^{2/3}
       \left(\int U_\epsilon^3g^6\right)^{1/3}
 \le C_6A_g^2\int|\nabla(\sqrt{U_\epsilon}g)|^2.
\end{equation}
Expanding the square and integrating the cross term,
\begin{equation}\label{eq:weighted-gradient-expansion}
 \begin{aligned}
 \int|\nabla(\sqrt{U_\epsilon}g)|^2
 &=X_v+\epsilon A_g+\frac12Q_v
       +\frac14\int g^2\frac{|\nabla U|^2}{U_\epsilon}\\
 &\le2X_v+\epsilon A_g+Q_v/2.
 \end{aligned}
\end{equation}
Insert \eqref{eq:weighted-gradient-expansion}, using \eqref{eq:weighted-gradient-cauchy}, into \eqref{eq:weighted-sobolev-product}, and let $\epsilon\downarrow0$. For $g\in H^1$, smooth approximation
in $H^1\cap L^6$ gives $L^1$ convergence of the squared, sixth and
gradient-square moments. Boundedness of $v,\Delta v$ passes
$N_v,X_v,Q_v$ to the limit, with no sign condition on $Q_v$.
\end{proof}

Combining the two estimates bounds the full-density tail.

\begin{lemma}[Self-tail and full-density tail bounds]\label{gauss:self-tail}
For real $g\in H^1$,
\begin{equation}\label{gauss:self-sixth}
 \int U_{\lambda,g}g^6
 \le C_6 A_g^2\left[\frac{243}{4}\mathcal H_\lambda(g;g)
                         +\frac12Q_{\phi,\lambda}(g^2)\right].
\end{equation}
If $f,v,g\in H^1$ are real and $f^2=v^2+g^2$, and $M_c=\int v^2$,
then
\begin{equation}\label{gauss:whole-tail}
 \int U_{\lambda,f}g^6
 \le C_6 A_g^2\left[\frac{243}{4}\mathcal H_\lambda(f;g)
                  +Q_\lambda(f^2)+\frac{5M_c^2}{\lambda^3}\right]
       +\frac{M_c}{\lambda}\int g^6.
\end{equation}
\end{lemma}
\begin{proof}
Use $\xi_\lambda=\phi_\lambda+w_\lambda$ and the Hessian decomposition \eqref{eq:gaussian-hessian-laplacian}.
Lemmas~\ref{gauss:angular} and \ref{gauss:isotropic}, with
$v(x)=\phi_\lambda(|x|)$, give \eqref{gauss:self-sixth}, since
$2\le243/4$ and both energy summands are nonnegative.

For the full center, $f^2=v^2+g^2$ gives
\[
 U_{\lambda,v}\le M_c/\lambda,\qquad
 \mathcal H_\lambda(g;g)\le\mathcal H_\lambda(f;g),\qquad
 \tfrac12Q_{\phi,\lambda}(g^2)
 \le Q_\lambda(f^2)+5M_c^2/\lambda^3.
\]
The second inequality uses $H_\lambda\ge0$, and the third is \eqref{gauss:Q-tail} from Lemma~\ref{gauss:Q}. Substitution proves \eqref{gauss:whole-tail}.
All pairings have the bounded-kernel domains already established;
only the core mass $M_c$ enters their bounds.
\end{proof}

The tail bounds leave the full density form available for absorption.
For the core, the directional gap must be made uniform for each location
of the other density before integrating over that location.

\subsection{Finite signed core stress and scale separation}
An angle cutoff of radius $L>0$ is a smooth real function
$0\le\theta\le\pi/2$ which is zero on $B(0,L/2)$ and equals
$\pi/2$ outside $B(0,L)$. For real $f\in H^1$ put
\[
 v=(\cos\theta)f,\qquad g=(\sin\theta)f,\qquad e_\theta=f\nabla\theta.
\]
Then $f^2=v^2+g^2$, and the exact matrix IMS identity is
\begin{equation}\label{gauss:kinetic-IMS}
 \nabla v\otimes\nabla v+\nabla g\otimes\nabla g
 =\nabla f\otimes\nabla f+e_\theta\otimes e_\theta.
\end{equation}
This identity is first algebraic for smooth functions and extends to
weak derivatives by the product rule.

Representing both pressures by a finite signed matrix measure
isolates the core stress. A slowly varying Gaussian Hessian then tests
its positive total matrix, with an error proportional to $L/\lambda$.

\begin{lemma}[A uniform core lower bound]\label{gauss:core}
Fix $m,e>0$. Put
\[
 \kappa=\frac{m}{8C_6G^2},\qquad \sigma=\frac{4e\kappa}{G}>0.
\]
For each $L>0$ there is $R_1>0$, depending only on
$\alpha,m,e,L$, such that for all $\lambda\ge R_1$, every angle cutoff
of radius $L$ and every real $f\in H^1$ satisfying
\[
 \int|v|^6\ge m,\quad A(v)<G,\quad E(v)\le E(W)-e,
 \quad\int|e_\theta|^2\le e\kappa/G,
\]
one has, with the full center $f$,
\begin{equation}\label{gauss:core-bound}
 \sigma U_{\lambda,f}(0)
 \le4\mathcal H_\lambda(f;v)-\mathcal P_\lambda(f;v)
       -4\iint f(y)^2H_\lambda(x-y)e_\theta(x)\cdot e_\theta(x)\dd x\dd y.
\end{equation}

\end{lemma}
\begin{proof}
Define the signed matrix measure $\mathfrak S$ by
\begin{equation}\label{eq:signed-core-stress}\begin{aligned}
 \langle\mathfrak S,H\rangle={}&
 4\int H(x):[\nabla v\otimes\nabla v-e_\theta\otimes e_\theta](x)\dd x\\
 &-\frac{2(p-2)}p\int\operatorname{tr}H(x)
                  (I_\alpha*|v|^p)(x)|v(x)|^p\dd x\\
 &-\frac{2(3-\alpha)}p\iint I_\alpha(x-z)|v(x)|^p|v(z)|^p\\
 &\qquad\qquad\times\int_0^1H(z+t(x-z)):
          \frac{(x-z)\otimes(x-z)}{|x-z|^2}\dd t\dd x\dd z.
\end{aligned}\end{equation}
The last term in \eqref{eq:signed-core-stress} pushes forward the finite Riesz pair measure times
$\dd t$ on $[0,1]$. Convexity of $\overline B_L$ gives
\begin{equation}\label{eq:core-stress-variation}
 \supp\mathfrak S_{ij}\subset\overline B_L,\qquad
 |\mathfrak S_{ij}|(\R^3)\le
 V_*:=4A(v)+4D(v)+4\int|e_\theta|^2.
\end{equation}
Here each unit-dyad entry is bounded by one and the pressure
coefficients sum to $8/p\le4$.
For unit $\omega$, Lemma~\ref{gauss:directional} gives
\begin{equation}\label{eq:core-stress-positive-matrix}
 A_\omega(v)\ge\kappa,\qquad
 \begin{aligned}
 \omega\cdot\mathfrak S(\R^3)\omega
 &=4(A_\omega(v)-P_\omega(v))-4\int|\omega\cdot e_\theta|^2\\
 &\ge\frac{8e\kappa}{G}-\frac{4e\kappa}{G}=\sigma.
 \end{aligned}
\end{equation}

Differentiating $H=\xi I+(\phi-\xi)n\otimes n$, using
$\xi'=(\phi-\xi)/r$ and $\phi'=-r\phi/\lambda^2$, yields
\[
 |\partial_a(H_\lambda)_{ij}(z)|
 \le32\lambda^{-1}\xi_\lambda(|z|).
\]
For $r\le\lambda$, use
$|\phi-\xi|\le r^2/(2\lambda^3)$ and
$\xi\ge(4\lambda)^{-1}$; for $r\ge\lambda$, use
$|\phi-\xi|/r\le\xi/\lambda$ and Gaussian decay for $\phi'$.
Smoothness covers zero. On the segment from $-y$ to $x-y$, with
$|x|\le L\le\lambda$, one has
$\xi_\lambda(|z|)\le12\xi_\lambda(|y|)$; hence
\begin{equation}\label{eq:gaussian-hessian-variation}
 |(H_\lambda)_{ij}(x-y)-(H_\lambda)_{ij}(-y)|
 \le384\frac L\lambda\xi_\lambda(|y|).
\end{equation}
Integrate \eqref{eq:gaussian-hessian-variation} against the nine signed variations in \eqref{eq:core-stress-variation}, and use \eqref{eq:core-stress-positive-matrix}:
\begin{equation}\label{eq:translated-core-stress-lower}
 \begin{aligned}
 \langle\mathfrak S,H_\lambda(\cdot-y)\rangle
 &\ge H_\lambda(-y):\mathfrak S(\R^3)
       -3456\frac L\lambda V_*\xi_\lambda(|y|)\\
 &\ge\left(2\sigma-3456\frac L\lambda V_*\right)
                   \xi_\lambda(|y|),
 \end{aligned}
\end{equation}
where $H_\lambda\ge0$ and $\operatorname{tr}H_\lambda\ge2\xi_\lambda$.
Sharp HLS--Sobolev gives $D(v)\le G$, so choose
\begin{equation}\label{eq:core-radius-choice}
 V_*\le\overline V:=8G+\sigma,\qquad
 R_1=\max\{L,3456L\overline V/\sigma\}+1.
\end{equation}
By \eqref{eq:translated-core-stress-lower}--\eqref{eq:core-radius-choice}, uniformly in $y$ and $\lambda\ge R_1$,
\begin{equation}\label{eq:uniform-core-stress-bound}
 \langle\mathfrak S,H_\lambda(\cdot-y)\rangle
 \ge\sigma\xi_\lambda(|y|).
\end{equation}
Multiply \eqref{eq:uniform-core-stress-bound} by $f(y)^2$ and integrate. The finite stress variations,
bounded Hessian and $f\in L^2$ justify Fubini; the chord identity
\[
 (\nabla k_\lambda(x-y)-\nabla k_\lambda(z-y))\cdot(x-z)
 =\int_0^1H_\lambda(z+t(x-z)-y)[x-z,x-z]\dd t
\]
identifies the result with \eqref{gauss:core-bound}, including both
pressures. The choice of $R_1$ is independent of the center mass.
\end{proof}

Choose the tail smallness before the core radius and its mass bound.
This order allows the angular error to be absorbed in the kinetic form
and the isotropic error in the full density form, retaining $99/100$ of the density form and a positive core term.

\begin{proposition}[Whole-form tail absorption]\label{gauss:absorption}
Fix $m,e>0$ and $M\ge0$, and define $\sigma,\kappa$ as in
Lemma~\ref{gauss:core}. There exists $\eta>0$ with the following order
of quantifiers. For every $L>0$ and $B_c\ge0$ there is $R_0>0$ such
that, for every $\lambda\ge R_0$, every angle cutoff of radius $L$,
and every real $f\in H^1$ satisfying the hypotheses of
Lemma~\ref{gauss:core} and
\[
 \|f\|_6\le M,\quad M_c=\int v^2\le B_c,\quad
 A(g)\le\eta,\quad \int g^6\le\eta,
\]
one has
\begin{equation}\label{gauss:absorption-bound}
 \frac{99}{100}Q_\lambda(f^2)+\frac{\sigma M_c}{16\lambda}
                         \le B_\lambda(f).
\end{equation}

\end{proposition}
\begin{proof}
Apply Lemma~\ref{gauss:IMS} with $\epsilon=\sigma/4$, choosing $K$
before $L$ or $f$. The matrix identity \eqref{gauss:kinetic-IMS} and
Lemma~\ref{gauss:core} give
\[
 4\mathcal H_\lambda(f;f)-\mathcal P_\lambda(f)
 \ge\frac{3\sigma}{4}U_{\lambda,f}(0)
       +4\mathcal H_\lambda(f;g)-K\int U_{\lambda,f}g^6.
\]
For $\lambda\ge L$, $U_{\lambda,f}(0)\ge M_c/(4\lambda)$.
Discarding $\sigma U(0)/4$ leaves
\begin{equation}\label{eq:core-tail-before-absorption}
 4\mathcal H_\lambda(f;f)-\mathcal P_\lambda(f)
 \ge\frac{\sigma M_c}{8\lambda}+4H-KT,
 \quad H=\mathcal H_\lambda(f;g),\quad T=\int U_{\lambda,f}g^6.
\end{equation}
Set $a=C_6A(g)^2$, $b=\int g^6$ and $Q=Q_\lambda(f^2)$.
Lemma~\ref{gauss:self-tail} gives
\begin{equation}\label{eq:whole-tail-three-terms}
 T\le a\left(\frac{243}{4}H+Q+\frac{5M_c^2}{\lambda^3}\right)
       +\frac{M_c}{\lambda}b.
\end{equation}
Choose, before $L,B_c$,
\[
 \eta=\min\left\{1,\frac1{100(KC_6+1)},\frac{\sigma}{32(K+1)}\right\}>0,
\]
and subsequently choose
\[
 R_0=\max\left\{R_1,L,1,\frac{160KC_6\eta^2B_c}{\sigma}\right\}+1.
\]
For $\lambda\ge R_0$,
$Ka\le1/100$, $Kb\le\sigma/32$, and
$5KaM_c/\lambda^2\le\sigma/32$; the last uses
$\lambda^2\ge\lambda$ and the definition of $R_0$.
Substitute these bounds in \eqref{eq:whole-tail-three-terms}; since $H,Q\ge0$ and $(243/4)/100<4$,
\begin{equation}\label{eq:tail-error-absorption}
 KT\le4H+Q/100+\frac{\sigma M_c}{16\lambda}.
\end{equation}
Use \eqref{eq:tail-error-absorption} in \eqref{eq:core-tail-before-absorption} and add $Q$ to obtain \eqref{gauss:absorption-bound}.
\end{proof}

Compactness supplies a common core with positive local mass and
uniformly small gradient tails. The resulting constants depend on this
core and the threshold gaps, while the total $L^2$ norm of nearby
highpass fields remains unrestricted.

\subsection{Compact neighborhoods and affine transfer}
\begin{theorem}[Uniform Gaussian coercivity near a compact carrier]
\label{gauss:compact}
Let $\mathcal K\subset\X$ be nonempty and compact in the gradient norm, with
$0\notin\mathcal K$, and suppose every $f\in\mathcal K$ satisfies
$A(f)<G$ and $E(f)<E(W)$. There are $c>0$, $\delta>0$, and $R_*>0$
such that for every $\lambda\ge R_*$ and every $h\in H^1$ with
$\operatorname{dist}_{\X}(h,\mathcal K)<\delta$,
\begin{equation}\label{gauss:compact-bound}
 B_\lambda(h)\ge\frac{99}{100}Q_\lambda(|h|^2)+\frac c\lambda.
\end{equation}

\end{theorem}
\begin{proof}
Continuity of $A,E,\|\cdot\|_6$ in $\X$ and compactness allow
$m,e>0$ and $M<\infty$ such that
\begin{equation}\label{eq:compact-carrier-margins}
 \inf_{\mathcal K}\int|f|^6>2m,\qquad
 \sup_{\mathcal K}E(f)<E(W)-2e,\qquad
 \sup_{\mathcal K}A(f)<G,\qquad
 \sup_{\mathcal K}\|f\|_6<M.
\end{equation}
Since $0\notin\mathcal K$, continuity of local $L^2$ norms and a
finite cover also give $A,m_2>0$ with
$\inf_{f\in\mathcal K}\int_{B_A}|f|^2>2m_2$.

Take a fixed angle profile dilated to radius $L$. Then
\[
 |\nabla\theta_L|\le C/L,\qquad
 \|f\nabla\theta_L\|_2
 \le C\|f\|_{L^6(\{L/2<|x|<L\})}.
\]
The core tends to $f$, the tail to zero in $\X\cap L^6$, and the
IMS error to zero in $L^2$. Uniform operator bounds and a finite
net make these limits uniform on $\mathcal K$.
Obtain $\eta$ from Proposition~\ref{gauss:absorption} for $m,e,M$;
then choose $L>2A$ and $\delta>0$ so that, whenever
$\operatorname{dist}_{\X}(h,\mathcal K)<\delta$, the fields
$v_h=(\cos\theta_L)h$, $g_h=(\sin\theta_L)h$ satisfy
\begin{equation}\label{eq:uniform-neighborhood-cutoffs}
 \begin{gathered}
 \int|v_h|^6\ge m,\qquad A(v_h)<G,\qquad E(v_h)\le E(W)-e,
 \qquad \|h\|_6\le M,\\
 A(g_h)\le\eta,\qquad \int|g_h|^6\le\eta,\qquad
 \int|h\nabla\theta_L|^2\le e\kappa/G,\\
 m_2\le\int|v_h|^2\le |B_L|^{2/3}M^2=:B_c^{\rm bound}.
 \end{gathered}
\end{equation}
The lower mass bound uses $\theta_L=0$ on $B_A$; all neighborhood
bounds follow from \eqref{eq:compact-carrier-margins} and finite-volume H\"older.

For $h\in H^1$, put $f=|h|$, $v=|v_h|$, $g=|g_h|$.
Lemma~\ref{gauss:phase} gives
\[
 A(v)\le A(v_h),\qquad E(v)\le E(v_h),\qquad A(g)\le A(g_h),
\]
while all modulus moments and the IMS error are unchanged.
Thus no continuity of the modulus-gradient map is needed.
The bounds \eqref{eq:uniform-neighborhood-cutoffs} permit Proposition~\ref{gauss:absorption} with fixed $L,B_c^{\rm bound}$; it gives $R_*$ and
\[
 B_\lambda(h)\ge B_\lambda(|h|)
 \ge\frac{99}{100}Q_\lambda(|h|^2)
          +\frac{\sigma m_2}{16\lambda}\qquad(\lambda\ge R_*).
\]
Take $c=\sigma m_2/16$. The mass of $h$ is used only to define the
Gaussian integrals, after these constants are fixed.
\end{proof}

Scaling returns the estimate to physical coordinates.

\begin{corollary}[Physical affine coercivity]\label{gauss:affine}
Use $c,\delta,R_*$ from Theorem~\ref{gauss:compact}. Let $r>0$,
$x_0\in\R^3$, and $h\in H^1$ satisfy
$\operatorname{dist}_{\X}(T_{r,x_0}^{-1}h,\mathcal K)<\delta$.
For every $\lambda\ge R_*r$,
\begin{equation}\label{gauss:affine-bound}
 B_\lambda(h)\ge\frac{99}{100}Q_\lambda(|h|^2)
                              +c\frac{r^2}{\lambda}.
\end{equation}

\end{corollary}
\begin{proof}
Put $g=T_{r,x_0}^{-1}h$ and $R=\lambda/r$. Then
\[
 k_{rR}(rz)=rk_R(z),\qquad H_{rR}(rz)=r^{-1}H_R(z),\qquad
 \Delta^2k_{rR}(rz)=r^{-3}\Delta^2k_R(z).
\]
Changes of variables and weak-gradient covariance give
\[
 Q_\lambda(|h|^2)=rQ_R(|g|^2),\qquad
 B_\lambda(h)=rB_R(g),\qquad M_\lambda(h)=r^3M_R(g).
\]
For the trace pressure the density, potential, convolved Hessian
and volume factors multiply to
\[
 r^{-p/2}r^{\alpha-p/2}r\,r^3=r^{\alpha-p+4}=r;
\]
the two-variable change in the chord pressure gives the same factor.
Absolute convergence justifies these identities for a.e.
representatives. Since $R\ge R_*$, the preceding bound \eqref{gauss:compact-bound} gives
\[
 B_\lambda(h)=rB_R(g)
 \ge\frac{99}{100}Q_\lambda(|h|^2)+\frac{rc}{R}
 =\frac{99}{100}Q_\lambda(|h|^2)+\frac{cr^2}{\lambda}.
\]
\end{proof}

The normalized highpass lies in this coercive neighborhood.

\begin{corollary}[The highpass in the coercive neighborhood]
\label{gauss:highpass-neighborhood}
Suppose $T_{r(s),x(s)}^{-1}U(s)$ belongs to a compact carrier
$\mathcal K$ as above and $0<r(s)\le1$. For all sufficiently small
fixed $\nu>0$, the actual $h_\nu=P_{>\nu}U$ satisfies
\eqref{gauss:affine-bound} for every $s$ and every
$\lambda\ge R_*r(s)$. 
\end{corollary}
\begin{proof}
By Proposition~\ref{ann:uniform-highpass},
\[
 T_{r,x}^{-1}P_{>\nu}T_{r,x}=P_{>\nu r},\qquad
 \omega_1(\nu):=
 \sup_{f\in\mathcal K,\,0<L\le\nu}\|\nabla P_{\le L}f\|_2\to0.
\]
Writing $f_s=T_{r(s),x(s)}^{-1}U(s)\in\mathcal K$ and using
$0<r(s)\le1$, we obtain
\[
 \operatorname{dist}_{\X}(T_{r(s),x(s)}^{-1}h_\nu(s),\mathcal K)
 \le\|\nabla P_{\le\nu r(s)}f_s\|_2\le\omega_1(\nu)<\delta
\]
for all sufficiently small $\nu$. The same proposition gives the
actual $H^1$ highpass. Corollary~\ref{gauss:affine} applies with its
fixed $c,\delta,R_*$, pointwise in $s$.
\end{proof}

The coercivity constants are thus independent of the highpass mass.

\section{Gaussian high-frequency analysis and rigidity}
\label{sec:highpass-source}
The positive Gaussian form will control one quartic time integral.
To close that bound, the projected source and the motion of the selected
radius must be compared with the same integral on common finite time
intervals. Only after these errors are absorbed can the time endpoint be
sent to infinity and the original mass be recovered.

\subsection{The complete highpass source estimate}
Fix the trajectory, orientation and normalization of
Theorem~\ref{crit:bounded-halfline}. Henceforth
$U(s)=u(\sigma s)$, $s\geq0$, where $\sigma\in\{-1,1\}$ is fixed.
The symbols $r(s),x(s)$ denote the normalization after composition with
this parameterization; no measurability of these two choices is needed.
Use the one cutoff of Section~\ref{sec:annular}, and set
\[
 B=\sup_{s\geq0}\|\nabla U(s)\|_2,\quad
 l_\nu=P_{\leq\nu}U,\quad h_\nu=P_{>\nu}U,\quad
 m_\nu=\sup_s\|h_\nu(s)\|_2,\quad \delta_\nu=\nu m_\nu.
\]
By Proposition~\ref{ann:uniform-highpass}, $m_\nu<\infty$,
$\delta_\nu\to0$, and
$\eta_\nu=\sup_s\|\nabla l_\nu(s)\|_2\to0$ as $\nu\downarrow0$.
Theorem~\ref{ann:orbit-l4} gives
$U\in C\cap L^\infty([0,\infty);L^4)$ and
$F(U)\in C\cap L^\infty([0,\infty);L^1)$, before any time-integrability
claim. Write
\[
 f(s)=\|F(U(s))\|_1,\qquad z(s)=\|U(s)\|_4^2.
\]
Estimate \eqref{eq:hartree-l1-interpolation} gives $f\leq C_{\alpha,B}z$.

\begin{lemma}[The forced highpass equation]\label{dyn:forced}
For fixed $\nu>0$ put
\[
 e_\nu=P_{>\nu}F(U)-F(h_\nu)
       =c_\nu+o_\nu,\quad
 c_\nu=F(U)-F(h_\nu),\quad o_\nu=-P_{\leq\nu}F(U).
\]
Then $h_\nu\in C([0,\infty);H^1)$,
$\nabla h_\nu\in L^2_{\mathrm{loc}}([0,\infty);L^6)$, and
\begin{equation}\label{eq:forced-highpass}
 i\partial_s h_\nu+\sigma\Delta h_\nu+\sigma F(h_\nu)
       =-\sigma e_\nu
\end{equation}
holds in spacetime distributions. Moreover $e_\nu$ is continuous and
bounded into $L^1\cap L^{6/5}$, and
\[
 \epsilon_\nu:=\sup_s\|e_\nu(s)\|_{6/5}
 \leq C_\alpha B^{2p-2}\eta_\nu
       +C_\varphi\nu^{1/2}\sup_s f(s)\longrightarrow0.
\]
\end{lemma}
\begin{proof}
Proposition~\ref{ann:uniform-highpass} gives $h_\nu\in CH^1$.
Young's inequality gives
\begin{equation}\label{eq:lowpass-endpoint-gradient}
 \|\nabla P_{\le\nu}U\|_6
 \le\|K_\nu\|_{3/2}\|\nabla U\|_2
 =\nu\|K\|_{3/2}\|\nabla U\|_2.
\end{equation}
Subtract the continuous gradient controlled by \eqref{eq:lowpass-endpoint-gradient} from \eqref{local:fullz}, reflecting
time when $\sigma=-1$, to obtain $\nabla h_\nu\in L^2_{\rm loc}L^6$.

Test the finite mild identity against Schwartz spatial and compact
smooth temporal tests. On each compact time interval,
$F(U)\in CL^{6/5}$ and the propagated tests belong to $C^1L^6$;
Fubini and integration by parts give
\[
 iU_s+\sigma\Delta U+\sigma F(U)=0.
\]
Commutation with the fixed multiplier gives \eqref{eq:forced-highpass}.
Power differences, strict HLS and Young yield
\[
 \begin{aligned}
 \|F(v)-F(w)\|_{6/5}
 &\le C_\alpha\|v-w\|_6(\|v\|_6+\|w\|_6)^{2p-2},\\
 \|c_\nu(s)\|_{6/5}&\le C_\alpha B^{2p-2}\eta_\nu,\\
 \|o_\nu(s)\|_{6/5}&\le\nu^{1/2}\|K\|_{6/5}f(s).
 \end{aligned}
\]
Thus $\epsilon_\nu\to0$. Corollary~\ref{ann:source-interface} and
the multiplier bounds give bounded $L^1$ continuity; no $L^1$
smallness is used.
\end{proof}

Use the complete density and current of $h=h_\nu$ in every weight:
\[
 \rho=|h|^2,\quad j=\im(\overline h\nabla h),\quad
 q=\nabla k_\lambda*\rho,\quad
 q_m(y)=\int\nabla k_\lambda(x-y)\cdot j(x)\dd x.
\]
Here $k_\lambda$ is the kernel defined in Lemma~\ref{gauss:kernel}.
At times with $h,\nabla h\in L^6$, define the derivative-free weak source
\begin{equation}\label{eq:source-flat}
 \mathcal S_\lambda^{\mathrm{flat}}(Z,h)
 =-4\re\int q\cdot Z\nabla\overline h
  -2\int\operatorname{div}q\,\re(Z\overline h)
  -4\int q_m\im(Z\overline h).
\end{equation}
It is a genuine integral for $Z\in L^{6/5}$. The last term comes from
variation of the density in $q$ and is essential.
For smooth fields it is obtained by integrating by parts in
\[
 -2\int q\cdot\re(Z\nabla\overline h-h\nabla\overline Z)
       -4\int q_m\im(Z\overline h).
\]
We next construct a continuous extension of this source for the actual
$e_\nu$, without requiring $\nabla h\in L^6$ at every time.

\begin{lemma}[Radius-independent weight bounds]\label{dyn:weights}
For $h\in H^1$, $m=\|h\|_2$, $b=\|\nabla h\|_2$, uniformly in
$\lambda>0$,
\begin{equation}\label{eq:full-density-weight-bounds}
 \|q\|_\infty\leq C m^2,\quad \|q_m\|_\infty\leq Cmb,
 \quad \|Dq\|_\infty+\|\operatorname{div}q\|_\infty\leq Cmb.
\end{equation}
If $a=\|h\|_4^4$, their genuine derivatives additionally satisfy
\begin{equation}\label{eq:full-density-weight-derivatives}
 \|D^2q\|_6\leq Ca^{1/2},\qquad
 \|\nabla q_m\|_{12}\leq Cb a^{1/4}.
\end{equation}

\end{lemma}
\begin{proof}
For $R=|x|$, Lemma~\ref{gauss:kernel} gives
\begin{equation}\label{eq:radius-uniform-kernel-derivatives}
 |\nabla k_\lambda|\le2,\qquad
 \|D^2k_\lambda(x)\|\le2/R,\qquad
 \|D^3k_\lambda(x)\|
 \le|\phi_\lambda'(R)|+4\xi_\lambda(R)/R\le CR^{-2}.
\end{equation}
Here $\xi'=(\phi-\xi)/R$, $\phi'=-R\phi/\lambda^2$ and
$\|D(x/R)\|\le R^{-1}$; fixed-radius smoothness covers zero.
Since $\|\rho\|_1=m^2$, $\|j\|_1\le mb$,
\[
 \|q\|_\infty\le2m^2,\qquad \|q_m\|_\infty\le2mb,
\]
and translated Hardy gives
\[
 \sup_x\int\frac{|h(y)|^2}{|x-y|}\dd y
 \le m\sup_x\|h/|x-\cdot|\|_2\le2mb.
\]
Combined with \eqref{eq:radius-uniform-kernel-derivatives}, this proves the $Dq$ and $\operatorname{div}q$ bounds in \eqref{eq:full-density-weight-bounds}.
Differentiate the $L^1$ convolutions using their bounded
fixed-radius kernels. HLS then gives
\[
 \begin{aligned}
 \|D^2q\|_6&\le C\|I_1*\rho\|_6
                  \le C\|\rho\|_2=Ca^{1/2},\\
 \|\nabla q_m\|_{12}&\le C\|I_2*|j|\|_{12}
          \le C\|j\|_{4/3}
          \le C\|h\|_4\|\nabla h\|_2=Cba^{1/4}.
 \end{aligned}
\]
\end{proof}

For the input difference $F(U)-F(h)$, differentiating the nonlinear
error is unnecessary. The Hartree pressure identity transfers the
derivative to the smooth low-frequency field while retaining both
nonlocal pressures and the variation of the density weight.

\begin{lemma}[Mixed HLS and the full input source]\label{dyn:mixed-hls}
For $V\geq0$ in $L^4\cap L^6$,
\[
 H(V):=\int(I_\alpha*V^p)V^{p-1}
       \leq C_\alpha\|V\|_4^2\|V\|_6^{2p-3}.
\]
For $v,w\in L^4\cap L^6$, with $V=|v|+|w|$,
\[
 \|F(v)-F(w)\|_1\leq C_\alpha\bigl[
 \|v-w\|_6\|V\|_4^2\|V\|_6^{2p-4}
 +\|v-w\|_4\|V\|_4\|V\|_6^{2p-3}\bigr].
\]
For the actual $c=F(U)-F(h)$, $l=U-h$, the extension of
\eqref{eq:source-flat} is
\begin{equation}\label{eq:input-normalform}
 \begin{split}
 \mathcal S_\lambda(c,h)={}&-[\mathcal P_q(U)-\mathcal P_q(h)]
       +2\int\operatorname{div}q\,\re(F(U)\overline l)\\
 &+4\re\int F(U)[q\cdot\nabla\overline l-iq_m\overline l],
 \end{split}
\end{equation}
where the weights still come from the full $h$, and
\[
 \mathcal P_q(v)=\frac{2(p-2)}p\int\operatorname{div}q
                       (I_\alpha*|v|^p)|v|^p
  +\frac{2(3-\alpha)}p\iint I_\alpha(x-y)C_q(x,y)|v(x)|^p|v(y)|^p,
\]
$C_q(x,y)=(q(x)-q(y))\cdot(x-y)/|x-y|^2$ off the diagonal, and zero
on it. This extension is continuous in the actual norms and agrees with
\eqref{eq:source-flat} almost everywhere in time. Pointwise,
\begin{equation}\label{eq:input-source-estimate}
 |\mathcal S_\lambda(c,h)|
 \leq C_{\alpha,B}\bigl[mb\|l\|_\infty+m^2\|\nabla l\|_\infty\bigr]z.
\end{equation}
\end{lemma}
\begin{proof}
Put $q_\alpha=6/p$. Strict HLS and
$(3-\alpha)/6+2/4+\alpha/6=1$ give
\[
 H(V)\le C_\alpha\|V^p\|_{q_\alpha}\|V^{p-1}\|_{q_\alpha}
 \le C_\alpha\|V\|_6^p\|V\|_4^2\|V\|_6^\alpha.
\]
Expand the density and local-factor differences, using
\begin{equation}\label{eq:mixed-source-factor-differences}
 \begin{aligned}
 \|V^{p-1}(v-w)\|_{q_\alpha}
 &\le\|v-w\|_6\|V\|_6^{p-1},\\
 \|V^{p-2}(v-w)\|_{q_\alpha}
 &\le\|v-w\|_4\|V\|_4\|V\|_6^\alpha.
 \end{aligned}
\end{equation}
The power-difference inequalities and HLS applied to \eqref{eq:mixed-source-factor-differences} give the two terms in the asserted difference bound.

For compact smooth fields,
\[
 \im(c\overline h)=-\im(F(U)\overline l),\qquad
 \re(F(v)\nabla\overline v)=\frac1p(I_\alpha*|v|^p)\nabla|v|^p.
\]
Regularize $I_\alpha$ by
$I_{\alpha,\varepsilon}(x)=c_\alpha(|x|^2+\varepsilon^2)^{(\alpha-3)/2}$.
For $g=|v|^p$, symmetrization gives
\begin{equation}\label{eq:regularized-pressure-symmetrization}
 \int qg\cdot\nabla(I_{\alpha,\varepsilon}*g)
 =\frac12\iint(q(x)-q(y))\cdot\nabla I_{\alpha,\varepsilon}(x-y)
                         g(x)g(y).
\end{equation}
The integrable HLS majorant is
\begin{equation}\label{eq:pressure-riesz-majorant}
 (3-\alpha)\|Dq\|_\infty I_\alpha(x-y)g(x)g(y).
\end{equation}
Dominated convergence in \eqref{eq:regularized-pressure-symmetrization}, using \eqref{eq:pressure-riesz-majorant}, gives the chord coefficient $-(3-\alpha)/2$;
strict HLS gives convergence of the undifferentiated potentials.
Together with $4/p-2=-2(p-2)/p$, these yield
\eqref{eq:input-normalform}. The symmetrization precedes the limit.

For $U\in\X\cap L^4$, take common compact smooth approximants
in $\X\cap L^6\cap L^4$ and apply the same multipliers. Then
\begin{equation}\label{eq:source-normal-form-approximation}
 \begin{gathered}
 h_n\to h\text{ in }H^1\cap L^4,\qquad
 l_n\to l\text{ in }W^{1,\infty},\qquad
 F(U_n)\to F(U)\text{ in }L^1,\\
 (q_n,Dq_n,q_{m,n})\to(q,Dq,q_m)\text{ uniformly},\qquad
 |C_{q_n}-C_q|\le\|Dq_n-Dq\|_\infty.
 \end{gathered}
\end{equation}
HLS and \eqref{eq:source-normal-form-approximation} pass both pressures and every remaining term to a
unique continuous limit. At times with $\nabla U\in L^6$, also
approximate in $\dot W^{1,6}$: the extra cutoff error is
$CR^{-1}\|U\|_6$, and the homogeneous error uses the annular $L^6$
tail. Convergence of the forcing in $L^{6/5}$ identifies the limit
with \eqref{eq:source-flat}. Such times have full measure by
\eqref{local:fullz}.

Finally set $d=|U|^p-|h|^p$ and $V=|U|+|h|$.
The bound $|d|\le p|l|V^{p-1}$ and expansion of both density slots give
\begin{equation}\label{eq:pressure-density-difference}
 \int|(I_\alpha*|U|^p)|U|^p-(I_\alpha*|h|^p)|h|^p|
 \le2p\|l\|_\infty H(V),
\end{equation}
with the same bound for the pair-density difference under
$I_\alpha(x-y)\dd x\dd y$. By \eqref{eq:pressure-density-difference} and \eqref{eq:input-normalform},
\[
 \begin{aligned}
 |\mathcal S_\lambda(c,h)|\le C_\alpha H(V)
 \bigl[&(\|Dq\|_\infty+\|q_m\|_\infty)\|l\|_\infty\\
       &+\|q\|_\infty\|\nabla l\|_\infty\bigr].
 \end{aligned}
\]
Use \eqref{eq:full-density-weight-bounds} from Lemma~\ref{dyn:weights} and the $L^4,L^6$ multiplier bounds to
obtain \eqref{eq:input-source-estimate}.
\end{proof}

The pointwise bounds contain a forward maximal average of the
original forcing. Since its global $L^2$ time norm is still unknown,
the comparison must first be made on finite intervals selected from
that forcing alone, before choosing the frequency cutoff.

\begin{lemma}[A common temporal comparison]\label{dyn:temporal-comparison}
There is a strictly increasing sequence $T_j\to\infty$, determined by
$f$ alone, such that, with
\[
 Y(s)=\sup_{n\in\mathbb Z}2^{-n}\int_s^{s+2^n}f(t)\dd t,
\]
one has $\|Y\|_{L^2(0,T_j)}\leq C_M\|f\|_{L^2(0,T_j)}$.
For all sufficiently small $\nu$, uniformly in $j$,
\begin{equation}\label{eq:finite-comparison}
 \int_0^{T_j}(f^2+Y^2+z^2)\dd s
        \leq C_{\alpha,B}\int_0^{T_j}\|h_\nu(s)\|_4^4\dd s.
\end{equation}
Moreover the actual lowpass satisfies
$\|P_{\leq L}U(s)\|_\infty\leq CLY(s)$ and
$\|\nabla P_{\leq L}U(s)\|_\infty\leq CL^2Y(s)$.
\end{lemma}
\begin{proof}
Apply Lemma~\ref{std:maximal} to the bounded nonnegative function $f$.
Its strictly increasing sequence $T_j\to\infty$ depends only on $f$ and gives
\begin{equation}\label{eq:common-terminal-maximal-bound}
 \|Y\|_{L^2(0,T_j)}\le C_M\|f\|_{L^2(0,T_j)},\qquad C_M=\sqrt{388}.
\end{equation}
No global time-integrability has been used.

Proposition~\ref{ann:annular-bochner} gives
\begin{equation}\label{eq:annular-maximal-bound}
 \|P_LU(s)\|_\infty
 \le C\int_0^\infty\min(L^3,\tau^{-3/2})f(s+\tau)\dd\tau
 \le CLM^+f(s)\le CLY(s).
\end{equation}
For the middle inequality in \eqref{eq:annular-maximal-bound}, use the layer-cake representation of the
decreasing kernel, whose integral is $3L$; for the last, use
$Y\le M^+f\le2Y$.
Sum the telescoping dyadic annuli for the lowpass and its gradient.
The remainders vanish by the Bernstein bounds
$C\sqrt aB$ and $Ca^{3/2}B$ as $a\downarrow0$.

The annular seminorm of $l_\nu\in L^4$ is at most $CY$.
Lemma~\ref{std:refined} gives
\begin{equation}\label{eq:lowpass-quartic-comparison}
 \|l_\nu(s)\|_4^2\le C\eta_\nu Y(s),\qquad
 z(s)\le2\|h_\nu(s)\|_4^2+C\eta_\nu Y(s).
\end{equation}
All finite-window norms are finite. Using $f\le C_{\alpha,B}z$,
choose $\nu$ so that $C\eta_\nu C_MC_{\alpha,B}\le1/2$. Combining \eqref{eq:common-terminal-maximal-bound} and \eqref{eq:lowpass-quartic-comparison} gives
\[
 \|z\|_{L^2(0,T_j)}
 \le2\left(\int_0^{T_j}\|h_\nu\|_4^4\right)^{1/2}
       +\tfrac12\|z\|_{L^2(0,T_j)},
\]
and \eqref{eq:finite-comparison} follows uniformly in $j$.
\end{proof}

The output term $P_{\le\nu}F(U)$ presents a different difficulty:
low output frequency need not come from a small low-frequency input.
We split only the explicit field factor into near and far frequencies,
keep the full density and current weights, and then use the frequency
gap against the low-frequency parts of those weights.

\begin{theorem}[Complete source budget]\label{dyn:source-budget}
For every positive continuous radius $\lambda(s)$ and every sufficiently
small fixed $\nu>0$, the continuous source extension satisfies
\begin{equation}\label{eq:source-budget}
 \int_0^{T_j}|\mathcal S_{\lambda(s)}(e_\nu,h_\nu)|\dd s
 \leq q_\nu\int_0^{T_j}\|h_\nu(s)\|_4^4\dd s,
 \qquad q_\nu=C_{\alpha,B,\varphi}
 (\sqrt{\delta_\nu}+\delta_\nu+\delta_\nu^2)\longrightarrow0.
\end{equation}
Here $T_j$ is from Lemma~\ref{dyn:temporal-comparison}; the constant
is uniform in $j,\nu,\lambda$.
\end{theorem}
\begin{proof}
Suppress $\nu$ in $h,l,m$ and put $a=\|h\|_4^4$.
Lemma~\ref{dyn:temporal-comparison} and
\eqref{eq:input-source-estimate} give
\begin{equation}\label{eq:input-budget-pointwise}
 |\mathcal S_\lambda(c,h)|
 \le C_{\alpha,B}(\delta_\nu+\delta_\nu^2)Yz.
\end{equation}
For $o=-P_{\le\nu}F(U)$, Young's inequality gives compatible representatives:
\begin{equation}\label{eq:output-source-young-bounds}
 \begin{array}{lll}
 \|o\|_1\le Cf,&\|\nabla o\|_1\le C\nu f,\\
 \|o\|_{3/2}\le C\nu f,&\|\nabla o\|_{3/2}\le C\nu^2f,\\
 \|o\|_2\le C\nu^{3/2}f,&\|\nabla o\|_2\le C\nu^{5/2}f.
 \end{array}
\end{equation}
In particular, \eqref{eq:output-source-young-bounds} gives $o\in H^1$. Holding the full-$h$ weights fixed, integrate by
parts in the field slot:
\begin{equation}\label{eq:output-source-form}
 \mathcal S_\lambda(o,v;q,q_m)
 =4\re\int\overline vq\cdot\nabla o
  +2\re\int\overline vo\operatorname{div}q
  -4\im\int q_mo\overline v.
\end{equation}
Bounded-weight $L^2$ pairings and $H^1$ approximation justify this
without a spatial moment.

Choose $\Lambda\ge8$, put $\mu=\Lambda\nu/8$, and split
\[
 H=P_{>\Lambda\nu}U=P_{>\Lambda\nu}h,\qquad
 v=h-H=(P_{\le\Lambda\nu}-P_{\le\nu})U.
\]
The equality for $H$ follows from disjoint multiplier supports.
The lowpass and multiplier bounds give
\begin{equation}\label{eq:frequency-gap-field-bounds}
 \|v\|_\infty\le C\Lambda\nu Y,\qquad
 \|\nabla v\|_\infty\le C\Lambda^2\nu^2Y,\qquad
 \|H\|_6\le CB,\quad \|H\|_4\le Ca^{1/4}.
\end{equation}
Integrating the first term in \eqref{eq:output-source-form} back
onto $q\overline v$, and using \eqref{eq:frequency-gap-field-bounds} and \eqref{eq:full-density-weight-bounds}, yields
\begin{equation}\label{eq:medium-output-source-bound}
 |\mathcal S_\lambda(o,v;q,q_m)|
 \le C_B(\Lambda\delta_\nu+\Lambda^2\delta_\nu^2)fY.
\end{equation}

For the far-frequency field $H$, decompose the complete weights at
$\mu$. Their high-frequency parts gain inverse powers of $\Lambda\nu$
from derivatives; the low-frequency parts will cancel by support.
The identities $\int K=1$, $\int yK(y)\dd y=0$, Taylor expansion
to first or second order, Minkowski and
\eqref{eq:full-density-weight-derivatives} give
\begin{equation}\label{eq:high-frequency-weight-bounds}
 \begin{gathered}
 \|q^{\rm hi}\|_6\le C(\Lambda\nu)^{-2}a^{1/2},\qquad
 \|\operatorname{div}q^{\rm hi}\|_6
       \le C(\Lambda\nu)^{-1}a^{1/2},\\
 \|q_m^{\rm hi}\|_{12}\le CB(\Lambda\nu)^{-1}a^{1/4}.
 \end{gathered}
\end{equation}
No integrability of the bounded weights themselves is required.
The angular-frequency supports of $o,H,w^{\rm lo}$ lie in
$\overline B_{2\nu}$, outside $B_{\Lambda\nu}$, and in
$\overline B_{\Lambda\nu/4}$, respectively. Since
$2+\Lambda/4<\Lambda$, all low-weight pairings with $H$ vanish.
For bounded $w$, justify this by writing
\[
 w^{\rm lo}=\phi*w,\qquad o\text{ or }\partial_io=-\psi*G,
 \quad G=F(U),\quad \phi=K_\mu,\quad
 \psi=K_\nu\text{ or }\partial_iK_\nu.
\]
The absolute Fubini integral is bounded by
\[
 \|w\|_\infty\|\phi\|_1\|\psi\|_2\|G\|_1\|H\|_2<\infty.
\]
For each translation pair, the Schwartz product
$\phi(x-y)\psi(x-z)$ has Fourier support in
$\overline B_{2\mu+2\nu}$; its pairing with $H$ vanishes by Plancherel.
This includes the divergence weight and places no frequency condition on $G$.

For the high weights, apply H\"older to \eqref{eq:high-frequency-weight-bounds} and \eqref{eq:output-source-young-bounds}, with $(6,6,3/2)$ and $(4,12,3/2)$:
\begin{equation}\label{eq:high-output-pairing-bounds}
 \begin{aligned}
 \left|\int\overline Hq^{\rm hi}\cdot\nabla o\right|
 &\le C_B\Lambda^{-2}f\sqrt a,\\
 \left|\int\overline Ho\operatorname{div}q^{\rm hi}\right|
 +\left|\int q_m^{\rm hi}o\overline H\right|
 &\le C_B\Lambda^{-1}f\sqrt a.
 \end{aligned}
\end{equation}
Combining \eqref{eq:medium-output-source-bound} and \eqref{eq:high-output-pairing-bounds} gives
\begin{equation}\label{eq:output-budget-pointwise}
 |\mathcal S_\lambda(o,h)|\le C_B
 \bigl[(\Lambda\delta_\nu+\Lambda^2\delta_\nu^2)fY
       +\Lambda^{-1}f\sqrt a\bigr].
\end{equation}
The near-frequency term costs $\Lambda\delta_\nu$, whereas the
far-frequency term gains $\Lambda^{-1}$. Balancing them, for
$0<\delta_\nu\le1/64$ fix $\Lambda=\delta_\nu^{-1/2}$.
Apply Cauchy--Schwarz and \eqref{eq:finite-comparison} to \eqref{eq:input-budget-pointwise} and \eqref{eq:output-budget-pointwise}. All time integrals are bounded by $C\int_0^{T_j}a$, giving \eqref{eq:source-budget}.
If $\delta_\nu=0$, then $h=0$ and the source vanishes.

On compact time and positive-radius intervals,
$h\in CH^1$, $U\in C(\X\cap L^4)$ and $F(U)\in CL^1$ give
joint continuity of \eqref{eq:input-normalform} and
\eqref{eq:output-source-form}. Their sum agrees a.e. with
\eqref{eq:source-flat} by Lemma~\ref{dyn:mixed-hls}, and hence is
the source of \eqref{eq:forced-highpass}.
\end{proof}

The estimate concerns the complete source pairing. The following
example explains why an $L^1$ estimate for the defect alone would not
give the required small coefficient.

\begin{remark}\label{rem:source-l1-obstruction}
The source budget cannot be replaced by $\|e_\nu\|_1\to0$.
For $g\in L^1$, translation continuity of $K$ gives
\[
 \nu^{-3}(P_{\leq\nu}g)(y/\nu)
   =\int K(y-\nu x)g(x)\dd x\longrightarrow K(y)\int g
       \quad\text{in }L^1_y.
\]
Hence $\|P_{\leq\nu}g\|_1\to|\int g|\|K\|_1$.
For a nonzero nonnegative compact smooth $U$, $\int F(U)>0$, while
the mixed continuity bound gives $F(U)-F(P_{>\nu}U)\to0$ in $L^1$.
The resulting defect has a positive limiting $L^1$ norm. The pressure
cancellation and the linear gap in the weights are what make
\eqref{eq:source-budget} small.
\end{remark}

\subsection{The selected Gaussian radius and its evolution}

We continue with the same half-line $U$, cutoff and highpass family.
All Gaussian quantities below use the complete density of $h_\nu$.
For clarity in Fourier formulas only, write
$\widehat v_{\rm ang}(\xi)=\int e^{-ix\cdot\xi}v(x)\dd x$,
$c_F=(2\pi)^{-3}$ and $\gamma_0=(2\pi)^{3/2}$.
This is the change of variable from normalized frequency to angular
frequency, not a change in the physical free group or cutoff.

The radius is chosen so that the positive Gaussian density form
controls $\|h_\nu\|_4^4$. A density moment estimate first prevents
the selected radius from being too small relative to this quartic
quantity.

\begin{lemma}[A density moment estimate]\label{dyn:density-moment}
If $h\in H^1(\R^3)$ and $\rho=|h|^2$, then $\rho\in L^1\cap L^2$ and
\[
 c_F\int|\xi|\,|\widehat\rho_{\rm ang}(\xi)|^2\dd\xi
       \leq\frac{25}{\pi^2}\|\nabla h\|_2^4.
\]
\end{lemma}
\begin{proof}
Scaling and rotation give
\begin{equation}\label{eq:density-moment-convolution-bound}
 \int\frac{\dd\eta}{|\eta|^2|\xi-\eta|^2}
 \le\frac{200\pi}{|\xi|}\qquad(\xi\ne0).
\end{equation}
At $|\xi|=1$, split into $|\eta|\le1/2$,
$|\eta-\xi|\le1/2$, the remaining $|\eta|\le2$, and $|\eta|\ge2$.
Their bounds sum to
$8\pi+8\pi+512\pi/3+8\pi<200\pi$.
For Schwartz $h$, the product formula, weighted Cauchy--Schwarz, and \eqref{eq:density-moment-convolution-bound} give
\begin{equation}\label{eq:density-fourier-product}
 |\widehat\rho_{\rm ang}(\xi)|^2
 \le c_F^2\frac{200\pi}{|\xi|}
 \int|\eta|^2|\widehat h_{\rm ang}(\eta)|^2
       |\eta-\xi|^2|\widehat h_{\rm ang}(\eta-\xi)|^2\dd\eta.
\end{equation}
Integrate \eqref{eq:density-fourier-product} using Tonelli and Plancherel:
\[
 c_F\int|\xi||\widehat\rho_{\rm ang}(\xi)|^2\dd\xi
 \le200\pi c_F
 \left(c_F\int|\eta|^2|\widehat h_{\rm ang}(\eta)|^2\dd\eta\right)^2
 =\frac{25}{\pi^2}\|\nabla h\|_2^4.
\]
For $h\in H^1$, smooth approximation in $H^1\cap L^4$ gives
density convergence in $L^1\cap L^2$. Plancherel identifies the
Fourier realizations; an a.e. convergent subsequence and Fatou pass
the weighted bound to $h$.
\end{proof}

For nonzero $h\in H^1$, define
\[
 G_h(\lambda)=\iint e^{-|x-y|^2/(2\lambda^2)}\rho(x)\rho(y)\dd x\dd y,
 \qquad H_h(\lambda)=G_h(\lambda)/\lambda^4\quad(\lambda>0).
\]
These are finite integrals because $\rho\in L^1$.
\begin{lemma}[The Gaussian density root]\label{dyn:root}
For every $\theta>0$ there is a unique $\lambda>0$ with
$G_h(\lambda)=\theta\lambda^4$. If $a=\|h\|_4^4$, then
\[
 \lambda\leq\gamma_0a/\theta,\qquad Q_\lambda[\rho]\geq2\theta\lambda.
\]
If $c_F\int|\xi||\widehat\rho_{\rm ang}|^2\leq D/2$ and
$\theta\leq\gamma_0e^{-D^2/2}/2$, then also $a\leq\lambda$.
The root is continuous under $H^1$ convergence of nonzero fields.
\end{lemma}
\begin{proof}
Set $I_\lambda=\int e^{-\lambda^2|\xi|^2/2}
|\widehat\rho_{\rm ang}(\xi)|^2\dd\xi$.
Gaussian convolution, Plancherel and \eqref{gauss:Q-fourier} give
\begin{equation}\label{eq:gaussian-root-fourier-comparison}
 H_h(\lambda)=\frac{c_F\gamma_0}{\lambda}I_\lambda
 \le\frac{\gamma_0a}{\lambda},\qquad
 Q_\lambda[\rho]\ge2c_F\gamma_0 I_\lambda
 =\frac{2G_h(\lambda)}{\lambda^3}.
\end{equation}
The identities extend by $L^1\cap L^2$ density.
Dominated convergence makes $I_\lambda$ positive, continuous and
nonincreasing. Hence $H_h$ is continuous and strictly decreasing, with
\[
 H_h(\lambda)\ge\frac{c_F\gamma_0I_1}{\lambda}\to\infty
 \quad(\lambda\downarrow0),\qquad
 H_h(\lambda)\le\frac{\gamma_0a}{\lambda}\to0
 \quad(\lambda\to\infty).
\]
The intermediate value theorem gives existence and uniqueness; \eqref{eq:gaussian-root-fourier-comparison} gives the two root bounds.

Put $J=c_F\int|\xi||\widehat\rho_{\rm ang}|^2$.
Nonzero $\rho$ gives $0<J\le D/2$, since otherwise its Fourier
transform would vanish off $\{0\}$. For $t\ge0$,
\[
 e^{-t^2/2}\ge e^{-D^2/2}(1-t/D).
\]
Use monotonicity for $t\le D$ and the sign of the right side for
$t>D$. Integration at $t=a|\xi|$ gives
\begin{equation}\label{eq:gaussian-root-lower-test}
 c_FI_a\ge e^{-D^2/2}(a-aJ/D)\ge\tfrac12e^{-D^2/2}a,
 \qquad H_h(a)\ge\tfrac12\gamma_0e^{-D^2/2}\ge\theta.
\end{equation}
Equation \eqref{eq:gaussian-root-lower-test} and strict decrease yield $a\le\lambda$.

If $h_n\to h\ne0$ in $H^1$, then $\rho_n\to\rho$ in $L^1$ and
\begin{equation}\label{eq:gaussian-root-continuity}
 |G_{h_n}(\ell)-G_h(\ell)|
 \le(\|\rho_n\|_1+\|\rho\|_1)\|\rho_n-\rho\|_1\to0.
\end{equation}
By \eqref{eq:gaussian-root-continuity}, the strict inequalities at any two radii straddling the limiting root persist and trap the new roots. This proves continuity.
\end{proof}

The Gaussian level must satisfy two requirements simultaneously:
the radius must lie in the coercive range and remain comparable with
the quartic density. Compactness and the uniform gradient bound allow
one level to serve every sufficiently small cutoff.

\begin{proposition}[A common Gaussian level]\label{dyn:selector}
There exist $\theta>0$, $\nu_0>0$, $A_4<\infty$ and $c_{\rm core}>0$
such that, for $0<\nu\leq\nu_0$, the actual root
$\lambda_\nu(s)>0$ of $G_{h_\nu(s)}(\lambda)=\theta\lambda^4$
is continuous and satisfies
\begin{equation}\label{eq:selector-bounds}
 a_\nu(s)\leq\lambda_\nu(s)\leq\gamma_0a_\nu(s)/\theta
       \leq\gamma_0A_4/\theta,\qquad
 a_\nu(s)\geq c_{\rm core}r(s),\quad
 \lambda_\nu(s)\geq R_*r(s),
\end{equation}
where $a_\nu=\|h_\nu\|_4^4$ and $c_*,R_*$ are the fixed physical
coercivity constants of Corollary~\ref{gauss:affine}. In particular
\begin{equation}\label{eq:selected-principal}
 B_{\lambda_\nu(s)}[h_\nu(s)]
 \geq\frac{99}{50}\theta a_\nu(s)
             +c_*\frac{r(s)^2}{\lambda_\nu(s)}.
\end{equation}
 The endpoint current
\[
 M_\lambda[h]=2\int j_h\cdot(\nabla k_\lambda*|h|^2)
\]
is continuous along this selected path and obeys
$|M_{\lambda_\nu(s)}[h_\nu(s)]|\leq2C_Gm_\nu^3B$.
\end{proposition}
\begin{proof}
The compact closure $K$ of the normalized orbit has positive energy
and a strict kinetic gap. Thus $0\notin K$; continuity of local
$L^2$ norms and a finite cover give $A,b>0$ with
$\inf_{f\in K}\|f\|_{L^2(B_A)}\ge b$.
Fix, independently of $\nu$,
\[
 m=b^2/4,\qquad c_{\rm core}=m^2/|B_A|,\qquad
 \theta_{\rm core}=e^{-2A^2/R_*^2}m^2/R_*^4,
\]
\[
 D=1+2(25/\pi^2)B^4,\qquad
 \theta=\tfrac12\min\{\theta_{\rm core},\gamma_0e^{-D^2/2}/2\}>0.
\]
Proposition~\ref{ann:uniform-highpass} and
Corollary~\ref{gauss:highpass-neighborhood} give $\nu_0>0$ so that
all normalized highpasses
$g_{\nu,s}(z)=r(s)^{1/2}h_\nu(s,x(s)+r(s)z)$ lie in the fixed
coercivity neighborhood and satisfy
\begin{equation}\label{eq:normalized-highpass-core}
 \int_{B_A}|g_{\nu,s}|^2\ge m,\qquad
 \|g_{\nu,s}\|_4^4\ge c_{\rm core},\qquad
 H_{g_{\nu,s}}(R_*)\ge\theta_{\rm core}.
\end{equation}
Lemma~\ref{dyn:density-moment} and gradient contraction verify the
moment bound in Lemma~\ref{dyn:root}, giving
$a_\nu\le\lambda_\nu\le\gamma_0a_\nu/\theta$.
For $g(z)=r^{1/2}h(x+rz)$,
\begin{equation}\label{eq:gaussian-root-affine-scaling}
 G_g(\lambda/r)=r^{-4}G_h(\lambda),\qquad
 H_g(\lambda/r)=H_h(\lambda),\qquad
 \|g\|_4^4=r^{-1}\|h\|_4^4.
\end{equation}
Equations \eqref{eq:normalized-highpass-core} and \eqref{eq:gaussian-root-affine-scaling} give $H_g(\lambda_\nu/r)=\theta\le H_g(R_*)$ at the root; hence
\[
 \lambda_\nu\ge R_*r,\qquad
 a_\nu\ge c_{\rm core}r,\qquad a_\nu\le A_4,
\]
as in \eqref{eq:selector-bounds}, where the last bound follows from Theorem~\ref{ann:orbit-l4} and
$L^4$ multiplier boundedness. Corollary~\ref{gauss:affine} gives
\[
 B_{\lambda_\nu}[h_\nu]
 \ge\frac{99}{100}Q_{\lambda_\nu}[|h_\nu|^2]+c_*r^2/\lambda_\nu
 \ge\frac{99}{50}\theta a_\nu+c_*r^2/\lambda_\nu.
\]

Since $h_\nu\in CH^1$, root continuity gives a positive minimum
on each compact time interval. Also
\begin{equation}\label{eq:gaussian-action-endpoint-bound}
 |M_\lambda[h]|\le2C_G\|h\|_2^3\|\nabla h\|_2
 \le2C_Gm_\nu^3B.
\end{equation}
Under $H^1$ convergence, density and current converge in $L^1$.
Their pair integral therefore varies continuously with the field;
dominated convergence gives continuity in positive radius.
\end{proof}

For both evolution identities, fix $\nu,\sigma,h=P_{>\nu}U$ and
$J=[a,b]$, $a<b$. The forced equation gives
\[
 h\in C(J;H^1)\cap L^2(J;W^{1,6}),\qquad
 \mathcal F=P_{>\nu}F(U)=F(h)+e\in C(J;L^1\cap L^{6/5}).
\]
Take normalized nonnegative $\varphi\in C_c^\infty$ and set
$\varphi_\varepsilon(x)=\varepsilon^{-3}\varphi(x/\varepsilon)$,
$h_\varepsilon=\varphi_\varepsilon*h$,
$\mathcal F_\varepsilon=\varphi_\varepsilon*\mathcal F$.
Young's $L^1\to L^2$ estimate gives all source Sobolev domains.
Test the smoothed equation after subtracting its Bochner primitive,
also after spatial differentiation. Then
\begin{equation}\label{eq:common-mollified-evolution}
 h_\varepsilon\in C(J;H^3)\cap C^1((a,b);H^1),\qquad
 \mathcal F_\varepsilon\in C(J;H^1),\qquad
 h_\varepsilon'=i\sigma(\Delta h_\varepsilon+\mathcal F_\varepsilon).
\end{equation}
The derivative extends continuously to the endpoints and integrates
to $h_\varepsilon(t)-h_\varepsilon(a)$.
Strong approximation, uniform on compact continuous images, gives
\begin{equation}\label{eq:common-mollified-convergence}
 h_\varepsilon\to h\text{ in }C(J;H^1)\cap C(J;L^6)
                       \cap L^2(J;W^{1,6}),\qquad
 \mathcal F_\varepsilon\to\mathcal F\text{ in }C(J;L^{6/5}).
\end{equation}
By \eqref{eq:common-mollified-convergence}, densities, currents, gradient tensors and
$\mathcal F_\varepsilon\overline h_\varepsilon$ converge in $C(J;L^1)$.
For the differentiated source product,
\begin{equation}\label{eq:mollified-source-product-convergence}
 \begin{aligned}
 &\|\mathcal F_\varepsilon\nabla\overline h_\varepsilon
                  -\mathcal F\nabla\overline h\|_{L^1(J\times\R^3)}\\
 &\quad\le
 \|\mathcal F_\varepsilon-\mathcal F\|_{L^2_JL^{6/5}}
                   \|\nabla h_\varepsilon\|_{L^2_JL^6}
 +\|\mathcal F\|_{L^2_JL^{6/5}}
                   \|\nabla h_\varepsilon-\nabla h\|_{L^2_JL^6}\to0.
 \end{aligned}
\end{equation}
Here the entire source is convolved. For $a=b$, the integrated
identities are immediate.

We first prove the balance with radius fixed. The regularization
acts on the field and the entire forcing, so it does not require
differentiating the nonlinear defect. The resulting joint continuity in
time and radius then permits the original radius function to vary.

\begin{theorem}[The balance for fixed and variable radii]
\label{dyn:balance}
For the same forced highpass and the continuous source of
Theorem~\ref{dyn:source-budget}, every fixed $\lambda>0$ satisfies
\begin{equation}\label{eq:fixed-balance}
 M_\lambda[h(b)]-M_\lambda[h(a)]
 =\sigma\int_a^b[B_\lambda[h(s)]+\mathcal S_\lambda(e(s),h(s))]\dd s.
\end{equation}
For any independently given positive locally absolutely continuous
radius $\lambda(s)$,
\begin{equation}\label{eq:ac-balance}
 \begin{split}
 M_{\lambda(b)}[h(b)]-M_{\lambda(a)}[h(a)]
 =\int_a^b\bigl\{\sigma[B_{\lambda(s)}[h(s)]
       +\mathcal S_{\lambda(s)}(e(s),h(s))]
       +\lambda'(s)N(s,\lambda(s))\bigr\}\dd s,\\
 N(s,\ell)=2\int j_h(s)\cdot(\nabla\psi_\ell*|h(s)|^2),
 \qquad \psi_\ell(x)=e^{-|x|^2/(2\ell^2)}.
 \end{split}
\end{equation}

\end{theorem}
\begin{proof}
For $v\in H^1$, use $\rho,j,T,q$ from Lemma~\ref{gauss:kernel}.
The kernel bounds and HLS give
\[
 |B_\lambda[v]|\le C\lambda^{-3}\|v\|_2^4
 +C\lambda^{-1}\|v\|_2^2\|\nabla v\|_2^2
 +C_\alpha\lambda^{-1}\|v\|_2^2\|v\|_6^{2p}.
\]
Under $H^1$ convergence, quadratic fields converge in $L^1$,
power densities in $L^{6/p}$, and weights uniformly.
Bilinear HLS and $|C_{q_n}-C_q|\le\|Dq_n-Dq\|_\infty$ pass
both pressures. Locally uniform Gaussian derivative bounds give
joint continuity of $B,M$ in the field and positive radius.

The $L^2$ product bounds make $M_\lambda$ a real $C^1$ quartic
polynomial on $H^1$, with
\begin{equation}\label{eq:gaussian-action-differential}
 \begin{split}
 DM_\lambda(v)[w]={}&2\int[\nabla k_\lambda*
          2\re(\overline vw)]\cdot j_v\\
 &+2\int q_v\cdot\im(\overline w\nabla v+\overline v\nabla w).
 \end{split}
\end{equation}
For Schwartz $v,W$ and $w=i\sigma(\Delta v+W)$,
\begin{equation}\label{eq:forced-density-evolution}
 \dot\rho=-2\sigma\operatorname{div}j-2\sigma\im(W\overline v),
\end{equation}
\begin{equation}\label{eq:forced-current-evolution}
 \dot j_i=\sigma\left(\tfrac12\partial_i\Delta\rho
                    -2\sum_j\partial_jT_{ij}\right)
          -\sigma\re(W\partial_i\overline v-v\partial_i\overline W).
\end{equation}
Substitute \eqref{eq:forced-density-evolution}--\eqref{eq:forced-current-evolution} into \eqref{eq:gaussian-action-differential}. The linear current and density slots give, respectively,
\[
 Q_\lambda+4\iint\partial_{ij}k_\lambda(x-y)\rho(y)T_{ij}(x),
 \qquad -4\iint\partial_{ij}k_\lambda(x-y)j_j(y)j_i(x),
\]
with summation in $i,j$. The forcing slots give
$\mathcal S_\lambda^{\rm flat}(W,v)$.
Writing $\mathcal K_\lambda$ for the full Hessian stress without its factor four,
\begin{equation}\label{eq:smooth-action}
 DM_\lambda(v)[i\sigma\Delta v+i\sigma W]
 =\sigma[Q_\lambda[|v|^2]+4\mathcal K_\lambda[v]
                       +\mathcal S_\lambda^{\rm flat}(W,v)].
\end{equation}
Cutoffs in both variables give errors $O(R^{-j})$, $j\ge1$,
times bounded kernel derivatives and integrable Schwartz products.
Let $R\to\infty$; density extends \eqref{eq:smooth-action} to
$v\in H^3$, $W\in H^1$.

For $v\in H^1$ with $\nabla v\in L^6$, the symmetrized identity
and approximation in the proof of Lemma~\ref{dyn:mixed-hls} give
\begin{equation}\label{eq:hartree-pressure-cancellation}
 \mathcal S_\lambda^{\rm flat}(F(v),v)
 =-\frac{2(p-2)}pE_\lambda[v]-\frac{2(3-\alpha)}pJ_\lambda[v].
\end{equation}

Apply \eqref{eq:smooth-action} to \eqref{eq:common-mollified-evolution} and integrate. Equations \eqref{eq:common-mollified-convergence}--\eqref{eq:mollified-source-product-convergence}, together with uniform weight convergence, pass the flat source in $L^1(J)$;
$H^1$ continuity passes the principal terms and endpoints.
At a.e. time split $\mathcal F=F(h)+e$.
Equation \eqref{eq:hartree-pressure-cancellation} and the continuous source extension give
\eqref{eq:fixed-balance}, including $a=0$.

Finally, differentiation of the $L^1$ current-density pairing gives
\begin{equation}\label{eq:gaussian-action-radius-derivative}
 \partial_\lambda k_\lambda=\psi_\lambda-1,\qquad
 \partial_\lambda\nabla k_\lambda=\nabla\psi_\lambda,\qquad
 \partial_\lambda M_\lambda[h(s)]=N(s,\lambda).
\end{equation}
The partial derivatives in \eqref{eq:fixed-balance} and \eqref{eq:gaussian-action-radius-derivative} are jointly continuous on compact positive-radius
rectangles, with the right time derivative at zero. The chain rule for
an independently given positive absolutely continuous radius yields \eqref{eq:ac-balance};
the radius derivative carries no factor $\sigma$.
\end{proof}

The same Gaussian occurs in the density equation defining the
radius and in the radius derivative of the action. This matching lets
the two derivatives be combined by completing a square; only the source
in the charge equation contributes an unfavorable term.

\begin{theorem}[Evolution of the selected radius]\label{dyn:radius-evolution}
For the root $\lambda(s)=\lambda_\nu(s)$ of
Proposition~\ref{dyn:selector}, put
\[
 R(s)=-2\im(e(s)\overline{h(s)}),\qquad
 V(s,\ell)=\psi_\ell*|h(s)|^2,
\]
\[
 P(s,\ell)=\int R(s)V(s,\ell),\qquad
 D(s)=4\theta\lambda(s)^3-\partial_\ell G_{h(s)}(\lambda(s)).
\]
The root is $C^1$ on $(0,\infty)$, with continuous right derivative
at zero, and is absolutely continuous on every compact half-line interval. It satisfies
\begin{equation}\label{eq:root-derivative}
 D(s)\geq\theta\lambda(s)^3>0,\qquad
 \lambda'(s)D(s)=2\sigma[N(s,\lambda(s))+P(s,\lambda(s))].
\end{equation}
Consequently
\begin{equation}\label{eq:signed-radius}
 \sigma\lambda'(s)N(s,\lambda(s))
 \geq-\frac{\pi^{3/2}}{2\theta}a_\nu(s)\|R(s)\|_1^2.
\end{equation}
Here $\sup_s\|R(s)\|_1\leq\zeta_\nu:=2C_SB\epsilon_\nu\to0$.
\end{theorem}
\begin{proof}
On $J=[a,b]$, use the common convolution family
$(h_n,\mathcal F_n)$ and a compact positive radius interval.
The Gaussian pair kernel is a continuous quartic polynomial on $L^2$;
differentiation and integration by parts give
\begin{equation}\label{eq:mollified-density-time-derivative}
 \partial_sG_{h_n(s)}(\ell)
 =2\sigma\left[N_n(s,\ell)
       -2\int(\psi_\ell*|h_n|^2)
                   \im(\mathcal F_n\overline h_n)\right].
\end{equation}
The density, current and source products converge in $C(J;L^1)$.
Uniform kernel bounds pass $G_{h_n}$ and its derivative in \eqref{eq:mollified-density-time-derivative} uniformly;
the uniform-limit derivative theorem and
$\im(F(h)\overline h)=0$ yield
\begin{equation}\label{eq:gaussian-density-time-derivative}
 \partial_sG_{h(s)}(\ell)=2\sigma[N(s,\ell)+P(s,\ell)],
 \qquad R\in C(J;L^1),\qquad
 \|R(s)\|_1\le2\|e(s)\|_{6/5}\|h(s)\|_6.
\end{equation}
The expressions in \eqref{eq:gaussian-density-time-derivative} are jointly continuous, with the right derivative at zero.

Since $\rho\otimes\rho\in L^1$, differentiation in radius gives
\begin{equation}\label{eq:gaussian-density-radius-derivative}
 \partial_\ell G=3G/\ell
 -c_F\gamma_0\ell^4\int|\xi|^2e^{-\ell^2|\xi|^2/2}
                         |\widehat\rho_{\rm ang}|^2\dd\xi.
\end{equation}
The bounded Gaussian multiplier and $\widehat\rho_{\rm ang}\in L^2$
justify the continuous partial derivative in \eqref{eq:gaussian-density-radius-derivative}. At $G=\theta\lambda^4$,
\[
 D=4\theta\lambda^3-\partial_\ell G\ge\theta\lambda^3>0.
\]
Subtract the root identities at two times and apply the mean-value
theorem to the radius increment. Root continuity makes its coefficient
tend to $-D\ne0$; division by the time increment gives
\[
 \lambda'D=2\sigma(N+P),
\]
which is \eqref{eq:root-derivative}, using right increments at zero. Continuous partials and the positive
minimum radius on $J$ give a continuous bounded derivative, hence
local absolute continuity.

Complete the square, retaining the orientation sign:
\begin{equation}\label{eq:signed-radius-square}
 \sigma\lambda'N=\frac{2N(N+P)}D
 =\frac{2(N+P/2)^2}D-\frac{P^2}{2D}\ge-\frac{P^2}{2D}.
\end{equation}
Cauchy--Schwarz gives
\begin{equation}\label{eq:radius-charge-pairing}
 \|V\|_\infty^2\le\|\psi_\lambda\|_2^2\|\rho\|_2^2
 =\pi^{3/2}\lambda^3a_\nu,\qquad
 P^2\le\pi^{3/2}\lambda^3a_\nu\|R\|_1^2.
\end{equation}
Combine \eqref{eq:signed-radius-square}--\eqref{eq:radius-charge-pairing} with $D\ge\theta\lambda^3$ to obtain
\eqref{eq:signed-radius}. Finally, Sobolev, gradient contraction and
Lemma~\ref{dyn:forced} give
\[
 \|R\|_1\le2C_SB\epsilon_\nu=\zeta_\nu\longrightarrow0.
\]
\end{proof}

The source and radius errors are now ready for absorption at a fixed frequency.

\subsection{Absorption, mass recovery, and rigidity}

All error coefficients tend to zero with $\nu$, but the endpoint
action bound may grow with $m_\nu$. Choose one sufficiently small
frequency and keep it fixed while sending the terminal time to infinity.

\begin{theorem}[Integrability of the original forcing]\label{mass:quartic}
For one sufficiently small fixed $\nu>0$,
\[
 \int_0^\infty\|h_\nu(s)\|_4^4\dd s<\infty,
 \qquad U\in L^4([0,\infty);L^4),\qquad
 F(U)\in L^2([0,\infty);L^1).
\]

\end{theorem}
\begin{proof}
Fix the trajectory, static neighborhood, $c_*,R_*,\theta$, and the
$f$-dependent terminals $T_j$ of Lemma~\ref{dyn:temporal-comparison}.
Set $\beta_\nu=\pi^{3/2}\zeta_\nu^2/(2\theta)$,
$\kappa=99\theta/200$ and $\kappa_0=2\kappa$.
Since $\eta_\nu,\delta_\nu,\zeta_\nu,q_\nu\to0$, fix one $\nu>0$
meeting the static, selector and finite-comparison conditions and
$\beta_\nu\le\kappa_0$, $q_\nu\le\kappa$.
For $a=\|h_\nu\|_4^4$ and $\lambda=\lambda_\nu$,
\eqref{eq:selected-principal} and \eqref{eq:signed-radius} give
\[
 \mathcal H:=B_\lambda[h_\nu]+\sigma\lambda'N(s,\lambda)-\kappa_0a
 \ge c_*r^2/\lambda\ge0,\qquad
 \mathcal A:=\sigma M_\lambda[h_\nu].
\]
Theorem~\ref{dyn:radius-evolution} gives absolute continuity of the selected radius, and $\mathcal H$ is continuous.
Theorem~\ref{dyn:balance} yields
\begin{equation}\label{eq:selected-action-balance}
 \mathcal A(T)-\mathcal A(0)
 =\kappa_0\int_0^Ta+\int_0^T\mathcal H
                  +\int_0^T\mathcal S_{\lambda(s)}(e_\nu,h_\nu).
\end{equation}
All integrals are finite. Apply \eqref{eq:source-budget} and \eqref{eq:gaussian-action-endpoint-bound} to \eqref{eq:selected-action-balance}:
\begin{equation}\label{eq:quartic-budget}
 \kappa\int_0^{T_j}a+\int_0^{T_j}\mathcal H
 \le\mathcal A(T_j)-\mathcal A(0)
 \le4C_Gm_\nu^3B=:C_{\rm end,\nu}<\infty.
\end{equation}
The source budget \eqref{eq:source-budget} controls the error without the $L^1$ smallness excluded by Remark~\ref{rem:source-l1-obstruction}. Monotone convergence applied to \eqref{eq:quartic-budget} gives $a,\mathcal H\in L^1(0,\infty)$.
Only the continuous surplus is integrated, so $r$ need not be measurable.
Moreover \eqref{eq:finite-comparison} gives
\begin{equation}\label{eq:original-source-time-integrability}
 \int_0^{T_j}(f^2+Y^2+z^2)
 \le C_{\alpha,B}\int_0^{T_j}a
 \le C_{\alpha,B}C_{\rm end,\nu}/\kappa.
\end{equation}
Exhaustion in \eqref{eq:original-source-time-integrability} yields $f,z\in L^2(0,\infty)$.
Continuity of the actual $L^1$ source and $L^4$ value gives their
strong measurability and the asserted Bochner memberships.
The frequency $\nu$ has remained fixed.
\end{proof}

Theorem~\ref{mass:quartic} supplies a vanishing forcing tail for the
original critical trajectory. A new cutoff $L>0$, independent of the
fixed $\nu$, now separates the high-frequency part controlled by the
energy norm from the low-frequency part controlled by the dual endpoint
estimate. Write
\[
 A_F(t)=\|F(U)\|_{L^2((t,\infty);L^1)}\longrightarrow0.
\]
\begin{proposition}[Mass recovery and decay]\label{mass:recovery}
For every $t\geq0$, the original value $U(t)$ belongs to $L^2$, with
its original weak gradient, and for every $L>0$,
\begin{equation}\label{eq:mass-frequency}
 \|U(t)\|_2\leq B/L+C_{\rm low}L^{1/2}A_F(t).
\end{equation}
These realizations form one path in $C([0,\infty);L^2)$, are uniformly
bounded there, and satisfy $\|U(t)\|_2\to0$ as $t\to\infty$.
\end{proposition}
\begin{proof}
Fix $L>0$ independently of $\nu$ and put
$K_L(x)=L^3K(Lx)$, $G_L=K_L*F(U)$.
Young's inequality gives the same convolution in $CL^{6/5}\cap CL^2$:
\[
 \|G_L(s)\|_{6/5}\le C_KL^{1/2}f(s),\qquad
 \|G_L(s)\|_2\le\|K\|_2L^{3/2}f(s).
\]
For $0\le t<T$, define on physical time
\begin{equation}\label{eq:reflected-endpoint-source}
 Q(y)=\begin{cases}G_L(\sigma y),&t\le\sigma y\le T,\\
                    0,&\text{otherwise},
       \end{cases}
 \qquad
 \|Q\|_{L^2_yL^{6/5}_x}
 \le C_KL^{1/2}\|F(U)\|_{L^2((t,T);L^1)}.
\end{equation}
Theorem~\ref{std:endpoint}, applied to \eqref{eq:reflected-endpoint-source}, gives a unique tested $L^2$ realization.
Schwartz pairings and $y=\sigma s$ identify it with
\[
 Z_{L,t,T}=\int_{\R}S(\sigma t-y)Q(y)\dd y
          =\int_t^TS(\sigma(t-s))G_L(s)\dd s.
\]
The finite Bochner domain follows from $G_L\in CL^2$ and strong
continuity of $S$; the Jacobian is one, with support $[-T,-t]$
when $\sigma=-1$. With $C_{\rm low}=C_0C_K$,
\begin{equation}\label{eq:finite-lowpass-dual-bound}
 \|Z_{L,t,T}\|_2
 \le C_{\rm low}L^{1/2}\|F(U)\|_{L^2((t,T);L^1)}.
\end{equation}
Apply \eqref{eq:finite-lowpass-dual-bound} also on difference intervals $(T,T')$:
\begin{equation}\label{eq:lowpass-endpoint-limit}
 Z_{L,t,T}\to Z_{L,t}\text{ in }L^2,\qquad
 \|Z_{L,t}\|_2\le C_{\rm low}L^{1/2}A_F(t).
\end{equation}
Source-norm convergence identifies the infinite endpoint realization;
no infinite-time $L^2$ Bochner integral is required.

Apply the lowpass to Proposition~\ref{ann:no-waste} on its unchanged
absolute sequence $n+1$. Its transpose preserves Schwartz space
and commutes with $S$, so
\begin{equation}\label{eq:original-field-mass-realization}
 P_{\le L}U(t)=-i\sigma Z_{L,t}\quad\text{in }\mathcal S',\qquad
 \|P_{>L}U(t)\|_2\le B/L.
\end{equation}
The lowpass limit in \eqref{eq:lowpass-endpoint-limit} and the highpass in \eqref{eq:original-field-mass-realization} sum to the distribution of
the original $L^6$ value. Distributional uniqueness on a countable
ball exhaustion identifies it a.e. with $U(t)$ and preserves its
weak gradient. Thus \eqref{eq:mass-frequency} holds for the same
scalar realization at every $L$.

Taking $L=1$ gives
$\sup_t\|U(t)\|_2\le M_0:=B+C_{\rm low}A_F(0)$.
Furthermore,
\begin{equation}\label{eq:highpass-path-continuity}
 \|P_{>L}(U(t)-U(s))\|_2
 \le L^{-1}\|\nabla(U(t)-U(s))\|_2,
\end{equation}
\begin{equation}\label{eq:lowpass-uniform-vanishing}
 \sup_{t\ge0}\|U(t)-P_{>L}U(t)\|_2
 \le C_{\rm low}L^{1/2}A_F(0)\to0\qquad(L\downarrow0).
\end{equation}
Equations \eqref{eq:highpass-path-continuity}--\eqref{eq:lowpass-uniform-vanishing} express $U$ as a uniform limit of continuous $L^2$ paths, including
at zero. For every fixed $L$, \eqref{eq:mass-frequency} gives
\[
 0\le\limsup_{t\to\infty}\|U(t)\|_2\le B/L.
\]
Let $L\to\infty$ to obtain decay.
\end{proof}

Mass conservation is used only after the original field has been
identified as a continuous $L^2$ path. Its conserved mass will contradict
the decay just obtained.

\begin{theorem}[No critical trajectory]\label{mass:rigidity}
The nonzero bounded-scale critical half-line of
Theorem~\ref{crit:bounded-halfline} cannot exist.
\end{theorem}
\begin{proof}
Proposition~\ref{mass:recovery} gives the original scalar path
$U\in CL^2$, with $\sup_t\|U(t)\|_2\le M_0$ and its original
weak gradient. Choose real $\zeta\in C_c^\infty$, $0\le\zeta\le1$,
equal to one near zero, and put
$\chi=\zeta^2$, $\chi_R(x)=\chi(x/R)$ and
$m_R(t)=\int\chi_R|U(t)|^2$.
For fixed $0\le s\le t<\infty$, Lemma~\ref{local:charge} gives
\[
 m_R(t)-m_R(s)=2\sigma\int_s^t\int\nabla\chi_R\cdot
                    \im(\overline U\nabla U)\dd x\dd\tau,
\]
\begin{equation}\label{eq:recovered-mass-flux}
 |m_R(t)-m_R(s)|
 \le\frac{2\|\nabla\chi\|_\infty}{R}(t-s)M_0B\longrightarrow0.
\end{equation}
Endpoint dominated convergence in \eqref{eq:recovered-mass-flux} yields
$\|U(t)\|_2^2=\|U(s)\|_2^2$.
Mass decay forces this value to be zero. Thus $U(0)=0$ a.e. and
its weak derivatives vanish, contradicting
$\|\nabla U(0)\|_2^2\ge2E_c>0$.
Both signs are covered, with no center estimate or positive lower
scale bound.
\end{proof}

The exclusion of the critical orbit completes the proof of the main theorem.

\begin{proof}[Proof of Theorem~\ref{thm:main}]
Under the hypotheses of Theorem~\ref{thm:main}, let
$u=u_{u_0}$ be the canonical solution through $u_0$ at zero.
Theorem~\ref{local:flow} and Lemma~\ref{local:mild} give, for every
competing solution $v$ through this datum,
\[
 I_v\subset I(u_0),\qquad v=u|_{I_v},
\]
with equality of intervals when $v$ is maximal.
Theorems~\ref{std:sharp}, \ref{local:energy} and
Lemma~\ref{local:trapping} give
\[
 \kappa=\frac{2p}{p-1},\qquad \kappa E(W)=G=A(W),\qquad
 E(u(t))=E(u_0),\quad A(u(t))\le\kappa E(u_0)<G.
\]
Let $\mathcal K$ be the stable set of Section~\ref{crit:section}.
If global existence or two-sided scattering fails,
Theorem~\ref{local:finite-s} and Corollary~\ref{local:small-data} yield
\[
 \mathcal B=\{E(f):f\in\mathcal K,\ \SSS_{I(f)}(u_f)=\infty\}
 \ne\varnothing,\qquad E_c=\inf\mathcal B\in(0,E(W)),
\]
\begin{equation}\label{eq:subcritical-scattering-levels}
 f\in\mathcal K,\ E(f)<E_c
 \quad\Longrightarrow\quad I(f)=\R,\quad\SSS_\R(u_f)<\infty.
\end{equation}
Only the individual bounds in \eqref{eq:subcritical-scattering-levels} enter reconstruction.
Propositions~\ref{crit:element}, \ref{crit:global} and
Theorem~\ref{crit:bounded-halfline} produce one canonical trajectory
$w$, one sign $\sigma\in\{-1,1\}$ and pointwise normalizations
$r(t)>0$, $x(t)\in\R^3$ such that
\[
 I_w=\R,\qquad E(w(t))=E_c,\qquad A(w(t))\le\kappa E_c<G,
 \qquad\SSS_{(-\infty,t]}(w)=\SSS_{[t,\infty)}(w)=\infty,
\]
\[
 Q=\overline{\{T_{r(t),x(t)}^{-1}w(t):t\in\R\}}^{\X}
 \text{ is compact},\qquad r(\sigma s)\le1\quad(s\ge0).
\]
Put $U(s)=w(\sigma s)$, $r_\sigma(s)=r(\sigma s)$ and
$x_\sigma(s)=x(\sigma s)$. Then
\[
 i\partial_sU+\sigma\Delta U+\sigma F(U)=0,\qquad
 T_{r_\sigma(s),x_\sigma(s)}^{-1}U(s)\in Q,
\]
\[
 0<r_\sigma(s)\le1,\qquad E(U(s))=E_c>0,\qquad
 A(U(s))\le\kappa E_c<G.
\]
Theorem~\ref{mass:rigidity} gives a contradiction.
Hence $I(u_0)=\R$ and Theorem~\ref{local:finite-s} supplies
$u_\pm\in\X$ with
$\|\nabla(u(t)-S(t)u_\pm)\|_2\to0$ as $t\to\pm\infty$.
\end{proof}

The scattering states belong to the original homogeneous energy class.

\section*{Acknowledgments}
A companion Lean development associated with this work is maintained in the
authors' private GitHub repository
\url{https://github.com/penghongzhong/open-HE}. The repository contains the
Lean source files used during verification, the pinned toolchain (Lean
4.35.0-rc3 and Mathlib revision \texttt{a98628e16c11f5167f16124105ddce53efa9bfe5}), dependency locks,
compilation checks, and theorem/dependency audit material. At the time of
this arXiv submission the repository is private; access can be provided for
verification purposes upon request. The formal development is intended as a
companion verification layer for the analytic argument and does not replace
the proofs presented in this paper.

The authors used GPT (OpenAI) to help clarify details of the derivations
and to assist with Lean verification. The authors reviewed the material
and take responsibility for the mathematical arguments and the final text.

\end{document}